\documentclass[12pt]{article}
\usepackage{amsfonts}
\usepackage{dsfont}
\usepackage{amsmath}
\usepackage{mathrsfs}
\usepackage{amssymb}
\usepackage{amsmath,color}
\usepackage{amsthm}
\usepackage{amstext}
\usepackage{amsopn}
\usepackage{texdraw}
\usepackage{graphicx}
\usepackage{multirow}
\usepackage{lscape}
\usepackage{epstopdf}
\usepackage{float}

\usepackage{subfigure}
\newtheorem{theorem}{Theorem}[section]
\newtheorem{lemma}[theorem]{Lemma}
\newtheorem{corollary}[theorem]{Corollary}
\newtheorem{proposition}[theorem]{Proposition}

\theoremstyle{remark}

\numberwithin{equation}{section}

\begin{document}
\date{}
\title{\large \textbf{Dynamics of a Reaction-Diffusion Benthic-Drift Model with Strong Allee Effect Growth}\footnote{Partially supported by  US-NSF grants  DMS-1715651 and DMS-1853598.} }
\author
{ Yan Wang\textsuperscript{a}\ \ Junping Shi\textsuperscript{b}\footnote{Corresponding author. Email: \texttt{jxshix@wm.edu}} \\
\\
{\small \textsuperscript{a} School of Mathematical Sciences, Beijing Normal University, \hfill{\ }}\\
\ \ {\small Beijing,  100875, P.R. China\hfill{\ }}\\
{\small \textsuperscript{b} Department of Mathematics, College of William and Mary, \hfill{\ }}\\
\ \ {\small Williamsburg, Virginia 23187-8795, USA \hfill {\ }}\\
}
\maketitle


\newcommand{\Om}{\Omega}
\newcommand{\ep}{\varepsilon}
\newcommand{\la}{\lambda}
\newcommand{\R}{{\mathbb R}}
\newcommand{\al}{\alpha}
\newcommand{\io}{\int_{\Om}}
\newcommand{\ds}{\displaystyle}
\newcommand{\noi}{\noindent}
\newcommand{\N}{\mathbb{N}}

\begin{abstract}
\noindent{The dynamics of a reaction-diffusion-advection benthic-drift population model that links changes in the flow regime and habitat availability with population dynamics is studied. In the model, the stream is divided into drift zone and benthic zone, and the population is divided into two interacting compartments, individuals residing in the benthic zone and individuals dispersing in the drift zone. The benthic population growth is assumed to be of strong Allee effect type. The influence of flow speed and individual transfer rates between zones on the population persistence and extinction is considered, and the criteria of population persistence or extinction are formulated and proved.
}
\end{abstract}
\vspace{0.1in}

\noindent{\textbf{Keywords}:  Reaction-diffusion-advection; benthic-drift; strong Allee effect; \\ persistence;  extinction}

\vspace{0.1in}


\section{Introduction}
Streams and rivers are characterized by a variety of physical, chemical and geomorphological features such as unidirectional flow, pools and riffles, bends and waterfalls, floodplains, lateral inflow and network structure and many more. These complex structures provide a wide range of qualitatively different habitat for aquatic species and organisms such as zooplankton, invertebrates, aquatic plant and fish. In \cite{muller1954investigations}, M{\"u}ller proposed an important issue in stream ecology, the ``drift paradox", which asks how stream dwelling organisms can persist in a river/stream environment when continuously subjected to a unidirectional water flow. Mathematical models, such as reaction-diffusion-advection equations and integro-differential equations have been established to study the population dynamic in advective environment. For species following  logistic type growth, a ``critical flow speed" has been identified, below which can ensure the persistence of the stream population \cite{Jin2011,LLL2015JBD,LLL2016SIAM,LL2014JMB,Lutscher2006,Lutscher2005,Mckenzie2012,speirs2001}. On the other hand, when the species following Allee effect type growth, population persistence for all initial conditions becomes not possible as the extinction state is always a stable state, and more delicate conditions are needed to ensure the population persistence \cite{SS2006JMB,ws2019,ws2018}. The solution of stream population persistence/extinction  not only leads to a better understanding of population dynamics in a stream environment, but also provides strategies for how to keep a native species persistent.

Stream hydraulic characteristics is another important factor in the ecology of stream populations. Of great importance is the presence of storage zones (zones of zero or near-zero flow) in stream channels. These zones are refuges for many organisms not adapted to high water velocity. And for some aquatic species, the individuals spend a proportion of their time immobile and a proportion of their time in an environment with a unidirectional current and do not reproduce there. Following \cite{bencala1983simulation,deangelis1995modelling}, the river can be partitioned into two zones, drift zone and benthic zone, and the population is also split into two interacting compartments: individuals residing in the benthic zone and the ones dispersing in the drift zone. Assuming that longitudinal movement occurs only in the drift zone,  a system of coupled reaction-diffusion-advection equation of drift population and equation of benthic population can be used to model the dynamic evolution of  aquatic species that reproduce on the bottom of the river and release their larval stages into the water column, such as sedentary water plant, oyster and coral  \cite{Lutscher2006,Pachepsky200561}.

Assuming logistic growth for the benthic population, the population spreading, invasion and the propagation speed were studied in \cite{Lutscher2006,Pachepsky200561}; the population persistence criteria on a finite length river based on the net reproductive rate was investigated in \cite{huang2016};  and the population dynamics of two competitive species in the river was studied in \cite{jin2018dynamics}. All these work assume logistic growth for the benthic population so the population persistence/extinction or spreading can be completely determined by a sharp threshold which is often expressed by a basic reproduction number or a critical advection rate. Benthic-drift models of algae and nutrient population have also been considered \cite{Grover2009,hsu2011,hsu2013,Wang2015}. Other studies also consider the effect of river network structure \cite{Jin2019,Ramirez2012,Sarhad2014,vasilyeva2019population}, the effect of advection on competition \cite{lou2019global,zz2016,z2016cvpde,zhoup2018,zz2018} and meandering structure \cite{Jin2017}.

In this paper, we investigate how interactions between the benthic zone and the drift zone affect the population dynamics of a benthic-drift model, when the species follows a strong Allee effect population growth in the benthic zone. Our main findings on the dynamics of benthic-drift model with strong Allee effect type growth in the benthic population are
\begin{enumerate}
  \item If the benthic population release rate is large, then for all the boundary conditions, extinction will always occur regardless of the initial conditions, the diffusive and advective movement and the transfer rate from the drift zone to the benthic zone;
  \item If the benthic population release rate is small (but not zero),  then for all the boundary conditions, the population persists for large initial conditions and becomes extinct for small initial conditions. Such a bistability in the system exists also independent of the diffusive and advective movement and the transfer rate from the drift zone to the benthic zone;
  \item If the benthic population release rate is in the intermediate range, the persistence or extinction depends on the diffusive and advective movement. It is shown that for the closed environment, the population can persist under small advection rate and large initial condition.
\end{enumerate}
These results are rigorously proved by using the theory of dynamical systems, partial differential equations, upper-lower solution methods, and numerical simulations are also included to verify or demonstrate theoretical results. Compared with the single compartment reaction-diffusion-advection equation with a strong Allee effect growth rate \cite{ws2018}, in which the advection rate $q$ plays an important role in the persistence/extinction dynamics, the benthic-drift model dynamics with strong Allee effect relies more critically on the strength of interacting between zones.

The dynamic behavior of the single compartment reaction-diffusion-advection equation modeling a stream population with a strong Allee effect growth rate was investigated in \cite{ws2018}. Compared to the well-studied logistic growth rate, the extinction state in the strong Allee effect case is always locally stable. It is shown that when both the diffusion coefficient and the advection rate are small, there exist multiple positive steady state solutions hence the dynamics is bistable so that different initial conditions lead to different asymptotic behavior. On the other hand, when the advection rate is large, the population becomes extinct regardless of initial condition under most boundary conditions. Corresponding dynamic behavior for weak Allee effect growth rate has been considered in \cite{ws2019}; and the role of protection zone on species persistence or spreading for species  with strong Allee effect growth has \cite{Cui2014JDE,du2019JDE}.

The benthic-drift model has the feature of a coupled partial differential equation (PDE) for the drift population and an ``ordinary differential equation" (ODE) for the benthic population. Note that the benthic population equation is not really one ODE but an ODE at each point of the spatial domain, or a reaction-diffusion equation with zero diffusion coefficient. Such degeneracy causes a noncompactness of the solution orbits in the function space, which brings an extra difficulty in analyzing the dynamics. Such PDE-ODE coupled systems have been also studied in the case of population that has a quiescent phase \cite{Zhang2007}, or some species are immobile \cite{marciniak2017,Wang2015b}. 

In Section 2, the benthic-drift model of stream population is established, and all model parameters and growth rate conditions are set up in a general setting. Some preliminary results are stated and proved in Section 3: the basic dynamics, global attractor, and linear stability problem. The main results on the population persistence and extinction are proved in Section 4, and some numerical simulations are shown in Section 5 to provide some more quantitative information of the dynamics. A few concluding remarks are in Section 6.

\section{Model}
Consider a population in which individuals live and reproduce in the storage zone, and occasionally enter the water column to drift until they settle on the benthos again. We assume that advective and diffusive transport occur only in the main flowing zone, not the storage zone. So we neglect the movement in the benthic zone. While in the drifting water, we consider the individual's movement as a combination of passive diffusion movement and advective movement which is from sensing and following the gradient of resource distribution (taxis) or a directional fluid/wind flow. Let $u(x,t)$ be the population density in the drift zone and let $v(x,t)$ be the population density in the benthic zone. And the river environment is modeled by a one-dimensional interval $[0,L]\subset \R$; the upstream endpoint is $x=0$, and the downstream endpoint is $x=L$, where $L$ is the length of the river. A mathematical model that describes the dynamics of the population in a river is given by \cite{huang2016,Lutscher2006}:
\begin{equation}\label{1}
  \begin{cases}
  u_t=d u_{xx}-q u_x+\ds\frac{A_b(x)}{A_d(x)}\mu v-\sigma u-m_1 u, & \qquad 0<x<L, \; t>0,\\
  v_t=vg(x,v)-m_2 v-\mu v+\ds\frac{A_d(x)}{A_b(x)}\sigma u, & \qquad 0\leq x\leq L, \; t>0,\\
  du_{x}(0,t)-qu(0,t)=b_u qu(0,t), & \qquad t>0,\\
  du_{x}(L,t)-qu(L,t)=-b_d qu(L,t), & \qquad t>0,\\
  u(x,0)=u_0(x)\ge 0, \; v(x,0)=v_0(x)\ge 0, & \qquad x\in (0,L),
  \end{cases}
\end{equation}
 where $d$ and $q$ are the diffusion rate and advection rate of the population in the drifting zone, respectively; $A_b(x)$ and $A_d(x)$ are the cross-sectional areas of the benthic zone and drift zone, respectively; $\sigma$ is the the transfer rate of the drift population to the benthic one and $\mu$ is the transfer rate of the benthic population to the drifting one; $m_1$ and $m_2$ are the mortality rates of the drift  and benthic population, respectively. Throughout the paper, we assume that the functions $A_b(x)$ and $A_d(x)$ and parameters satisfy  the following conditions:
 \begin{enumerate}
  \item[(A1)] $A_b(x), A_d(x)\in C[0,L]$, $A_b(x)>0$ and $A_d(x)>0$ on $x\in[0,L]$.
  \item[(A2)] $d>0$, $q\geq 0$, $\mu>0$, $\sigma>0$, $m_1\geq 0$ and $m_2\geq 0$.
\end{enumerate}
The boundary conditions for the drift population in \eqref{1} is given in a flux form following \cite{LL2014JMB,ws2018} (see also \cite{huang2016} for slightly different setting). Here the parameters $b_u\geq 0$ and $b_d\geq 0$ determine the magnitude of population loss at the upstream end $x=0$ and the downstream end $x=L$, respectively. At the boundary ends $x=0$ and $x=L$, if $b_u=0$ and $b_d=0$, that is the no-flux (NF) boundary condition $du_x(x,t)-q u(x,t)=0$, for instance, can be effectively used to study the sinking, self-shading phytoplankton model (see, e.g., \cite{Hsu2010, huisman2002sinking}); $b_d=1$ gives the free-flow (FF) boundary condition $u_x(x,t)=0$, referred as the Danckwerts condition, can be applied to the situation like stream to lake (see \cite{Vasilyeva2010}); and when $b_d$ becomes sufficiently large, i.e. $b_d\rightarrow\infty$, we have the hostile (H) boundary condition $u(x,t)=0$, which can be used in the scenario of stream to ocean (see \cite{speirs2001}).

The growth rate per capita $g(x,v)$ satisfy the following general conditions as in \cite{ws2018} (see also \cite{CC2003,SS2006JMB}):
 \begin{enumerate}
  \item[(g1)] For any $v\geq0$, $g(\cdot,v)\in C[0,L]$, and for any $x\in [0,L]$, $g(x,\cdot)\in C^1[0,L]$.
  \item[(g2)] For any $x\in[0,L]$, there exists $r(x)\geq0$, where $0<r(x)<M$ and $M>0$ is a constant, such that $g(x,r(x))=0$, and $g(x,v)<0$ for $v>r(x)$.
  \item[(g3)] For any $x\in[0,L]$, there exists $s(x)\in [0,r(x)]$ such that $g(x,\cdot)$ is increasing in $[0,s(x)]$ and non-increasing in $[s(x),\infty]$; and there also exists $N>0$ such that $g(x,s(x))\leq N$.
\end{enumerate}
Here $r(x)$ is the local carrying capacity at $x$ which has a uniform upper bound $M$; $v=s(x)$ is where $g(x,\cdot)$ reaches the maximum value, and the number $N$ is a uniform bound for $g(x,v)$ at all $(x,v)$.  Moreover we assume that $g(x,v)$ takes one of the following three forms: (see \cite{SS2006JMB,ws2018})
\begin{enumerate}
  \item[(g4a)] Logistic: $s(x)=0$, $g(x,0)>0$, and $g(x,\cdot)$ is decreasing in $[0,\infty)$;
  \item[(g4b)] Weak Allee effect: $s(x)>0$, $g(x,0)>0$, $g(x,\cdot)$ is increasing in $[0,s(x)]$, and is non-increasing in $[s(x),\infty)$; or
  \item[(g4c)] Strong Allee effect: $s(x)>0$, $g(x,0)<0$, $g(x,s(x))>0$, $g(x,\cdot)$ is increasing in $[0,s(x)]$, and is non-increasing in $[s(x),\infty)$. Hence there exists a unique $h(x)\in (0,s(x))$ such that $g(x,h(x))=0$ for all $0<x<L$.
\end{enumerate}

\begin{figure}[htb]
\centering
  \includegraphics[width=0.5\textwidth]{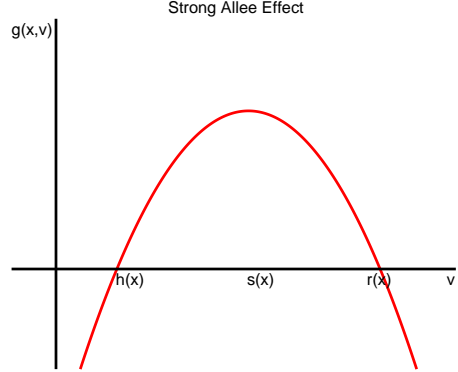}
  \caption{\label{fig1} {\small Growth rate per capita $g(x,v)$ of Strong Allee effect.}}
\end{figure}

For later applications, we also define
\begin{equation}\label{gmax}
    \begin{split}
        g_{max}&=\ds\max_{x\in[0,L]}g(x,s(x))=\ds\max_{x\in[0,L]}\max_{v\geq 0} g(x,v),\\
        g_{min}&=\ds\min_{x\in[0,L]}g(x,s(x))=\ds\min_{x\in[0,L]}\max_{v\geq 0} g(x,v).
    \end{split}
\end{equation}
The growth rate of the population is $f(x,v)=v g(x,v)$, and we also define
\begin{equation}\label{fv}
    \overline{f_v}=\ds\max_{x\in[0,L]}\max_{v\geq 0}f_v(x,v).
\end{equation}
One can observe that $g_{max}\leq \overline{f_v}$ as $\max_{v\geq 0} f_v(x,v)=\max_{v\geq 0}(g+vg_v)\geq g(x,s(x))+s(x)g_v(x,s(x))=g(x,s(x))=\max_{v\geq 0}g(x,v)$ for $x\in [0,L]$.

In the following we will study the dynamics of system \eqref{1} under the conditions (A1)-(A2), (g1)-(g3) and (g4c) (strong Allee effect growth). In particular, we are interested in the existence, multiplictity and stability of non-negative steady state solutions $(u(x),v(x))$ which satisfy the following steady state system:
\begin{equation}\label{1ss}
  \begin{cases}
  d u_{xx}-q u_x+\ds\frac{A_b(x)}{A_d(x)}\mu v-\sigma u-m_1 u=0, & \qquad 0<x<L,\\
  vg(x,v)-m_2 v-\mu v+\ds\frac{A_d(x)}{A_b(x)}\sigma u=0, & \qquad 0\leq x\leq L,\\
  du_{x}(0)-qu(0)=b_u qu(0),\\
  du_{x}(L)-qu(L)=-b_d qu(L).
  \end{cases}
\end{equation}

\section{Basic properties of solutions}
This section is devoted to establishing some basic properties of \eqref{1}.
\subsection{The well-posedness}
We first study the well-posedness of the initial-boundary-value problem \eqref{1}. Using the transform $u=e^{\alpha x}w, v=e^{\alpha x}z$ on the system \eqref{1}, where $\alpha=\ds\frac{q}{d}$, we obtain the following system of new variables $(w,z)$:
\begin{equation}\label{wz}
\begin{cases}
  w_t=d w_{xx}+q w_x+\ds\frac{A_b(x)}{A_d(x)}\mu z-\sigma w-m_1w, & \qquad 0<x<L, \; t>0,\\
  z_t=zg(x,e^{\alpha x}z)-m_2 z-\mu z+\ds\frac{A_d(x)}{A_b(x)}\sigma w, & \qquad 0\leq x\leq L, \; t>0,\\
  -dw_{x}(0,t)+b_u qw(0,t)=0, & \qquad t>0,\\
  dw_{x}(L,t)+b_d qw(L,t)=0, & \qquad t>0,\\
  w(x,0)=e^{-\alpha x}u_0(x):=w_0(x)\ge 0, & \qquad x\in (0,L),\\
  z(x,0)=e^{-\alpha x}v_0(x):=z_0(x)\ge 0, & \qquad x\in (0,L),
\end{cases}
\end{equation}
The boundary conditions of system \eqref{wz} are either no-flux ($b_u=b_d=0$), hostile ($b_u, b_d\rightarrow\infty$ ) or Robin ($b_u, b_d>0$) types. With $b_u\geq 0$ and $b_d\geq0$,  we have the following settings following similar ones in \cite{hsu2013,huang2016}. Let $X=C([0,L],\mathds{R})$ be the Banach space with the usual supremum norm $\| u\|_{\infty}=\ds\max_{x\in[0,L]}|u(x)|$ for $u\in X$. Then the set of non-negative functions forms a solid cone $X_+$ in the Banach space $X$. 
Suppose that $T_1(t)$ is the $C_0$ semi-group associated with the following linear initial value problem
\begin{equation}\label{po}
 \begin{cases}
 w_t=d w_{xx}+qw_x-m_1w, & \qquad 0<x<L, \; t>0,\\
 -dw_{x}(0,t)+b_u qw(0,t)=0, & \qquad t>0,\\
  dw_{x}(L,t)+b_d qw(L,t)=0, & \qquad t>0,\\
  w(x,0)=w_0(x)\ge 0, & \qquad x\in (0,L).
 \end{cases}
\end{equation}
From \cite[Chapter 7]{smith2008}, it follows that the solution of \eqref{po} is given by $w(x,t)=T_1(t)w_0$ and $T_1(t): X\rightarrow X$ is compact, strongly positive and analytic for any $t>0$. We also define
\begin{equation*}
 (T_2(t)\varphi)(x)=e^{-m_2t}\varphi,
\end{equation*}
for any $\varphi\in X$, $t\geq 0$. Then $T(t):=(T_1(t),T_2(t)): X\times X\rightarrow X\times X$, $t\geq 0$, defines a $C_0$ semigroup. Define the nonlinear operator $B=(B_1,B_2): X_+\times X_+\rightarrow X\times X$ by
\begin{equation}
 \begin{split}
  B_1(\phi)(x)&=\ds\frac{A_b(x)\mu}{A_d(x)}\phi_2-\sigma\phi_1,\\
  B_2(\phi)(x)&=\phi_2g(x,e^{\alpha x}\phi_2)+\ds\frac{A_d(x)\sigma}{A_b(x)}\phi_1-\mu\phi_2,
 \end{split}
\end{equation}
for $x\in[0,L]$ and $\phi=(\phi_1,\phi_2)\in X_+\times X_+$. Then system \eqref{wz} can be rewritten as the following integral equation
\begin{equation}
  U(t)=T(t)\phi+\int_0^t T(t-s)B(U(s))ds,
\end{equation}
where $U(t)=(w(t),z(t))$ and $\phi=(\phi_1,\phi_2)\in X_+\times X_+$. By \cite[Theorem 1 and Remark 1.1]{martin1990}, it follows that for any $(w_0,z_0)\in X_+\times X_+$, system \eqref{wz} has a unique non-negative mild solution $(w(x,t;w_0,z_0),z(x,t;w_0,z_0))$  with initial condition $(w_0,z_0)$. Moreover, $(w(x,t;w_0,z_0),z(x,t;w_0,z_0))$ is a classical solution of system \eqref{wz} for $t>0$. Then, we can have the local existence and positivity of solutions of system \eqref{wz} and  \eqref{1}.

\begin{lemma}
Suppose that  $A_b(x)$ and $A_d(x)$ and parameters satisfy (A1)-(A2), $g(x,u)$ satisfies (g1)-(g2), then system \eqref{1} has a unique solution for any initial value in $X_+\times X_+$ and the solutions to \eqref{1} remain non-negative on their interval of existence.
\end{lemma}

Next we discuss the global existence of the solutions of system \eqref{1}. To achieve that, we start with the boundedness of the steady state solutions of system \eqref{1}.

\begin{proposition}\label{pro:bound}
Suppose that $g(x,u)$ satisfies (g1)-(g2) and $r(x)$ is defined in (g2). Let $(u(x), v(x))$ be a positive steady state solution of system \eqref{1}, then for $x\in[0,L]$,
\begin{equation}\label{ux}
   u(x)\leq e^{\alpha x}M\overline{\theta}_1,\;\;
   v(x)\leq e^{\alpha x} M \ds\max \{1, \overline{\theta}_1\overline{\theta}_2\},
\end{equation}
where
\begin{equation}
    M=\ds\max_{y\in [0,L]}r(y), \;\; \alpha=\ds\frac{q}{d},
\end{equation}
and
\begin{equation}\label{theta}
    \overline{\theta}_1=\ds\max_{y\in [0,L]}\ds\frac{A_b(y)}{A_d(y)}\frac{\mu}{\sigma+m_1}, \;\; \overline{\theta}_2=\ds\max_{y\in [0,L]}\ds\frac{A_d(y)}{A_b(y)}\ds\frac{\sigma}{\mu+m_2}.
\end{equation}
\end{proposition}

\begin{proof}
Using the transform $u=e^{\alpha x}w$ and $v=e^{\alpha x}z$ on system \eqref{1ss}, we obtain the following system
\begin{equation}\label{2ss}
\begin{cases}
  dw_{xx}+qw_x+\ds\frac{A_b(x)}{A_d(x)}\mu z-\sigma w-m_1 w=0, & \qquad 0<x<L,\\
  zg(x,e^{\alpha x}z)-m_2 z-\mu z+\ds\frac{A_d(x)}{A_b(x)}\sigma w=0, & \qquad 0\leq x\leq L,\\
  -dw_{x}(0)+b_u qw(0)=0,\\
  dw_{x}(L)+b_d qw(L)=0.\\
\end{cases}
\end{equation}
Multiplying the second equation of \eqref{2ss} by $\ds\frac{A_b(x)}{A_d(x)}$ and adding to the first equation of \eqref{2ss}, we have
\begin{equation}\label{3.9}
  dw_{xx}+qw_x-m_1 w+\ds\frac{A_b(x)}{A_d(x)}zg(x,e^{\alpha x}z)-\ds\frac{A_b(x)}{A_d(x)}m_2 z=0.
\end{equation}
Let $w(x_0)=\ds\max_{x\in [0,L]}w(x)$ for $x_0\in [0,L]$. If $x_0\in (0,L)$, then $w_{xx}(x_0)\leq 0$ and $w_x(x_0)=0$. Consequently, from \eqref{3.9}
\begin{equation}\label{x0}
 \ds\frac{A_b(x_0)}{A_d(x_0)}z(x_0)g(x,e^{\alpha x}z(x_0))>0.
\end{equation}
Now from (g2) and $z(x_0)>0$ , $e^{\alpha x_0}z(x_0)<r(x_0)$, which implies that $z(x_0)<M_0$, where $M_0=\ds\max_{y\in[0,L]}e^{-\alpha y}r(y)\le M$. Using the first equation of system \eqref{2ss}, and the fact that $w_{xx}(x_0)\leq 0$, $w_x(x_0)=0$, we have that for $x\in [0,L]$,
\begin{equation*}
 \ds\max_{y\in [0,L]}\ds\frac{A_b(y)}{A_d(y)}\mu M\ge \ds\frac{A_b(x_0)}{A_d(x_0)}\mu z(x_0)>(\sigma+m_1)w(x_0)\ge (\sigma+m_1)w(x),
\end{equation*}
which implies the estimate for $u(x)$ in \eqref{ux}.

Now for the bound of $z(x)$, if  $e^{\alpha x}z(x)\leq M$, then we can obtain $z(x)\leq Me^{-\alpha x}\leq M$; or if $e^{\alpha x}z(x)> M$, from which we have $g(x, e^{\alpha x}z(x))\leq 0$. Then from the second equation of system \eqref{2ss}, we know that
\begin{equation*}
  (\mu+m_2)z(x)\leq[\mu+m_2-g(x,e^{\alpha x}z(x))]z(x)=\ds\frac{A_d(x)}{A_b(x)}\sigma w(x),
\end{equation*}
which implies that $z(x)\le \overline{\theta}_1\overline{\theta}_2 M$ where $\overline{\theta}_1,\overline{\theta}_2$ are defined in \eqref{theta}. Combining the two cases, we obtain the estimate for $u(x)$ in \eqref{ux}.
\end{proof}

Now we have the following result on the global dynamics of \eqref{1}.
\begin{theorem}\label{ge} Suppose that $g(x,u)$ satisfies (g1)-(g2), then \eqref{1} has a unique positive solution $(u(x,t),v(x,t))$ defined on $t\in[0,\infty)$, and the solutions of \eqref{1} generates a dynamical system in $X_1$, where
\begin{equation}
\begin{split}
X_1=&\{(\phi,\psi)\in W^{2,2}(0,L)\times C(0,L): \; \phi(x)\geq 0,\psi(x)\geq 0,\\ \; &d\phi'(0)-q\phi(0)=b_u q\phi(0), \;\;
d\phi'(L)-q\phi(L)=-b_d q\phi(L)\}.
\end{split}
\end{equation}
Furthermore, \eqref{1} is a point dissipative system.
\end{theorem}

\begin{proof}
We consider the equivalent system \eqref{wz} of \eqref{1}. Assume that $(u(x,t),v(x,t))$ is a solution of system \eqref{1}, then $(w(x,t),z(x,t))$ is a solution of system \eqref{wz}. We choose
\begin{equation}
M_1= \max\left\{ M \ds\max\{1, \overline{\theta}_1,\overline{\theta}_1\overline{\theta}_2\}, \max_{y\in [0,L]}e^{-\alpha y}u_0(y), \max_{y\in [0,L]}e^{-\alpha y}v_0(y)\right\},
\end{equation}
where $\overline{\theta}_1,\overline{\theta}_2$ are defined in \eqref{theta}. Then $(M_1,M_1)$ is an upper solution of \eqref{wz} and $(0,0)$ is a lower solution of \eqref{wz}. According to \cite[Theorem 4.1]{pao1996JMAA}, we obtain that
\begin{equation*}
0\leq w(x,t)\leq w_1(x,t),\qquad 0\leq z(x,t)\leq z_1(x,t),
\end{equation*}
where $(w_1(x,t),z_1(x,t))$ is the solution of \eqref{wz} with initial condition $w_1(x,0)=M_1$ and $z_1(x,0)=M_1$. Moreover the solution $(w_1(x,t),z_1(x,t))$ is non-increasing in $t$ and $\ds\lim_{t\to+\infty}(w_1(x,t),z_1(x,t))=(w_{max}(x),z_{max}(x))$ which is maximum steady state of \eqref{wz} not larger than $(M_1,M_1)$. From Proposition \ref{pro:bound}, we obtain that $(u(x,t),v(x,t))$ exists globally for $t\in (0,\infty)$, stays positive and
\begin{equation}\label{3.13}
 \limsup_{t\to\infty}u(x,t)\leq e^{\alpha x}M \overline{\theta}_1, \;\;
 \limsup_{t\to\infty}v(x,t)\leq e^{\alpha x} M \ds\max\{1, \overline{\theta}_1\overline{\theta}_2\}.
\end{equation}
So the system \eqref{1} is point dissipative.
\end{proof}

\subsection{Global attractor}
From Proposition \ref{ge}, it follows that solutions of system \eqref{1} are uniformly bounded. Thus, we can define a solution semiflow of \eqref{1} on $X_+\times X_+$ by
\begin{equation}
\Sigma(t)\phi=\begin{pmatrix}
u(t,\cdot,\phi_1(x))\\
v(t,\cdot,\phi_2(x))\\
\end{pmatrix} \qquad \forall \phi=(\phi_1,\phi_2)\in X_+\times X_+,\;\;t\geq 0.
\end{equation}
$\Sigma(t)\phi$ is the solution of \eqref{1} with initial condition $(\phi_1,\phi_2)$ and $\Sigma(t)$ is a positive semigroup for all $t\geq 0$. 
Notice that $\Sigma(t)$ is not compact since the second equation in \eqref{1} has no diffusion term. Due to the lack of compactness, we need to impose the following condition
\begin{equation}\label{h1}
  \overline{f_v}<m_2+\mu,
\end{equation}
where $\overline{f_v}$ is defined in \eqref{fv}, and recall that $f(x,v)=vg(x,v)$. Recall that the Kuratowski measure of noncompactness (see \cite[Chapter 1]{zhao2012}), which is defined by the formula
\begin{equation}
\alpha(K):=\inf\{r:K \;\;\text{has a finite cover of diameter}< r\},
\end{equation}
on any bounded set $K\subset X_+$. And the diameter of the set is defined by the relation $\text{diam}K=\sup\{dist(x,y):x,y\in K\}$. We set $\alpha(K)=\infty$ whenever $K$ is bounded. From the definition of $\alpha$-contracting, we know that $\alpha(K)\leq \text{diam}K$, $\alpha(K)=0$ if and only if the closure $\overline{K}$ of $K$ is compact and the set $K$ is bounded if and only if $\alpha(K)<\infty$.

\begin{lemma}\label{contracting}
Suppose that $g(x,u)$ satisfies (g1)-(g2) and \eqref{h1}, then $\Sigma(t)$ is $\alpha$-contracting in the sense that
\begin{equation}
 \ds\lim_{t\rightarrow\infty}\alpha(\Sigma(t)K)=0,
\end{equation}
for any bounded set $K\subset X_+$.
\end{lemma}
\begin{proof}
The right hand side  of the $v$-equation in \eqref{1} is represented by
\begin{equation}
H(u,v)=vg(x,v)-m_2 v-\mu v+\ds\frac{A_d(x)}{A_b(x)}\sigma u.
\end{equation}
Then from \eqref{h1}, there exists a real number $r>0$ satisfies
\begin{equation}
\ds\frac{\partial H(u,v)}{\partial v}=f_v(x,v)-m_2-\mu<-r<0.
\end{equation}
With this inequality, the rest of the proof is similar to the one in Lemmas 3.2 and 4.1 in \cite{hsu2011}.
\end{proof}
Now we are ready to show that solutions of system \eqref{1} converge to a compact attractor on $X_+\times X_+$ when $t\rightarrow\infty$ under the condition \eqref{h1}.
\begin{theorem}
Suppose that $g(x,u)$ satisfies (g1)-(g2), then $\Sigma(t)$ admits a global attractor on $X_+\times X_+$ provided that \eqref{h1} holds.
\end{theorem}
\begin{proof}
From Lemma \ref{contracting} and Theorem \ref{ge}, it follows that $\Sigma(t)$ is $\alpha$-contracting on $X_+$ and system \eqref{1} is point dissipative. By Proposition \ref{pro:bound}, we also know that the positive orbits of bounded subsets of $X_+$ for $\Sigma(t)$ are uniformly bounded. Then according to \cite[Theorem 2.6]{mzhao2005}, $\Sigma(t)$ has a global attractor that attracts every bounded set in $X_+$.
\end{proof}

From the discussion above, we can obtain the convergence of the solutions to equilibria of system \eqref{1} by constructing a Lyapunov function.

\begin{theorem}\label{ly}
Suppose that $g(x,u)$ satisfies (g1)-(g2) and \eqref{h1}, then for any $(u_0,v_0)\in X_1$ and $u_0\not\equiv 0$, $v_0\not\equiv 0$, the $\omega$-limit set $\omega((u_0,v_0))\subset S$, where $S$ is the set of non-negative steady state solutions of \eqref{1}.
\end{theorem}

\begin{proof}
We prove that the solution $(u(x,t),v(x,t))$ is always convergent. For that purpose, we define a function
\begin{equation}\label{lyp}
  \begin{split}
  E(u,v)=&\int_0^Le^{-\alpha x}\left[\ds\frac{d}{2}(u_x)^{2}-\ds\frac{A_b(x)}{A_d(x)}\mu uv+\ds\frac{\sigma+m_1}{2}u^2\right]dx\\
  &-\ds\frac{\mu}{\sigma}\int_0^Le^{-\alpha x}\ds\frac{A_b^2(x)}{A_d^2(x)}\left[F(x,v)-\ds\frac{\mu+m_2}{2}v^2\right]dx+\ds\frac{q}{2}(1+b_u)u^2(0,t)\\
  &-\ds\frac{q}{2}(1-b_d)e^{-\alpha L}u^2(L,t),
  \end{split}
\end{equation}
for $(u,v)\in X_1$, where $F(x,v)=\ds\int_0^v f(x,s) ds$. Assume that $(u(x,t),v(x,t))$ is a solution of system \eqref{1}, we have
\begin{equation*}
\begin{split}
  \ds\frac{d}{dt}E(u(\cdot,t),v(\cdot,t))=&\int_0^Le^{-\alpha x}(du_xu_{xt}-\ds\frac{A_b(x)}{A_d(x)}\mu(u_tv+uv_t)+(\sigma+m_1)uu_t)dx\\
  &-\ds\frac{\mu}{\sigma}\int_0^Le^{-\alpha x}\ds\frac{A_b^2(x)}{A_d^2(x)}(fv_t-(\mu+m_2)vv_t)dx\\
  &+q(1+b_u)u(0,t)u_t(0,t)-q(1-b_d)e^{-\alpha L}u(L,t)u_t(L,t) \\
  =&[de^{-\alpha x}u_xu_t]\mid_0^L+q(1+b_u)u(0,t)u_t(0,t)\\
  &-q(1-b_d)e^{-\alpha L}u(L,t)u_t(L,t)-\int_0^L(de^{-\alpha x}u_x)_xu_tdx\\
  &-\int_0^Le^{-\alpha x}(\ds\frac{A_b(x)}{A_d(x)}\mu(u_tv+uv_t)-(\sigma+m_1)uu_tdx\\
  &-\ds\frac{\mu}{\sigma}\int_0^Le^{-\alpha x}\ds\frac{A_b^2(x)}{A_d^2(x)}(fv_t-(\mu+m_2)vv_t)dx\\
  =&-\int_0^Le^{-\alpha x}u_t(du_{xx}-qu_x+\ds\frac{A_b(x)}{A_d(x)}\mu v-\sigma u-m_1u)dx\\
  &-\int_0^Le^{-\alpha x}v_t\ds\frac{\mu A_b^2(x)}{\sigma A_d^2(x)}(f+\ds\frac{A_d(x)}{A_b(x)}\sigma u-\mu v-m_2v)dx\\
  =&-\int_0^Le^{-\alpha x}(u_t)^2dx-\ds\frac{\mu}{\sigma}\int_0^Le^{-\alpha x}\ds\frac{A_b^2(x)}{A_d^2(x)}(v_t)^2dx\\
  \leq&0.
\end{split}
\end{equation*}
According to (g2), we have $F(x,v)\leq F(x,r(x))$ and $r(x)\leq M$. Hence when $t>T$ for some $T>0$ large,  from \eqref{3.13},
\begin{equation*}\label{lb}
\begin{split}
 &E(u(\cdot,t),v(\cdot,t))\\
 \geq& -\mu\int_0^Le^{-\alpha x}\ds\frac{A_b(x)}{A_d(x)}uvdx-\ds\frac{\mu}{\sigma}\int_0^Le^{-\alpha x}\ds\frac{A_b^2(x)}{A_d^2(x)}F(x,r(x))dx
 -\ds\frac{q}{2}e^{-\alpha L}u^2(L,t) \\
  \geq& -\ds\max_{y\in[0,L]}\ds\frac{\mu A_b(x)}{A_d(x)}e^{2\alpha L}M^2L\overline{\theta}_1\ds\max \{1, \overline{\theta}_1\overline{\theta}_2\}
  -\ds\max_{y\in[0,L]}\ds\frac{\mu A_b^2(x)}{\sigma A_d^2(x)}M_2L
  -\ds\frac{q M^2\overline{\theta}_1^2}{2}e^{\alpha L},
  \end{split}
\end{equation*}
where $M_2=\displaystyle\max_{y\in[0,L]}F(y,r(y))$. Therefore $E(u(\cdot,t),v(\cdot,t))$ is bounded from below. Notice $\ds\frac{d}{dt}E(u,v)=0$ holds if and only if $u_t=0$ and $v_t=0$, which also means that $(u,v)$ is a steady state solution of system \eqref{1}. From Lemma \ref{contracting}, the solutions of orbits or \eqref{1} are pre-compact, then from the LaSalle's Invariance Principle \cite[Theorem 4.3.4]{henry1981}, we have that for any initial condition $u_0(x)\geq 0$ and $v_0(x)\geq 0$, the $\omega$-limit set of $(u_0,v_0)$ is contained in the largest invariant subset of $S$. If every element in $S$ is isolated, then the $\omega$-limit set is a single steady state.
\end{proof}

It is clear that the Lyapunov function constructed above decreases along the solution orbit of \eqref{1} even without the condition \eqref{h1}, but we cannot conclude that the solutions converge without \eqref{h1}. The dynamic behavior of \eqref{1} without the condition \eqref{h1} remains as an interesting question.


\subsection{Eigenvalue problem}
We consider the linear stability of steady state solutions of system \eqref{1}. Suppose that $(u^*,v^*)$ is a non-negative steady state solution of system \eqref{1}. Substituting $u=e^{\lambda t}\phi_1$ and $v=e^{\lambda t}\phi_2$, where $\phi=(\phi_1,\phi_2)\in X_1$, into system \eqref{1}, we get the following associated eigenvalue problem:
\begin{equation}\label{1eigen}
\begin{cases}
  \lambda\phi_1=d(\phi_1)_{xx}-q (\phi_1)_x+\ds\frac{A_b(x)}{A_d(x)}\mu \phi_2-\sigma\phi_1-m_1\phi_1, & \qquad 0<x<L,\\
  \lambda\phi_2=f_v(x,v^*)\phi_2-m_2 \phi_2-\mu \phi_2+\ds\frac{A_d(x)}{A_b(x)}\sigma \phi_1, & \qquad 0\leq x\leq L,\\
  d(\phi_1)_x(0)-q\phi_1(0)=b_u q\phi_1(0), &\\
  d(\phi_1)_x(L)-q\phi_1(L)=-b_d q\phi_1(L), &
\end{cases}
\end{equation}
where $f_v(x,v^*)=g(x,v^*)+v^*g_v(x,v^*)$. Let $Y=C([0,L])\times C([0,L])$ and denote the linearized operator $\mathcal{L}: X_1\rightarrow Y$ of system \eqref{1ss} by
\begin{equation}
\mathcal{L}=\begin{pmatrix}
d\ds\frac{\partial^2}{\partial x^2}-q\ds\frac{\partial}{\partial x}\\
0\\
\end{pmatrix}+\begin{pmatrix}
-\sigma-m_1 & \ds\frac{A_b(x)\mu}{A_d(x)}\\
\ds\frac{A_d(x)\sigma}{A_b(x)} & f_v(x,v(x))-\mu-m_2\\
\end{pmatrix}.
\end{equation}

The following proposition provides the information of the spectral set  $\sigma(\mathcal{L})$ of  the linearized operator $\mathcal{L}$, especially  the principal eigenvalue of \eqref{1eigen}.
\begin{proposition}\label{eigenp}
Suppose that $g(x,u)$ satisfies (g1)-(g3), $d>0$ and  $q, b_u, b_d\geq 0$. Let $(u^*(x),v^*(x))$ be a non-negative steady state solution of \eqref{1}. Then
\begin{enumerate}\label{eigen}
 \item The eigenvalue problem \eqref{1eigen} has a simple principal eigenvalue $\lambda_1$ with a positive eigenfunction $\phi=(\phi_1,\phi_2)$. Moreover, the principal eigenvalue $\lambda_1$ satisfies 
 \begin{equation}\label{R}
     -\lambda_1=\frac{E_1(\psi_1,\psi_2)}{\kappa(\psi_1,\psi_2)},
 \end{equation}
where $\psi=e^{-\alpha x}\phi\in X_2$,
\begin{equation}\label{eigenvalueafter}
   \begin{split}
    &E_1(\psi_1,\psi_2)=\ds\int_0^L e^{\alpha x}\left[d(\psi_1)_x^2-2\ds\frac{A_b(x)}{A_d(x)}\mu\psi_1\psi_2+(\sigma+m_1)\psi_1^2\right]dx\\
    -&\ds\frac{\mu}{\sigma}\int_0^Le^{\alpha x}\ds\frac{A_b^2(x)}{A_d^2(x)}(f_v(x,v^*)-\mu-m_2)\psi_2^2 dx+qb_u\psi_1^2(0)+qb_de^{\alpha L}\psi_1^2(L),
    \end{split}
\end{equation}
\begin{equation}\label{kappa}
    \kappa(\psi_1,\psi_2)=\ds\int_0^Le^{\alpha x}\left(\psi_1^2+\ds\frac{A_b^2(x)\mu}{A_d^2(x)\sigma}\psi_2^2\right)dx,
\end{equation}
 $\alpha=q/d$ and $X_2=H^1(0,L)\times C(0,L)$.
\item\label{s2} The spectral set $\sigma(\mathcal{L})$ of the linearized operator $\mathcal{L}$ consists of isolated eigenvalues and the set $[\ds\min_{x\in[0,L]}f_v(x,v^*(x))-m_2-\mu, \ds\max_{x\in[0,L]}f_v(x,v^*(x))-m_2-\mu]$.

\item\label{s5} If $\ds \max_{x\in [0,L]} f_v(x,v^*(x))>m_2+\mu$, then $(u^*(x),v^*(x))$ is unstable.
\item\label{s} If $\ds\max_{x\in[0,L]} f_v(x,v^*(x))<m_2+\ds\frac{ m_1\mu}{m_1+\sigma}$, then $(u^*(x),v^*(x))$ is linearly stable. 
\end{enumerate}


\end{proposition}
\begin{proof}
1. The existence of the simple eigenvalue $\lambda_1$ with positive eigenfunction follows from \cite[Lemma 4.1]{zhao2012} (see also \cite[Theorem 3]{huang2016}). Using the transform $\phi=e^{\alpha x}\psi$,  system \eqref{1eigen} becomes
\begin{equation}\label{2eigen}
\begin{cases}
  \lambda\psi_1=d (\psi_1)_{xx}+q (\psi_1)_x+\ds\frac{A_b(x)}{A_d(x)}\mu \psi_2-\sigma\psi_1-m_1\psi_1, & \qquad 0<x<L,\\
  \lambda\psi_2=f_v(x,v^*)\psi_2-m_2 \psi_2-\mu \psi_2+\ds\frac{A_d(x)}{A_b(x)}\sigma \psi_1, & \qquad 0\leq x\leq L,\\
  -d(\psi_1)_{x}(0)+b_u q\psi_1(0)=0,\\
  d(\psi_1)_{x}(L)+b_d q\psi_1(L)=0.
\end{cases}
\end{equation}
A direct calculation shows that a solution $(\psi_1,\psi_2)$ of  \eqref{2eigen} satisfies \eqref{R}.

2. 
Let $\underline{A}=\ds\min_{x\in[0,L]}f_v(x,v^*(x))$ and
 $\overline{A}=\ds\max_{x\in[0,L]}f_v(x,v^*(x))$.
If $\lambda\not\in[\underline{A}-m_2-\mu, \overline{A}-m_2-\mu]$, and $\la$ is an eigenvalue of \eqref{1eigen}, then from the second equation in \eqref{1eigen}, we have
\begin{equation}\label{phi2}
  \phi_2=\ds\frac{-A_d(x)\sigma\phi_1}{A_b(x)(f_v(x,v^*)-m_2-\mu-\lambda)},
\end{equation}
and the first equation of \eqref{1eigen} becomes
\begin{equation}\label{reduse}
d\phi_1''-q \phi_1'-(\sigma+m_1+\lambda)\phi_1-\ds\frac{\sigma\mu}{f_v(x,v^*)-m_2-\mu-\lambda}\phi_1=0.
\end{equation}
One can follow the arguments in \cite[Section 4.4]{marciniak2017} to show that the set $\sigma(\mathcal{L})\backslash [\underline{A}-m_2-\mu, \overline{A}-m_2-\mu]$ consists of isolated eigenvalues from the analytic Fredholm theorem (see \cite[Theorem 4.6]{marciniak2017}). On the other hand, by following the same proof as  \cite[Theorem 4.5]{marciniak2017}, we can show that each point in $[\underline{A}-m_2-\mu, \overline{A}-m_2-\mu]$ is in the continuous spectrum of $\mathcal{L}$.

3. If $\ds\max_{x\in [0,L]}f_v(x,v^*(x))>\mu+m_2$, then the set of continuous spectrum  $[\underline{A}-m_2-\mu, \overline{A}-m_2-\mu]\cap(0,\infty)\ne\emptyset$, so  $(u^*(x),v^*(x))$ must be unstable. In addition we also prove that $\la_1>0$ in this case. Assume that $\overline{A}=\ds\max_{x\in[0,L]}f_v(x,v^*(x))=f_v(x_0,v^*(x_0))$ for $x_0\in [0,L]$. From the second equation of \eqref{1eigen}, $A_b(x_0)>0$, $A_d(x_0)>0$ and $\phi_1(x_0)>0$, we have
\begin{equation*}
\lambda_1\phi_2(x_0)>[f_v(x_0,v^*(x_0))-m_2 -\mu] \phi_2(x_0)=[\overline{A}-m_2 -\mu] \phi_2(x_0)>0.
\end{equation*}
Hence $\lambda_1>0$ as $\phi_2(x_0)>0$.

4. Since $f_v(x,v^*(x))\le\ds\max_{x\in [0,L]}f_v(x,v^*(x))<m_2+\ds\frac{m_1\mu}{m_1+\sigma}$, then
\begin{equation*}
   \begin{split}
    &E_1(\psi_1,\psi_2)\\
    =&\ds\int_0^L e^{\alpha x}[d(\psi_1)_x^2(x)-2\ds\frac{A_b(x)}{A_d(x)}\mu\psi_1\psi_2+(\sigma+m_1)\psi_1^2]dx\\
    &-\ds\frac{\mu}{\sigma}\int_0^Le^{\alpha x}\ds\frac{A_b^2(x)}{A_d^2(x)}(f_v(x,v)-\mu-m_2)\psi_2^2 dx
    +qb_u\psi_1^2(0)+qb_de^{\alpha L}\psi_1^2(L)\\
    >&\ds\int_0^L e^{\alpha x}[-2\ds\frac{A_b(x)}{A_d(x)}\mu\psi_1\psi_2+(\sigma+m_1)\psi_1^2
    +\ds\frac{A_b^2(x)\mu}{A_d^2(x)\sigma}(\mu+m_2-f_v(x,v))\psi_2^2]dx\\
    =&\ds\int_0^L e^{\alpha x}(\ds\frac{A_b(x)\mu}{A_d(x)\sqrt{\sigma+m_1}}\psi_2-\sqrt{\sigma+m_1}\psi_1)^2dx\\
    &+\ds\int_0^L e^{\alpha x}\ds\frac{A_b^2(x)\mu[(m_1+\sigma)(\mu+m_2)-\mu\sigma-(m_1+\sigma)f_v]}{A_d^2(x)\sigma(\sigma+m_1)}\psi_2^2dx>0.
    \end{split}
  \end{equation*}
Then from \eqref{R} and $\kappa(\psi_1,\psi_2)>0$, the principal eigenvalue $\lambda_1<0$. On the other hand, since $\ds\max_{x\in[0,L]}f_v(x,v^*(x))-m_2-\mu<0$, then all continuous spectrum points are also negative.  Hence the non-negative steady state solution $(u^*(x),v^*(x))$ is linearly stable as the spectral set $\sigma(\mathcal{L})$ lies in the negative complex half plane.
\end{proof}

Notice that $(0,0)$ is always a steady state solution of system \eqref{1} with strong Allee effect growth rate, then for the stability of the zero steady state $(0,0)$, we have the following result.
\begin{corollary}\label{zeross}
Suppose that $g(x,v)$ satisfies (g1)-(g3) and (g4c), $b_u\geq 0$ and $b_d\geq 0$, the zero steady state $(0,0)$ of system \eqref{1} is linearly stable.
\end{corollary}
\begin{proof}
Since $f_v(x,0)=g(x,0)<0$, then we have $\ds\max_{x\in [0,L]} g(x,0)<0<m_2+\ds\frac{\mu m_1}{m_1+\sigma}$. Then from part 3  of Proposition \ref{eigenp}, $(0,0)$ is always linearly stable.
\end{proof}

Unlike the strong Allee effect case, the zero steady state $(0,0)$ of \eqref{1} is not always stable if the growth rate is logistic or weak Allee effect type. Here we show how the principal eigenvalue $\la_1=\la_1(q,m_2)$ at the zero steady state defined in Proposition \ref{eigenp} varies with respect to $q$, which also implies the stability of the zero steady state. For scalar reaction-diffusion-advection equation,  when the population follows a typical logistic growth, there often exists a critical parameter value (diffusion coefficient, advection coefficient, domain size, growth rate) for the population persistence or extinction \cite{LLL2015JBD,LL2014JMB,Mckenzie2012}. Here we show a similar result holds for the benthic-drift model \eqref{1}, following methods in \cite{LL2014JMB,Louzhou2015} for scalar equation on a river. 

\begin{proposition}\label{eigenp0}
Suppose that $g(x,u)$ satisfies (g1)-(g3) and (g4a) or (g4b), $d>0$ and $q, b_u, b_d\geq 0$. The corresponding eigenvalue problem at $(0,0)$ is
\begin{equation}\label{0eigen}
\begin{cases}
  \lambda\phi_1=d(\phi_1)_{xx}-q (\phi_1)_x+\ds\frac{A_b(x)}{A_d(x)}\mu \phi_2-\sigma\phi_1-m_1\phi_1, & \qquad 0<x<L,\\
  \lambda\phi_2=f_v(x,0)\phi_2-m_2 \phi_2-\mu \phi_2+\ds\frac{A_d(x)}{A_b(x)}\sigma \phi_1, & \qquad 0\leq x\leq L,\\
  d(\phi_1)_x(0)-q\phi_1(0)=b_u q\phi_1(0), &\\
  d(\phi_1)_x(L)-q\phi_1(L)=-b_d q\phi_1(L). &
\end{cases}
\end{equation}
The principal eigenvalue $\lambda_1(q,m_2)$ of \eqref{0eigen} satisfies
\begin{enumerate}
\item if $b_d>0$ and $\ds\max_{x\in [0,L]}f_v(x,0)<m_2+\ds\frac{m_1\mu}{m_1+\sigma}$, then $\ds\lim_{q\rightarrow\infty} \lambda_1(q,m_2)=-\infty$;
\item\label{qde} if $b_d>1/2$, then $\lambda_1(q,m_2)$ is strictly decreasing in $q$;
\item\label{dm2} if $b_d>0$, then $\lambda_1(q,m_2)$ is strictly decreasing in $m_2$.
\end{enumerate}
\end{proposition}

\begin{proof}
Using the transform $\phi=e^{\alpha x}\psi$, system \eqref{0eigen} becomes
\begin{equation}\label{02eigen}
\begin{cases}
  \lambda\psi_1=d (\psi_1)_{xx}+q (\psi_1)_x+\ds\frac{A_b(x)}{A_d(x)}\mu \psi_2-\sigma\psi_1-m_1\psi_1, & \qquad 0<x<L,\\
  \lambda\psi_2=f_v(x,0)\psi_2-m_2 \psi_2-\mu \psi_2+\ds\frac{A_d(x)}{A_b(x)}\sigma \psi_1, & \qquad 0\leq x\leq L,\\
  -d(\psi_1)_{x}(0)+b_u q\psi_1(0)=0,\\
  d(\psi_1)_{x}(L)+b_d q\psi_1(L)=0.
\end{cases}
\end{equation}
1. To prove this, we use a different characterization of $\la_1(q,m_2)$. By using the transform $(\psi_1,\psi_2)=e^{-\delta\alpha x}(\xi_1,\xi_2)$, where $0<\delta<\ds\min\{1,b_d\}$, then \eqref{R} becomes
\begin{equation*}
  -\lambda_1(q,m_2)=\ds\frac{E_{\delta}(\xi_1,\xi_2)}{\kappa_{\delta}(\xi_1,\xi_2)},
\end{equation*}
where $(\xi_1,\xi_2)$ is the corresponding eigenfunction,
\begin{equation*}
\begin{split}
    E_{\delta}(\xi_1,\xi_2)=&\ds\int_0^L e^{(1-2\delta)\alpha x}[d(\xi_1)_x^2-2\delta q\xi_1(\xi_1)_x+d\delta^2\alpha^2\xi_1^2-2\ds\frac{A_b(x)\mu}{A_d(x)}\xi_1\xi_2\\
    &+(\sigma+m_1)\xi_1^2]dx-\ds\int_0^Le^{(1-2\delta)\alpha x}\ds\frac{A_b^2(x)\mu}{A_d^2(x)\sigma}(f_v(x,0)-\mu-m_2)\xi_2^2 dx\\
    &+qb_u\xi_1^2(0)+qb_de^{(1-2\delta)\alpha L}\xi_1^2(L),
\end{split}
\end{equation*}
and
\begin{equation*}
    \kappa_{\delta}(\xi_1,\xi_2)=\int_0^Le^{(1-2\delta)\alpha x}\left(\xi_1^2+\frac{A_b^2(x)\mu}{A_d^2(x)\sigma}\xi_2^2\right)dx.
\end{equation*}
We can calculate that
\begin{equation}
  \begin{split}
  E_{\delta}(\xi_1,\xi_2)=&\ds\int_0^L e^{(1-2\delta)\alpha x}[d(\xi_1)_x^2+\delta(1-2\delta)\alpha q\xi_1^2+d\delta^2\alpha^2\xi_1^2]dx\\
 &+\ds\int_0^L e^{(1-2\delta)\alpha x}\left(\sqrt{\sigma+m_1}\xi_1-\ds\frac{A_b(x)\mu}{A_d(x)\sqrt{\sigma+m_1}}\xi_2\right)^2 dx\\
 &+\ds\int_0^L e^{(1-2\delta)\alpha x}\frac{A_b^2(x)\mu}{A_d^2(x)\sigma}\xi_2^2\left(\frac{m_1 \mu}{\sigma+m_1}
 +m_2-f_v(x,0)\right) dx\\
 &+q(b_u+\delta)\xi_1^2(0)+q(b_d-\delta)e^{(1-2\delta)\alpha L}\xi_1^2(L)\\
>&\ds\frac{q^2}{d}(\delta-\delta^2)\int_0^Le^{(1-2\delta)\alpha x}\xi_1^2dx-p_*\int_0^Le^{(1-2\delta)\alpha x}\frac{A_b^2(x)\mu}{A_d^2(x)\sigma}\xi_2^2dx,
  \end{split}
\end{equation}
where $p_*=\ds\max_{x\in[0,L]}f_v(x,0)$, and $\ds \frac{\mu m_1}{\sigma+m_1}+m_2-f_v(x,0)\ge 0$.
Thus,
\begin{equation}\label{3.29}
  -\lambda_1(q,m_2)>\ds\frac{\ds\frac{q^2}{d}(\delta-\delta^2)I_1}{I_1+I_2}-p_*,
\end{equation}
where
\begin{equation*}
    I_1=\int_0^Le^{(1-2\delta)\alpha x}\xi_1^2 dx, \;\; I_2=\int_0^Le^{(1-2\delta)\alpha x}\frac{A_b^2(x)\mu}{A_d^2(x)\sigma}\xi_2^2dx.
\end{equation*}
Therefore we obtain that
\begin{equation}\label{i12}
\ds\frac{q^2}{d}(\delta-\delta^2)I_1+(\lambda_1(q,m_2)-p_*)(I_1+I_2)< 0.
\end{equation}
From the second equation of \eqref{02eigen}, we know that
\begin{equation}
\xi_2=\ds\frac{A_d(x)\sigma}{A_b(x)(\lambda_1(q,m_2)+m_2+\mu-p(x))}\xi_1.
\end{equation}
Substituting it into \eqref{i12}, and after some calculations, we can get
\begin{equation}
  \ds\int_0^Le^{(1-2\delta)\alpha x}\xi_1^2\left[\ds\frac{q^2}{d}(\delta-\delta^2)+\lambda_1(q)-p_*+\ds\frac{\mu\sigma(\lambda_1-p_*)}{(\lambda_1+m_2+\mu-p_*)^2}\right]dx< 0.
\end{equation}
Therefore, we have
\begin{equation}
  \ds\frac{q^2}{d}(\delta-\delta^2)+\lambda_1(q)-p_*+\ds\frac{\mu\sigma(\lambda_1-p_*)}{(\lambda_1+m_2+\mu-p_*)^2}<0.
\end{equation}
When $q\rightarrow\infty$, we have $\lambda_1(q,m_2)-p_*\rightarrow-\infty$. Thus, we have $\ds\lim_{q\rightarrow\infty} \lambda_1(q,m_2)=-\infty$.

2. Differentiating \eqref{02eigen} with respect to $q$ with $\la=\la_1(q,m_2)$  and denote $\ds\frac{\partial}{\partial q}='$,  we obtain that
\begin{equation}\label{02eigenq}
\begin{cases}
  \lambda_1'\psi_1+\lambda_1\psi_1'=d (\psi_1)'_{xx}+(\psi_1)_x+q (\psi_1)'_x+\ds\frac{A_b(x)\mu}{A_d(x)}\psi'_2-(\sigma+m_1)\psi'_1, &0<x<L,\\
  \lambda_1'\psi_2+\lambda_1\psi_2'=f_v(x,0)\psi'_2-(m_2+\mu)\psi'_2+\ds\frac{A_d(x)\sigma}{A_b(x)}\psi'_1, & 0\leq x\leq L,\\
  -d(\psi_1)'_{x}(0)+b_u\psi_1(0)+b_u q\psi'_1(0)=0,\\
  d(\psi_1)'_{x}(L)+b_d\psi_1(L)+b_d q\psi'_1(L)=0. \\
\end{cases}
\end{equation}
Multiplying the first equation of \eqref{02eigen} by $e^{\alpha x}\psi'_1$ and the first equation of  \eqref{02eigenq} by $e^{\alpha x}\psi_1$, then integrating over $[0,L]$ and subtracting the two equations, we have
\begin{equation}\label{q1}
 \begin{split}
 &d\left[\psi_1e^{\alpha x}(\psi_1)'_x-\psi_1'e^{\alpha x}(\psi_1)_x\right]{\Big |}_0^L+\int_0^Le^{\alpha x}(\psi_1)_x\psi_1dx+q\int_0^Le^{\alpha x}[(\psi_1)'_x\psi_1-(\psi_1)_x\psi'_1]dx\\
 &+\int_0^L\ds\frac{A_b(x)\mu}{A_d(x)}e^{\alpha x}(\psi_1\psi_2'-\psi_2\psi_1') dx=\lambda_1'(q,m_2)\int_0^Le^{\alpha x}\psi_1^2dx.
 \end{split}
\end{equation}
The boundary terms together with the boundary conditions give
\begin{equation}\label{q11}
d[\psi_1e^{\alpha x}(\psi_1)'_x-\psi_1'e^{\alpha x}\psi_x]{\Big |}_0^L=-b_de^{\alpha L}\psi_1^2(L)-b_u\psi_1^2(0).
\end{equation}
The first integral in equation \eqref{q1} becomes
\begin{equation}\label{q12}
\int_0^Le^{\alpha x}(\psi_1)_x\psi_1dx=\int_0^Le^{\alpha x}(\ds\frac{\psi_1^2}{2})_xdx=e^{\alpha L}\ds\frac{\psi_1^2(L)}{2}-\ds\frac{\psi_1^2(0)}{2}-\alpha\int_0^Le^{\alpha x}\ds\frac{\psi_1^2}{2}dx.
\end{equation}
And
\begin{equation}
 \begin{split}
 \int_0^Le^{\alpha x}[(\psi_1)'_x\psi_1-(\psi_1)_x\psi'_1]dx=\alpha\int_0^Le^{\alpha x}[(\psi_1)_x\psi'_1-(\psi_1)'_x\psi_1]dx,
 \end{split}
\end{equation}
which gives
\begin{equation}\label{q13}
 \begin{split}
 \int_0^Le^{\alpha x}[(\psi_1)'_x\psi_1-(\psi_1)_x\psi'_1]dx=0.
 \end{split}
\end{equation}
Multiplying the second equation of \eqref{02eigen} by $e^{\alpha x}\psi'_2$ and the second equation of \eqref{02eigenq} by $e^{\alpha x}\psi_2$, then subtracting the two equations and multiplying by $\ds\frac{A_b^2(x)\mu}{A_d^2(x)\sigma}$, and integrating on $[0,L]$, we  have
\begin{equation}\label{q14}
 \int_0^L\ds\frac{A_b(x)\mu}{A_d(x)}e^{\alpha x}(\psi_1\psi_2'-\psi_2\psi_1') dx=-\lambda_1'(q,m_2)\int_0^L\ds\frac{A_b^2(x)\mu}{A_d^2(x)\sigma}e^{\alpha x}\psi_2^2dx.
\end{equation}
Together with \eqref{q1}, \eqref{q11}, \eqref{q12}, \eqref{q13} and \eqref{q14}, if $b_d>1/2$, we have
\begin{equation}
 \lambda_1'(q,m_2)=-\ds\frac{(b_d-\ds\frac{1}{2})e^{\alpha L}\psi_1^2(L)+(b_u+\ds\frac{1}{2})\psi_1^2(0)+\frac{\alpha}{2}\int_0^Le^{\alpha x}\psi_1^2dx}{\ds\int_0^Le^{\alpha x}\psi_1^2dx+\ds\int_0^L\ds\frac{A_b^2(x)\mu}{A_d(x)^2\sigma}e^{\alpha x}\psi_2^2dx}<0.
\end{equation}

3. Differentiating \eqref{02eigen} with respect to $m_2$ with $\la=\la_1(q,m_2)$  and denote in the following $\ds\frac{\partial}{\partial m_2}='$,  we obtain that
\begin{equation}\label{02eigenm2}
\begin{cases}
  \lambda_1'\psi_1+\lambda_1\psi_1'=d (\psi_1)'_{xx}+q (\psi_1)'_x+\ds\frac{A_b(x)\mu}{A_d(x)}\psi'_2-(\sigma+m_1)\psi'_1, &0<x<L,\\
  \lambda_1'\psi_2+\lambda_1\psi_2'=f_v(x,0)\psi'_2-\psi_2-(m_2+\mu)\psi'_2+\ds\frac{A_d(x)\sigma}{A_b(x)}\psi'_1, & 0\leq x\leq L,\\
  -d(\psi_1)'_{x}(0)+b_u q\psi'_1(0)=0,\\
  d(\psi_1)'_{x}(L)+b_d q\psi'_1(L)=0. \\
\end{cases}
\end{equation}
By using similar calculation as part 2, we have
\begin{equation}
 \lambda_1'(q,m_2)=-\ds\frac{b_de^{\alpha L}\psi_1^2(L)+b_u\psi_1^2(0)+\int_0^Le^{\alpha x}\psi_2^2dx}{\ds\int_0^Le^{\alpha x}\psi_1^2dx+\ds\int_0^L\ds\frac{A_b^2(x)\mu}{A_d(x)^2\sigma}e^{\alpha x}\psi_2^2dx}<0,
\end{equation}
then $\lambda_1(q,m_2)$ is strictly decreasing in $m_2$.
\end{proof}


Now we have the following result on the linear stability/instability of the zero steady state solution with respect to \eqref{1} when the growth rate function is of logistic or weak Allee effect type.
\begin{corollary}\label{last}
Suppose that $g(x,u)$ satisfies (g1)-(g3) and (g4a) or (g4b), $d>0$ and $b_u\geq 0$, $b_d>0$. Then for any $q\geq 0$, there exists a unique $m_2^*(q)$ satisfying $\ds\max_{x\in[0,L]}f_v(x,0)-\mu< m_2^*(q)<\ds\max_{x\in[0,L]}f_v(x,0)-\ds\frac{m_1\mu}{m_1+\sigma}$ such that $\la_1(q,m_2^*(q))=0$; the extinct state $(0,0)$ is unstable when $0<m_2<m_2^*(q)$ and it is linearly stable when $m_2>m_2^*(q)$. Moreover, if $b_d>\ds\frac{1}{2}$, then $m_2^*(q)$ is strictly decreasing in $q$. 
\end{corollary}
\begin{proof} Again we denote $p_*=\ds\max_{x\in[0,L]}f_v(x,0)$.
 From part \ref{s} of Proposition \ref{eigenp}, when $m_2>p_*-\ds\frac{m_1\mu}{m_1+\sigma}$, the zero steady state is linearly stable, and $\la_1(q,m_2)<0$. From the proof of part 3 of Proposition \ref{eigenp}, when $m_2<p_*-\mu$, $\la_1(q,m_2)>0$. From part \ref{dm2} of Proposition \ref{eigenp0}, $\lambda_1(q,m_2)$ is strictly decreasing with respect to $m_2$. Therefore there exists a unique $m^*_2(q)\in \left(p_*-\mu,p_*-\ds\frac{m_1\mu}{m_1+\sigma}\right)$ such that $\la_1(q,m^*_2(q))=0$. For all $m_2>p_*-\mu$, the continuous spectrum is always in the negative half plane. Hence the zero steady state $(0,0)$ is unstable when $0<m_2<m_2^*(q)$ and it is linearly stable when $m_2>m_2^*(q)$. From part \ref{qde} of Proposition \ref{eigenp0}, when $b_d>\ds\frac{1}{2}$, $\lambda_1(q,m_2)$ is strictly decreasing in $q$. Hence $m_2^*(q)$ is strictly decreasing in $q$ if $b_d>\ds\frac{1}{2}$.
\end{proof}
Note that in \cite{huang2016}, it is shown that the sign of the principal eigenvalue $\la_1$ or $R_0-1$ (where $R_0$ is the basic reproduction number) is the indicator of persistence or extinction for \eqref{1} in the logistic case. Hence for the logistic case considered in \cite{huang2016}, Corollary \ref{last} provides a more specific criterion of persistence or extinction for \eqref{1} in terms of advection rate $q$ and benthic population mortality rate $m_2$.

\section{Persistence/Extinction Dynamics}
In this section, we consider the dynamical behavior of system \eqref{1} with the strong Allee effect growth rate in the bethic population. Assume that $(u(x),v(x))$ is a positive  solution of system \eqref{1ss}, then from the second equation of system \eqref{1ss}, we have
\begin{equation}\label{1.1}
  u(x)=\ds\frac{A_b(x)v(x)}{A_d(x)\sigma}(\mu+m_2-g(x,v(x))),
\end{equation}
which implies that $g(x,v(x))< m_2+\mu$ for every $x\in[0,L]$. This implies that the transfer rate $\mu$ from benthos to drift zone needs to be large to ensure the existence of positive steady state solutions. Notice that we consider the following three possible scenarios: (see Fig. \ref{hh})
\begin{equation*}
 \begin{split}
 (H1)&\qquad \mu>(g_{max}-m_2)\ds\frac{\sigma+m_1}{m_1}:=\mu_1,\\
 (H2)&\qquad \mu_3:=g_{min}-m_2<\mu<(g_{min}-m_2)\ds\frac{\sigma+m_1}{m_1}:=\mu_2,\\
 (H3)&\qquad \mu<g_{min}-m_2:=\mu_3.
 \end{split}
\end{equation*}

\begin{figure}[ht]
\centering
\includegraphics[width=0.6\textwidth]{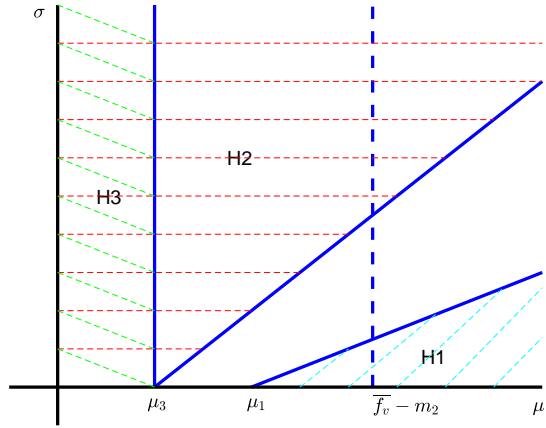}\\
\caption{Parameter regions on $\mu-\sigma$ plane satisfying (H1), (H2) or (H3).}
\label{hh}
\end{figure}

In the following, we will discuss the dynamical behavior of system \eqref{1} under $(H1)$, $(H2)$ or $(H3)$, respectively. When $(H1)$ is satisfied, we  show in subsection \ref{e} that system \eqref{1} has no positive steady state solutions, which indicates a global extinction of the population for all initial conditions. And in subsection \ref{p}, under the condition $(H3)$, we  prove the existence of multiple positive steady state solutions for any diffusion coefficient $d$ and advection rate $q$, and the persistence of population for all large initial conditions. Finally under the condition $(H2)$, which is in between $(H1)$ and $(H3)$, we show the existence of multiple positive steady state solutions in closed environment when the advection rate $q$ is small. This indicates that the extinction/persistence of benthic-drift population in the intermediate parameter range $(H2)$ is more complicated, and it depends on the movement parameters $q,d$ and also boundary conditions. Note that when $g(x,v)\equiv g(v)$ (spatially homogeneous), the conditions $(H1)$, $(H2)$ and $(H3)$ completely partitions the positive parameter quadrant $\{(\mu,\sigma):\mu,\sigma>0\}$, but there is a gap between $(H1)$ and $(H2)$ when $g(x,v)$ is spatially heterogeneous.  

\subsection{Extinction}\label{e}

First we prove the following nonexistence results of steady state solution $(u(x),v(x))$ to \eqref{1}.
\begin{theorem}\label{pro:4.3}
Suppose $g(x,u)$ satisfies (g1)-(g3) and (g4c), $d>0$ and $q,b_u,b_d\ge 0$.
\begin{enumerate}
    \item If $(H1)$ is satisfied, then the system \eqref{1} has no positive steady state solutions.
    \item The system \eqref{1} has no positive steady state solutions satisfying $v(x)<h(x)$ for all $x\in [0,L]$, where $h(x)$ is defined in (g4c).
\end{enumerate}
\end{theorem}
\begin{proof}
1. Suppose that $(u(x),v(x))$ is a positive solution of \eqref{1ss}.  Substituting \eqref{1.1} into the first equation of \eqref{1ss}, we obtain
\begin{equation}\label{3ss}
  \begin{cases}
  d u_{xx}-q u_x+\ds\frac{A_b(x)}{A_d(x)} v\left[\mu-\frac{\sigma+m_1}{\sigma}(m_2+\mu-g(x,v))\right]=0, & 0<x<L,\\
  du_{x}(0)-qu(0)=b_u qu(0),\\
  du_{x}(L)-qu(L)=-b_d qu(L).
  \end{cases}
\end{equation}
Integrating \eqref{3ss}, we get
\begin{equation}\label{h}
-b_d qu(L)-b_u qu(0)+\int_0^L\ds\frac{A_b(x)}{A_d(x)}v(x)\left[\mu -\ds\frac{\sigma+m_1}{\sigma}(\mu+m_2-g(x,v(x)))\right] dx=0.
\end{equation}
Note that $(H1)$ implies that the function
\begin{equation}\label{si}
  \tilde{v}(x)=\mu -\ds\frac{\sigma+m_1}{\sigma}(\mu+m_2-g(x,v(x)))
\end{equation}
is strictly negative. Since $v(x)>0$ and $b_u,b_d\ge 0$, we reach a contradiction with \eqref{h}. Hence there is no positive steady state solutions of \eqref{1} when $(H1)$ is satisfied.

2. Suppose that $(u(x),v(x))$ is a positive solution of \eqref{1ss}. If $0<v(x)<h(x)$, then $g(x,v(x))<0$ for all $x\in [0,L]$ and consequently $\tilde{v}(x)<\ds \mu -\ds\frac{\sigma+m_1}{\sigma}(\mu+m_2)<0$. This again leads to a contradiction.  Therefore, there is no positive solution $(u(x),v(x))$ of \eqref{1ss} satisfying $v(x)<h(x)$ for all $x\in [0,L]$.
\end{proof}

A direct corollary of Theorem \ref{pro:4.3} and Theorem \ref{ly} is the global extinction of population when the transfer rate $\mu$ of the benthic population to the drift population is too large.

\begin{corollary}\label{pp}
Suppose $g(x,u)$ satisfies (g1)-(g3) and (g4c),  $d>0$ and $q,b_u,b_d\ge 0$. If
\begin{equation}\label{psp}
    \mu>\max \left\{(g_{max}-m_2)\ds\frac{\sigma+m_1}{m_1},\overline{f_v}-m_2\right\},
\end{equation}
then for any initial condition $(u_0(x),v_0(x))\ge 0$, the solution $(u(x,t),v(x,t)$ of \eqref{1} satisfies $\ds\lim_{t\rightarrow +\infty}u(x,t)=0$ and $\ds\lim_{t\rightarrow +\infty}v(x,t)=0$.
\end{corollary}
\begin{proof}
The condition \eqref{psp} implies both $(H1)$ and \eqref{h1}. Then from  Theorem \ref{ly}, the solution converges to a nonnegative steady state as $t\to\infty$, and from Theorem \ref{pro:4.3}, the trivial steady state is the only nonnegative steady state. Therefore $\ds\lim_{t\rightarrow +\infty}u(x,t)=0$ and $\ds\lim_{t\rightarrow +\infty}v(x,t)=0$.
\end{proof}
The global extinction shown in  Corollary \ref{pp} indicates that when the the transfer rate $\mu$ of the benthic population to the drift population is too high, the benthic population becomes too low and the Allee effect drives it to extinction when the benthic population is below the threshold level.
We conjecture that the global extinction described in Corollary \ref{pp} holds when $(H1)$ is satisfied, and the condition \eqref{h1} is not necessary. But it is not known  whether the solution flow  has sufficient compactness without \eqref{h1}.

From part 5 of  Proposition \ref{eigenp}, we know that the zero steady state solution is locally asymptotically stable. In the following proposition, we describe the basin of attraction of the zero steady state solution of system \eqref{1} for different boundary conditions.

\begin{proposition}\label{0basin}
Suppose $g(x,u)$ satisfies (g1)-(g3) and (g4c),  $d>0$ and $q,b_u,b_d\ge 0$. Assume that the cross-section $A_b(x)$ and $A_d(x)$ are homogeneous. Let $(u(x,t),v(x,t))$ be the solution of \eqref{1} with initial condition $(u_0(x),v_0(x))$. Then
\begin{enumerate}
\item When $b_u\geq 0$ and $b_d\geq 0$, if $0<u_0(x)<\theta_1e^{\alpha x}\ds\min_{y\in [0,L]}e^{-\alpha y}h(y)$ and $0<v_0(x)<\ds e^{\alpha x}\min_{y\in [0,L]}e^{-\alpha y}h(y)$, then $\ds\lim_{t\rightarrow +\infty}u(x,t)=0$ and $\ds\lim_{t\rightarrow +\infty}v(x,t)=0$;
\item When $b_u\geq 0$ and $b_d\geq 1$, if $0<u_0(x)<\theta_1\ds\min_{y\in [0,L]} h(y)$ and $0<v_0(x)<\ds\min_{y\in [0,L]} h(y)$, then $\ds\lim_{t\rightarrow +\infty}u(x,t)=0$ and $\ds\lim_{t\rightarrow +\infty}v(x,t)=0$.
\end{enumerate}
\end{proposition}

\begin{proof}
1. When $b_u\geq 0$ and $b_d\geq 0$, we set $\overline{w}_1=\theta_1\ds\min_{y\in [0,L]}e^{-\alpha y}h(y)$ and $\overline{z}_1=\ds\min_{y\in [0,L]}e^{-\alpha y}h(y)$. Then we have
\begin{equation*}
d(\overline{w}_1)_{xx}+q(\overline{w}_1)_x+\ds\frac{A_b}{A_d}\mu \overline{z}_1-\sigma \overline{w}_1-m_1\overline{w}_1=0,
\end{equation*}
and
\begin{equation*}
 \begin{split}
  &\overline{z}_1g(x,e^{\alpha x}\overline{z}_1)-m_2\overline{z}_1-\mu\overline{z}_1+\ds\frac{A_d\sigma}{A_b}\overline{w}_1\\
  \leq& \overline{z}_1g(x,e^{\alpha x}\overline{z}_1)-(m_2+\mu-\ds\frac{\sigma\mu}{\sigma+m_1})\overline{z}_1
  \leq \ds\min_{y\in [0,L]}e^{-\alpha y}h(y)g(x,e^{\alpha x}\ds\min_{y\in [0,L]}e^{-\alpha y}h(y))\\
  \leq& \ds\min_{y\in [0,L]}e^{-\alpha y}h(y)g(x,e^{\alpha x}e^{-\alpha x}h(x))
  =\ds\min_{y\in[0,L]}e^{-\alpha y}h(y)g(x,h(x))=0,
 \end{split}
\end{equation*}
and the boundary conditions $-d(\overline{w}_1)_x(0)+b_u q\overline{w}_1(0)\geq 0$, $d(\overline{w}_1)_x(L)+b_d q\overline{w}_1(L)\geq 0$. Thus, $(\overline{w}_1,\overline{z}_1)$ is an upper solution of system \eqref{2ss}. Let $(\underline{w}_1,\underline{z}_1)=(0,0)$ to be the lower solution of system \eqref{2ss}.
Now assume that $0<w_0(x)<\theta_1\ds\min_{y\in [0,L]}e^{-\alpha y}h(y)$ and $0\leq z_0(x)\leq\ds\min_{y\in[0,L]}e^{-\alpha y}h(y)$, and let $(w(x,t),v(x,t))$ be the solution of \eqref{wz}. Then there exist solutions $(\overline{W}_1(x,t),\overline{Z}_1(x,t))$ and $(\underline{W}_1(x,t),\underline{Z}_1(x,t))$ of system \eqref{wz},
\begin{equation}
 \underline{W}_1(x,t)\leq w(x,t)\leq\overline{W}_1(x,t),\;\;
  \underline{Z}_1(x,t)\leq z(x,t)\leq\overline{Z}_1(x,t),
\end{equation}
where $(\overline{W}_1(x,t),\overline{Z}_1(x,t))$ and $(\underline{W}_1(x,t),\underline{Z}_1(x,t))$ are the solutions of system \eqref{wz} with the initial condition $(\overline{W}_1(x,0),\overline{Z}_1(x,0))=(\overline{w}_1,\overline{z}_1)$ and $(\underline{W}_1(x,0),\underline{Z}_1(x,0))=(\underline{w}_1,\underline{z}_1)$. Moreover,
\begin{equation}
 \begin{split}
  &\ds\lim_{t\rightarrow +\infty}(\overline{W}_1(x,t),\overline{Z}_1(x,t))=(w_{max}(x),z_{max}(x)),\\
  &\ds\lim_{t\rightarrow +\infty}(\underline{W}_1(x,t),\underline{Z}_1(x,t))=(w_{min}(x),z_{min}(x)),\\
 \end{split}
\end{equation}
where $(w_{max}(x),z_{max}(x))$, $(w_{min}(x),z_{min}(x))$ are the maximal and minimal solution of \eqref{2ss} between $(0,0)$ and $(\overline{w}_1,\overline{z}_1)$. From Proposition \ref{pro:4.3}, there is no positive solution $(u(x),v(x))$ satisfying $v(x)<h(x)$ for all $x\in[0,L]$, hence $z_{min}(x)=z_{max}(x)=0$. And consequently $w_{min}(x)=w_{max}(x)=0$. This implies that $\ds\lim_{t\rightarrow +\infty}u(x,t)=0$ and $\ds\lim_{t\rightarrow +\infty}v(x,t)=0$;.

2. When $b_u\geq 0$ and $b_d\geq 1$, we apply the upper and lower solution method directly to \eqref{1}, and we choose $(\overline{u}_1(x),\overline{v}_1(x))=(\theta_1\ds\min_{y\in[0,L]}h(y),\ds\min_{y\in[0,L]}h(y))$ to be the upper solution and $(\underline{u}_1(x),\underline{v}_1(x))=(0,0)$ be the lower solution. We can follow the same argument in the above paragraph to reach the conclusion.
\end{proof}

\subsection{Persistence}\label{p}

In this section, we provide some criteria for the population persistence of system \eqref{1} with the strong Allee effect growth rate in the benthic population.  We first show some properties of the set of positive steady state solutions of \eqref{1ss} if there exists any.

\begin{proposition}\label{maxsol}
Suppose $g(x,u)$ satisfies (g1)-(g3), $d>0$, $q,b_u,b_d\ge 0$, and the cross-section $A_b(x)$ and $A_d(x)$ are homogeneous. If there exists a positive steady state solution of \eqref{1}, then there exists a maximal steady state solution $(u_{max}(x),v_{max}(x))$ such that for any positive steady state $(u(x),v(x))$ of system \eqref{1}, $(u_{max}(x),v_{max}(x))\geq (u(x),v(x))$.
\end{proposition}
\begin{proof}
We consider the equivalent steady state equation \eqref{2ss}.
Set
\begin{equation*}
    \overline{w}_2=\theta_1\ds\max_{y\in [0,L]}e^{-\alpha y}r(y), \;\;\; \overline{z}_2=\displaystyle\max_{y\in[0,L]}e^{-\alpha y}r(y).
\end{equation*}
From (g3), we have $g_v(x,v)\leq 0$ for $v\geq r(x)$. Hence
\begin{equation*}
 g(x,e^{\alpha x}\overline{z}_2)=g(x,e^{\alpha x}\max_{y\in[0,L]}e^{-\alpha y}r(y))\leq g(x,e^{\alpha x}e^{-\alpha x}r(x))= g(x,r(x))=0.
\end{equation*}
Substituting $(\overline{w}_2, \overline{z}_2)$ into system \eqref{2ss}, we have
\begin{equation}
\begin{cases}
 d(\overline{w}_2)_{xx}+q(\overline{w}_2)_x+\ds\frac{A_b}{A_d}\mu \overline{z}_2-\sigma \overline{w}_2-m_1\overline{w}_2=0, & \qquad 0<x<L,\\
 \overline{z}_2g(x,e^{\alpha x}\overline{z}_2)-m_2\overline{z}_2-\mu\overline{z}_2+\ds\frac{A_d\sigma}{A_b}\overline{w}_2\leq 0, & \qquad 0\leq x\leq L,\\
  -d\overline{w}_2'(0)+b_u q\overline{w}_2(0)\geq 0,\\
  d\overline{w}_2'(L)+b_d q\overline{w}_2(L)\geq 0.\\
\end{cases}
\end{equation}
Thus $(\overline{w}_2, \overline{z}_2)$ is an upper solution of system \eqref{2ss}. Moreover from Proposition \ref{pro:bound}, any positive steady state solution $(w(x),z(x))$ of \eqref{2ss} satisfies $(w(x),z(x))\le (\overline{w}_2, \overline{z}_2)$. Since $(u(x),v(x))$ is a positive steady state of \eqref{1},  we can set the lower solution of \eqref{2ss} to be $(\underline{w}_2(x),\underline{z}_2(x))=(e^{-\alpha x}u(x),e^{-\alpha x}v(x))$. Then there exists a maximal solution $(w_{max}(x),z_{max}(x))$ of \eqref{2ss} satisfying $(\underline{w}_2(x),\underline{z}_2(x))\leq (w_{max}(x),z_{max}(x))$. Since $(w_{max}(x),z_{max}(x))$ is obtained through the monotone iteration process (see \cite{amann1976,cvpao}) from the upper solution $(\overline{w}_2, \overline{z}_2(x))$ and any positive steady state solution $(w(x),z(x))$ of \eqref{2ss} satisfies $(w(x),z(x))\le (\overline{w}_2, \overline{z}_2(x))$, we conclude that $(w_{max}(x),z_{max}(x))$ is the maximal steady state solution of \eqref{2ss}.
\end{proof}

Next we show a monotonicity result for the maximal steady state solution.
\begin{proposition}\label{4.7}
Suppose $g(x,u)$ satisfies (g1)-(g3), $g(x,v)\equiv g(v)$, that is $g$ is spatially homogeneous and the cross-section $A_b(x)$ and $A_d(x)$ are also homogeneous. Then if $b_u\geq 0$ and $0\leq b_d\leq1$, the maximal steady state solution $(u_{max}(x),v_{max}(x))$ of equation \eqref{1} is strictly increasing in $[0,L]$.
\end{proposition}
\begin{proof} We prove that $(u_{max})_x>0$ and $(v_{\max})_x>0$ for $x\in (0,L)$.  From \cite[Page 992]{sattinger}, the maximal solution $(u_{max},v_{max})$ is semistable, and the corresponding eigenvalue problem is \eqref{1eigen}. From Proposition \ref{eigenp}, the eigenvalue problem \eqref{1eigen} has a principal eigenvalue $\lambda_1\le 0$ with positive eigenfunction $\phi=(\phi_1,\phi_2)>0$.

We first prove that $u_{max}$ and $v_{max}$ always have the same sign for $x\in [0,L]$.
Differentiating \eqref{1ss} with respect to $x$, we have
\begin{gather}
  d(u_{max})_{xxx}-q(u_{max})_{xx}-(m_1+\sigma)(u_{max})_x+\ds\frac{A_b\mu}{A_d}(v_{max})_x=0,\label{4.4}\\
    f_v(v_{max})(v_{max})_x-(m_2+\mu)(v_{max})_x+\ds\frac{A_d\sigma}{A_b}(u_{max})_x=0,\label{4.42}
\end{gather}
where $f(v)=vg(v)$. Multiplying equation \eqref{4.42} by $\phi_2$ and multiplying the first equation in \eqref{1eigen} by $(v_{max})_x$, then subtracting, we obtain
\begin{gather}\label{m2}
\ds\frac{A_d\sigma}{A_b}\phi_2(u_{max})_x=\left(\frac{A_d\sigma}{A_b}\phi_1-\lambda_1 \phi_2\right)(v_{max})_x.
\end{gather}
Then $u_{max}$ and $v_{max}$ always have the same sign as $\phi_1>0$, $\phi_2>0$ and $\la_1\le 0$.

We prove the proposition by contradiction. Assuming that the maximal solution $(u_{max},v_{max})$ is not increasing for all $x\in [0,L]$. From boundary conditions in \eqref{1} and the condition $b_u\geq 0$, $0\leq b_d\leq1$, we have
\begin{equation*}
  \begin{split}
  (u_{max})_x(0)&=\alpha(b_u+1)u_{max}(0)>0,\\
  (u_{max})_x(L)&=\alpha(-b_d+1)u_{max}(L)\geq0.
  \end{split}
\end{equation*}
Then $(u_{max})_x(x)$ has at least two zero points in $(0,L]$. We choose the two smallest zero points $x_1, x_2\in (0,L]$ ($x_1<x_2$) such that $(u_{max})_x(x_1)=(u_{max})_x(x_2)=0$, $(u_{max})_x(x)<0$ on $(x_1,x_2)$. We claim that $(u_{max})_{xx}(x_1)<0$ and  $(u_{max})_{xx}(x_2)>0$.  Since $(u_{max})_x(x)<0$ on $(x_1,x_2)$, then $(u_{max})_{xx}(x_1)\le 0$ and $(u_{max})_{xx}(x_2)\ge 0$. If $(u_{max})_{xx}(x_1)= 0$, then from $(u_{max})_x(x_1)=0$, we conclude that $(u_{max})_x(x)\equiv 0$ near $x=x_1$ from the uniqueness of solution of ordinary differential equation, which contradicts with the assumption that $(u_{max})_x(x)<0$ on $(x_1,x_2)$. Hence we have $(u_{max})_{xx}(x_1)<0$, and similarly we can show that $(u_{max})_{xx}(x_2)>0$. Since $u_{max}$ and $v_{max}$ have the same sign, then we also have $(v_{max})_x(x)<0$  on $(x_1,x_2)$.

Multiplying equation \eqref{4.4} by $e^{-\alpha x}\phi_1$ and multiplying the first equation in \eqref{1eigen} by $e^{-\alpha x}(u_{max})_x$, then subtracting, we obtain
\begin{equation}\label{m1}
 \begin{split}
 &de^{-\alpha x}(-(u_{max})_x\phi_{1xx}+(u_{max})_{xxx}\phi_1)+qe^{-\alpha x}(-(u_{max})_{xx}\phi_1+(u_{max})_x\phi_{1x})\\
 &+e^{-\alpha x}\ds\frac{A_b\mu}{A_d}((v_{max})_x\phi_1-\phi_2(u_{max})_x)=-\lambda_1 e^{-\alpha x}\phi_1(u_{max})_x.
 \end{split}
\end{equation}
Then solving $(u_{max})_x\phi_2-(v_{max})_x\phi_1$ from \eqref{m2}, substituting  into \eqref{m1}, we have
\begin{equation}\label{4.13}
 \begin{split}
 &de^{-\alpha x}(-(u_{max})_x\phi_{1xx}+(u_{max})_{xxx}\phi_1)+qe^{-\alpha x}(-(u_{max})_{xx}\phi_1+(u_{max})_x\phi_{1x})\\
 =&-\lambda_1 e^{-\alpha x}\phi_1(u_{max})_x-\ds\frac{A_b^2\mu}{A_d^2\sigma}\lambda_1 e^{-\alpha x}\phi_2(v_{max})_x.\\
 \end{split}
\end{equation}
Integrating \eqref{4.13} on $[x_1,x_2]$, the right hand side becomes
\begin{equation}\label{4.14}
-\lambda_1\int_{x_1}^{x_2}e^{-\alpha x}\left[\phi_1(u_{max})_x+\ds\frac{A_b^2\mu}{A_d^2\sigma} e^{-\alpha x}\phi_2(v_{max})_x\right]dx< 0,
\end{equation}
as $(u_{max})_x(x)<0$ and $(v_{max})_x(x)<0$  on $(x_1,x_2)$, $\phi_1>0$, $\phi_2>0$ and $\la_1\le 0$.
On the other hand,  the left hand side becomes
\begin{equation}\label{4.15}
\begin{split}
&-d\int_{x_1}^{x_2}[(e^{-\alpha x}\phi_{1x})_x (u_{max})_x-(e^{-\alpha x}(u_{max})_{xx})_x\phi_1]dx\\
=& -de^{-\alpha x}(\phi_{1x}(u_{max})_x-\phi_1(u_{max})_{xx})\ds\Big|_{x_1}^{x_2}\\
=&-de^{-\alpha x_1}\phi_1(x_1)(u_{max})_{xx}(x_1)+de^{-\alpha x_2}\phi_1(x_2)(u_{max})_{xx}(x_2)>0,
\end{split}
\end{equation}
as $(u_{max})_{xx}(x_1)<0$ and $(u_{max})_{xx}(x_2)>0$.
So \eqref{4.14} and \eqref{4.15} are in contradiction. Thus the maximal solution $(u_{max}(x),v_{max}(x))$ of \eqref{1ss} is increasing for $x\in [0,L]$. Moreover the strong maximum principle implies that $(u,_{max}(x),v_{max}(x))$ must be strictly increasing.
\end{proof}

Next we assume that the condition $(H3)$ holds, i.e. $g_{min}> m_2+\mu$. Then for every $x\in [0,L]$, from (g3) and (g4c), there exist $v_1(x)$ and $v_2(x)$ such that $v_1(x)<v_2(x)$ and $g(x,v_i(x))=m_2+\mu$, $i=1,2$.
Moreover there also exist $v_3(x)$ and $v_4(x)$ such that  $v_3(x)<v_4(x)$, $g(x,v_i(x))=m_2+\ds\frac{m_1\mu}{\sigma+m_1}$, $i=3,4$. It is clear that $h(x)<v_3(x)<v_1(x)<v_2(x)<v_4(x)<r(x)$. When $(H2)$ is satisfied but $(H3)$ is not, $v_3(x)$ and $v_4(x)$ still exist but not $v_1(x)$ and $v_2(x)$. When $(H1)$ is satisfied, then all $v_i(x)$ ($i=1,2,3,4$) do not exist (see Fig. \ref{abc}).

\begin{figure}[ht]
\centering
\includegraphics[width=0.5\textwidth]{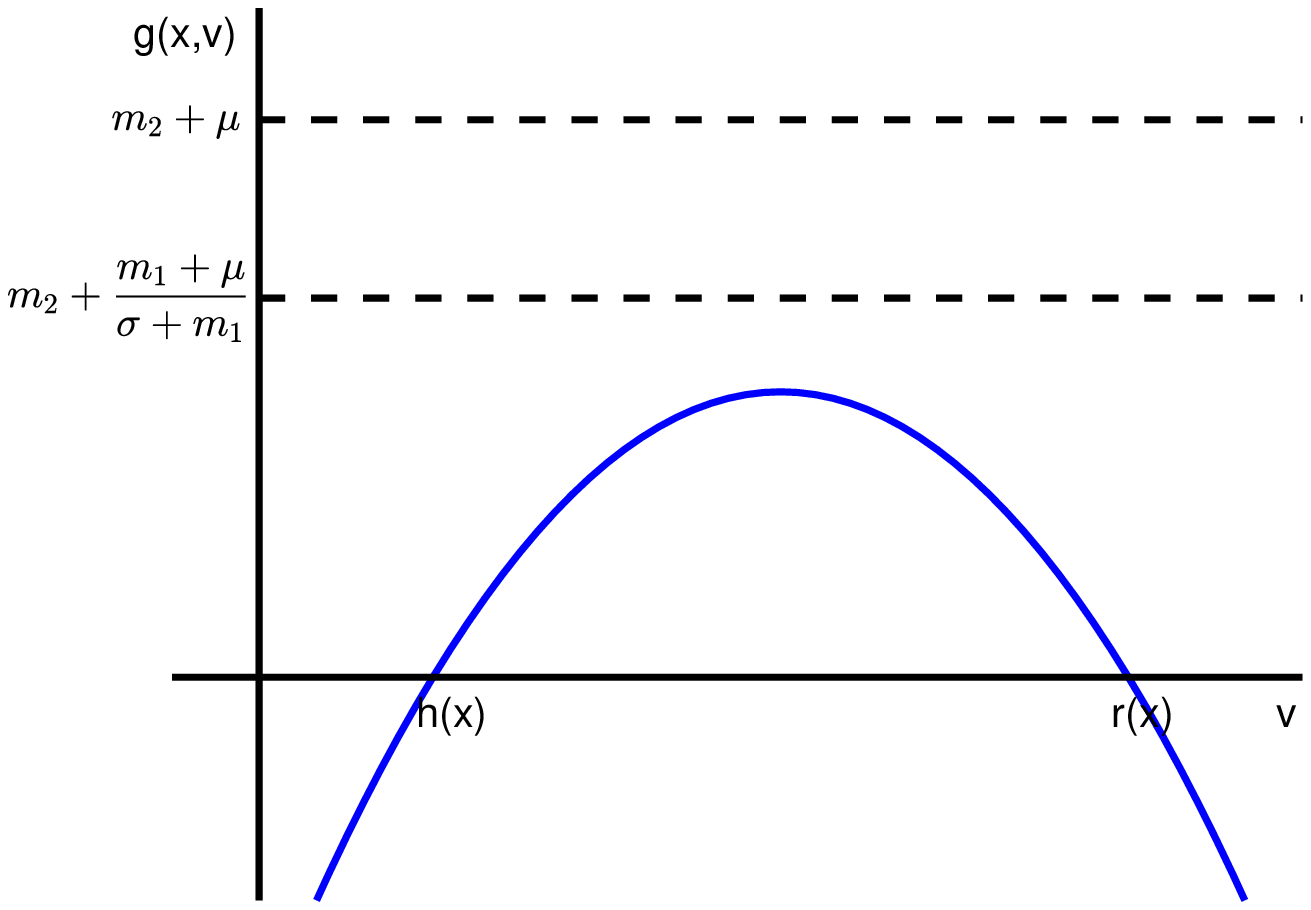}\includegraphics[width=0.5\textwidth]{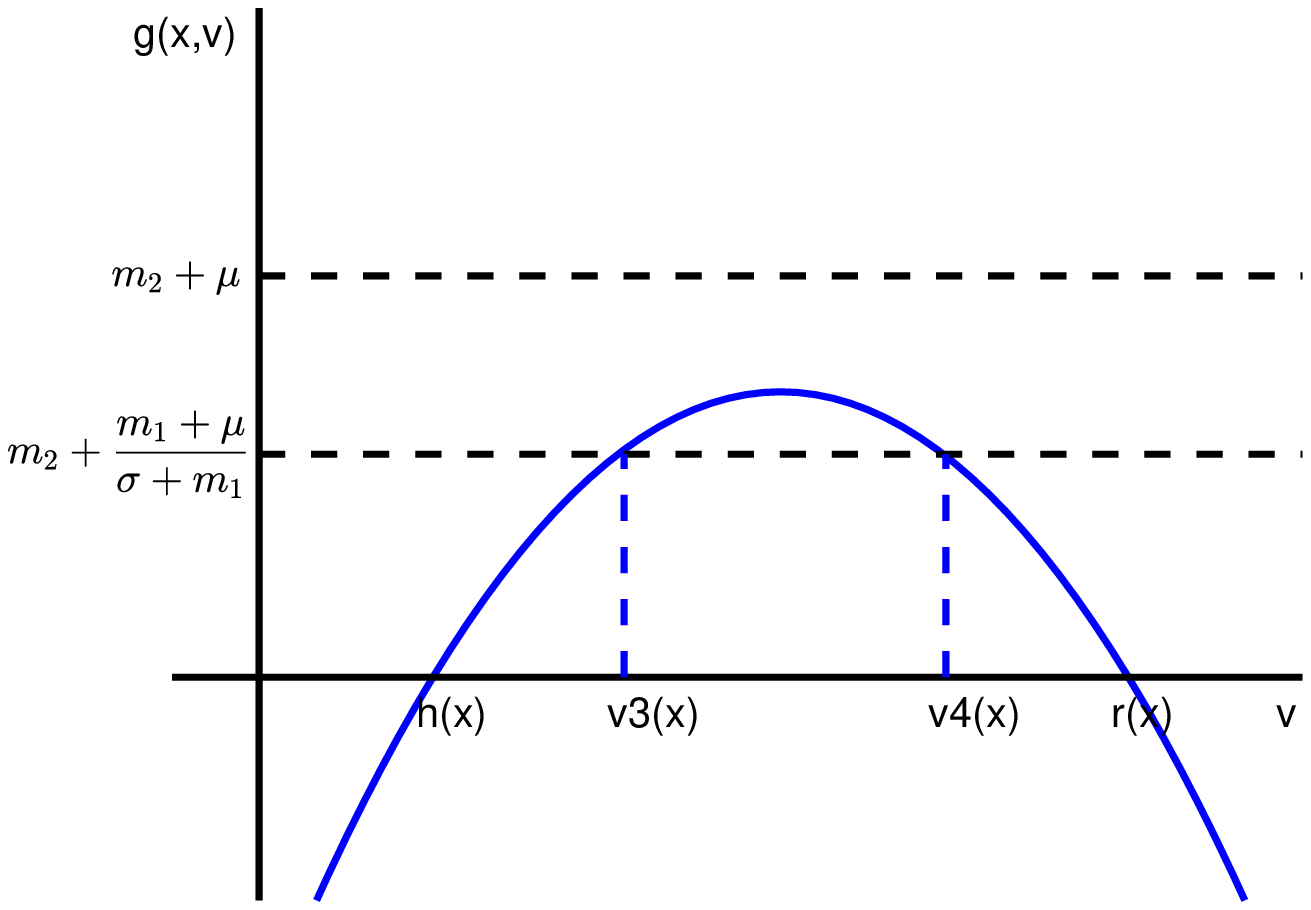}\\
\includegraphics[width=0.5\textwidth]{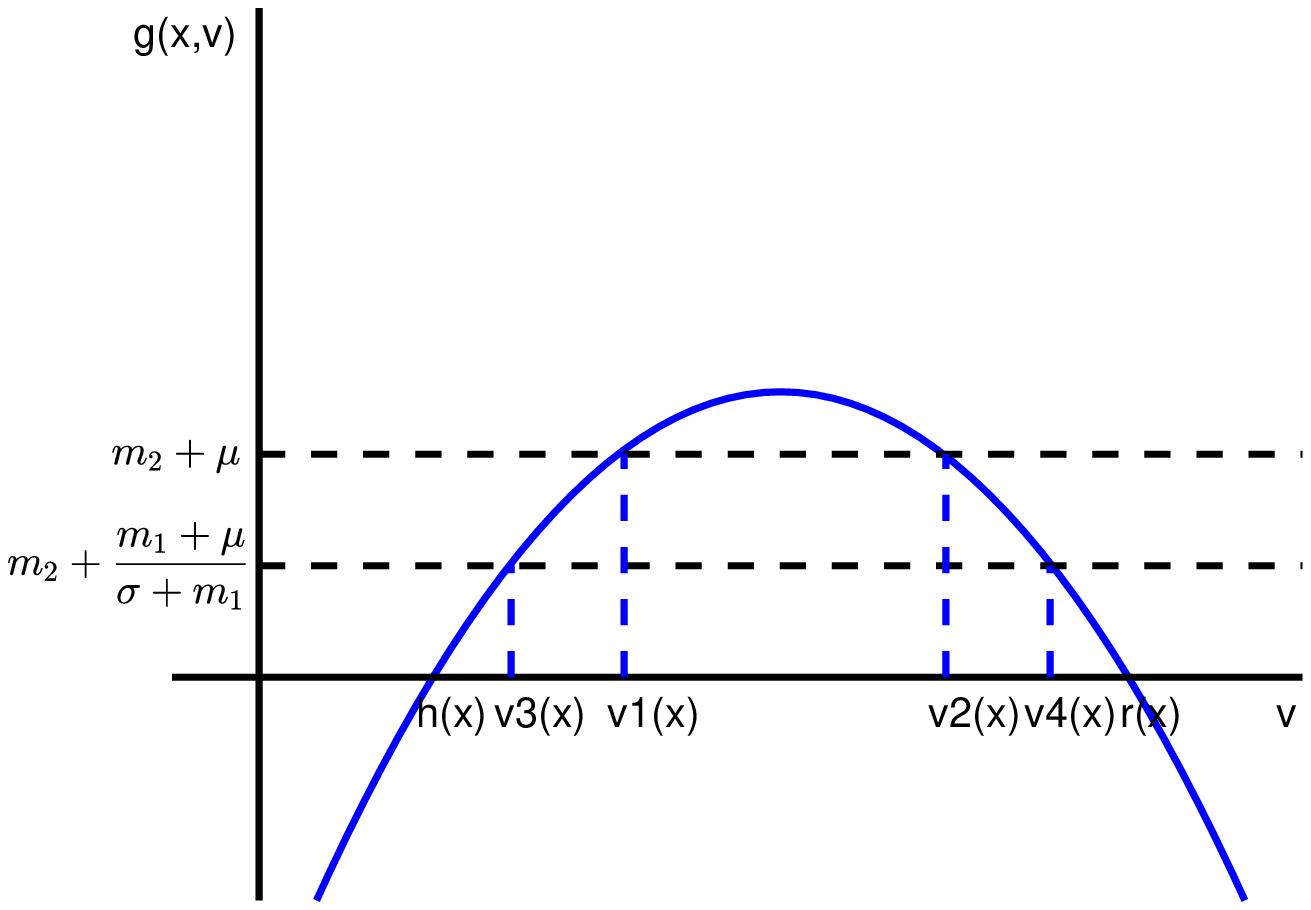}
\caption{Graph of $g(x,v)$ under conditions $(H1)$, $(H2)$ or $(H3)$. Upper left: $(H1)$; Upper right: $(H2)$; Lower: $(H3)$.}
\label{abc}
\end{figure}

We first prove the following lemma which will used to construct a lower solution of the system \eqref{2ss}.

\begin{lemma}\label{3le}
Let $p(x)\in C[0,L]$ and $d>0$, $q\geq 0$. Then the system
\begin{equation}\label{3wp}
    \begin{cases}
    dw_{xx}+qw_x+p(x)-(\sigma+m_1)w=0, & \qquad 0<x<L,\\
    -dw_x(0)+b_u qw(0)=0,\\
    dw_x(L)+b_d qw(L)=0,\\
    \end{cases}
\end{equation}
has a unique positive solution $w_p(x)$.
\end{lemma}
\begin{proof}
Consider the following eigenvalue problem
\begin{equation}\label{3wpeigen}
    \begin{cases}
    d\phi_{xx}+q\phi_x-(\sigma+m_1)\phi=\lambda\phi, & \qquad 0<x<L,\\
    -d\phi_x(0)+b_u q\phi(0)=0,\\
     d\phi_x(L)+b_d q\phi(L)=0,\\
    \end{cases}
\end{equation}
Then \eqref{3wpeigen} has a principal eigenvalue $\lambda_1$ satisfying
\begin{equation}
    -\lambda_1=\ds\inf_{\phi\in H^1(0,L),\phi\neq0}\ds\frac{\int_0^L e^{\alpha x}(d\phi_x^2+(\sigma+m_1)\phi^2)dx}{\int_0^L e^{\alpha x}\phi^2dx}
\end{equation}
Then $\lambda_1<0$ and the corresponding eigenfunction $\phi_1>0$.
We use the upper-lower solution method to prove the existence of a positive steady state solution. Let $\overline{W}(x)=\ds\frac{1}{\sigma+m_1}\ds\max_{x\in[0,L]}p(x)$, and $\underline{W}(x)=\varepsilon\phi_1(x)$ where $\varepsilon>0$ is small so that $\underline{W}(x)<\overline{W}(x)$ and $\phi_1$ is the positive eigenfunction corresponding to $\lambda_1$ of \eqref{3wpeigen}. Then it is easy to verify that $\overline{W}(x)$ and $\underline{W}(x)$ is a pair of upper-lower solution. From \cite[Theorem 4.1]{pao1996JMAA}, there exists a solution $w_p(x)$ of system \eqref{3wp} satisfying $\underline{W}(x)\leq w_p(x)\leq \overline{W}(x)$. The uniqueness follows from the maximum principle: if $w_1(x)$ and $w_2(x)$ are two solutions of system \eqref{3wp}, then $w_1(x)-w_2(x)$ satisfies a boundary value problem of linear ODE, and $w_1(x)-w_2(x)=0$ is the unique solution. Hence the solution of \eqref{3wp} is unique.
\end{proof}

Now we show that under condition $(H3)$, the benthic-drift population system is always persistent for large initial condition for any diffusion coefficient $d>0$ and advection rate $q\ge 0$, despite of strong Allee growth rate.
\begin{theorem}\label{H3}
Suppose that $g(x,v)$ satisfies (g1)-(g3) and (g4c), $d>0$ and $q\geq 0$, $b_u\geq0$ and $b_d\geq0$. Assume that $(H3)$ holds. Define
\begin{equation}
    \Sigma_1=\{(u(x),v(x))\in X_1: u(x)\geq e^{\alpha x}w_1(x), v(x)\geq v_1(x)\},
\end{equation}
where $v_1(x)$ is the smaller root of $g(x,v)=\mu+m_2$ and $w_1(x)$ is the unique positive solution of the  system
\begin{equation}\label{3w1}
    \begin{cases}
    dw_{xx}+qw_x+\ds\frac{A_b(x)\mu}{A_d(x)}e^{-\alpha x}v_1(x)-(\sigma+m_1)w=0, & \qquad 0<x<L,\\
    -dw_x(0)+b_u qw(0)=0,\\
    dw_x(L)+b_d qw(L)=0.\\
    \end{cases}
\end{equation}
Then $\Sigma_1$ is a positive invariant set for system \eqref{1}. Moreover, system \eqref{1} has a maximum steady state $(u_{max}(x),v_{max}(x))\in\Sigma_1$, and at least another positive steady state.
\end{theorem}
\begin{proof}
Assume that $(H3)$ is satisfied, we consider the equivalent system \eqref{wz} of \eqref{1}. From Lemma \ref{3le}, system \eqref{3w1} has a unique solution $w_1(x)$. We set $(\underline{w}(x),\underline{z}(x))=(w_1(x),e^{-\alpha x}v_1(x))$. Then
\begin{equation}
    \begin{cases}
    d\underline{w}_{xx}+q\underline{w}_x+\ds\frac{A_b(x)\mu}{A_d(x)}\underline{z}(x)-(\sigma+m_1)\underline{w}=0, & \qquad 0<x<L,\\
    \underline{z}g(x,e^{\alpha x}\underline{z})-(m_2+\mu)\underline{z}+\ds\frac{A_d(x)\sigma}{A_b(x)}\underline{w}\geq 0, & \qquad 0<x<L,\\
    -d\underline{w}_x(0)+b_u q\underline{w}(0)=0,\\
    d\underline{w}_x(L)+b_d q\underline{w}(L)=0.\\
    \end{cases}
\end{equation}
So $(\underline{w}(x),\underline{z}(x))$ is a lower solution of \eqref{2ss}. On the other hand, from Proposition \ref{maxsol}, $(\overline{w}(x),\overline{z}(x))=(\overline{\theta}_1\ds\max_{x\in [0,L]}e^{-\alpha x}r(x),\ds\max_{x\in[0,L]}e^{-\alpha x}r(x))$ is an upper solution of \eqref{2ss}. It is easy to check that $\overline{z}(x)=\ds\max_{x\in[0,L]}e^{-\alpha x}r(x))>e^{-\alpha x}v_1(x)=\underline{z}(x)$ and $\overline{w}(x)=\overline{\theta}_1\ds\max_{x\in [0,L]}e^{-\alpha x}r(x)>\underline{w}(x)$ from the construction of $\underline{w}(x)$ in Lemma \ref{3le}. So from \cite[Theorem 4.1]{pao1996JMAA}, there exists a positive solution $(w(x),z(x))$ of system \eqref{2ss} satisfying $\underline{w}(x)<w(x)<\overline{w}(x)$ and $\underline{z}(x)<z(x)<\overline{z}(x)$. From Proposition \ref{maxsol}, there exists a maximal solution $(u_{max}(x),v_{max}(x))\in\Sigma_1$. Since the solution of \eqref{wz} with initial condition $(\underline{w}(x),\underline{z}(x))$ is increasing, then $\Sigma_1$ is positively invariant for the dynamics of \eqref{1}. The existence of another positive steady state follows from \cite[Theorem 14.2]{amann1976} and the existence of another pair of upper-lower solutions in the proof of Proposition \ref{0basin} part 1.
\end{proof}

In the last we show that the persistence of population in the intermediate $\mu$ range (satisfying $(H2)$) may depend on the diffusion coefficient $d$ and advection rate $q$. The following result on the existence of positive steady state solutions only holds for the closed environment (NF/NF $b_u=b_d=0$) case.

\begin{theorem}\label{H2}
Suppose that $g(x,u)$ satisfies (g1)-(g3) and (g4c),  $d>0$ and $q\geq 0$. Assume that $(H2)$ holds, and the cross-section $A_b(x)$ and $A_d(x)$ are spatially homogeneous. Let $v_3(x)$ and $v_4(x)$ be the roots of $g(x,v)=m_2+\ds\frac{m_1\mu}{\sigma+m_1}$ satisfying $v_3(x)<v_4(x)$, and assume that
\begin{equation}\label{upp-sub}
\ds\max_{y\in [0,L]} e^{-\alpha y}v_3(y)<\min_{y\in [0,L]} e^{-\alpha y} v_4(y).
\end{equation}
Then when $b_u=0$ and $b_d=0$, \eqref{1} has at least two positive steady state solutions. In particular the condition \eqref{upp-sub} is satisfied if
\begin{equation}\label{upp-sub2}
0<\frac{q}{d}<\frac{1}{L}\ln\left(\frac{\min_{y\in [0,L]}  v_4(y)}{\max_{y\in [0,L]} v_3(y)}\right).
\end{equation}
\end{theorem}
\begin{proof}
Using transform $u=e^{\alpha x}w$ and $v=e^{\alpha x}z$, the steady state equation in this case is of the form
\begin{equation}\label{2ssnfnf}
\begin{cases}
  dw_{xx}+qw_x+\ds\frac{A_b}{A_d}\mu z-\sigma w-m_1 w=0, & \qquad 0<x<L,\\
  zg(x,e^{\alpha x}z)-m_2 z-\mu z+\ds\frac{A_d}{A_b}\sigma w=0, & \qquad 0\leq x\leq L,\\
  w_{x}(0)=0, \;\; w_{x}(L)=0.\\
\end{cases}
\end{equation}
From Proposition \ref{maxsol}, $(\overline{w}_2,\overline{z}_2)=(\theta_1\ds\max_{y\in [0,L]}e^{-\alpha y}r(y),\displaystyle\max_{y\in[0,L]}e^{-\alpha y}r(y))$ is an upper solution of \eqref{2ssnfnf}. Set
\begin{equation*}
    \underline{w}_2=\theta_1\ds\max_{y\in [0,L]} e^{-\alpha y}v_3(y),\;\;
\underline{z}_2=\ds\max_{y\in [0,L]} e^{-\alpha y}v_3(y).
\end{equation*}
Then from \eqref{upp-sub},
\begin{equation*}
\begin{split}
 g(x,e^{\alpha x}\underline{z}_2)=&g(x,e^{\alpha x}\ds\max_{y\in [0,L]} e^{-\alpha y}v_3(y))\\
 \geq& g(x,e^{\alpha x}e^{-\alpha x}v_3(x))= g(x,v_3(x))=m_2+\mu-\ds\frac{\mu\sigma}{\sigma+m_1},
 \end{split}
\end{equation*}
we obtain that
\begin{equation*}
\begin{cases}
  d\underline{w}_2''+q\underline{w}_2'+\ds\frac{A_b}{A_d}\mu \underline{z}_2-\sigma \underline{w}_2-m_1 \underline{w}_2=0, & \qquad 0<x<L,\\
  \underline{z}_2g(x,e^{\alpha x}\underline{z}_2)-m_2 \underline{z}_2-\mu \underline{z}_2+\ds\frac{A_d}{A_b}\sigma \underline{w}_2\geq0, & \qquad 0\leq x\leq L,\\
  \underline{w}_2'(0)=0, \;\; \underline{w}_2'(L)=0.\\
\end{cases}
\end{equation*}
So $(\underline{w}_2,\underline{z}_2)$ is a lower solution of \eqref{2ssnfnf}, and we have $\underline{w}_2<\overline{w}_2$, $\underline{z}_2<\overline{z}_2$. Therefore \eqref{2ssnfnf} has at least one positive solution between $(\underline{w}_2,\underline{z}_2)$ and $(\overline{w}_2,\overline{z}_2)$. Moreover $(\underline{w}_1,\underline{z}_1)=(0,0)$ is a lower solution of \eqref{2ssnfnf}, and from Proposition \ref{0basin}, $(\overline{w}_1,\overline{z}_1)$ is an upper solution of \eqref{2ssnfnf}, hence we have two pairs of upper and lower solutions which satisfy
\begin{equation*}
  (\underline{w}_1,\underline{z}_1)<(\overline{w}_1,\overline{z}_1)<(\underline{w}_2,\underline{z}_2)<(\overline{w}_2,\overline{z}_2).
\end{equation*}
From \cite[Theorem 14.2]{amann1976}, \eqref{2ssnfnf} has at least three nonnegative solutions, which implies that there exist at least two positive solutions. The condition \eqref{upp-sub2} can be derived from \eqref{upp-sub} since
\begin{equation*}
\ds\max_{y\in [0,L]} e^{-\alpha y}v_3(y)\le \ds\max_{y\in [0,L]}v_3(y),  \;\; e^{-\alpha L} \min_{y\in [0,L]} v_4(y)\le \min_{y\in [0,L]} e^{-\alpha y} v_4(y).
\end{equation*}
\end{proof}

\section{Numerical Simulations}

In this section we show some numerical simulation results to demonstrate our theoritical results proved above and also provide some further quantitative information on the dynamical behavior of the system \eqref{1}. In particular we show the effect of the tranfer rate $\mu$ and advection $q$ on the maximal steady states. In this section we always assume that
\begin{equation}\label{434}
\begin{split}
    &f(x,v)=vg(x,v)=v(1-v)(v-0.4), \;\; d=0.02, \;\; L=10, \;\; \\
    & m_1=m_2=0.02, \;\; \sigma=0.2, \;\; A_d(x)=A_b(x)=1.
    \end{split}
\end{equation}
and we consider the special case of \eqref{1}:
\begin{equation}\label{3s}
\begin{cases}
  u_t=d u_{xx}-q u_x+\ds\frac{A_b}{A_d}\mu v-m_1 u, & \qquad 0<x<L, \; t>0,\\
  v_t=v(1-v)(v-h)-m_2 v-\mu v+\ds\frac{A_d}{A_b}\sigma u, & \qquad 0\leq x\leq L, \; t>0,\\
  du_{x}(0,t)-qu(0,t)=0, & \qquad t>0,\\
  du_{x}(L,t)-qu(L,t)=-b_d qu(L,t), & \qquad t>0,\\
  u(x,0)=u_0(x)\ge 0, \; v(x,0)=v_0(x)\ge 0, & \qquad x\in (0,L).
\end{cases}
\end{equation}

\begin{figure}[ht]
\centering
\includegraphics[width=0.5\textwidth]{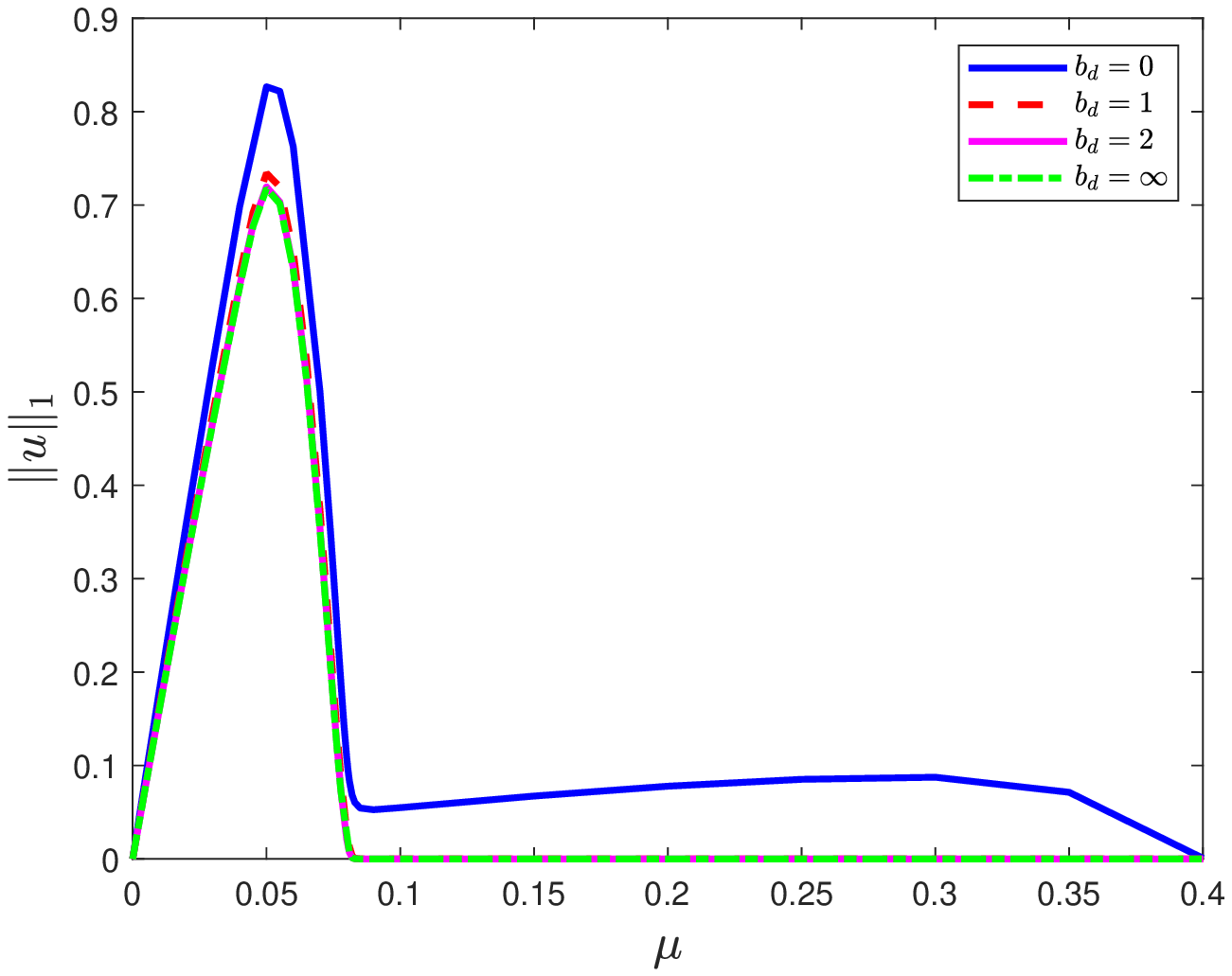}\includegraphics[width=0.5\textwidth]{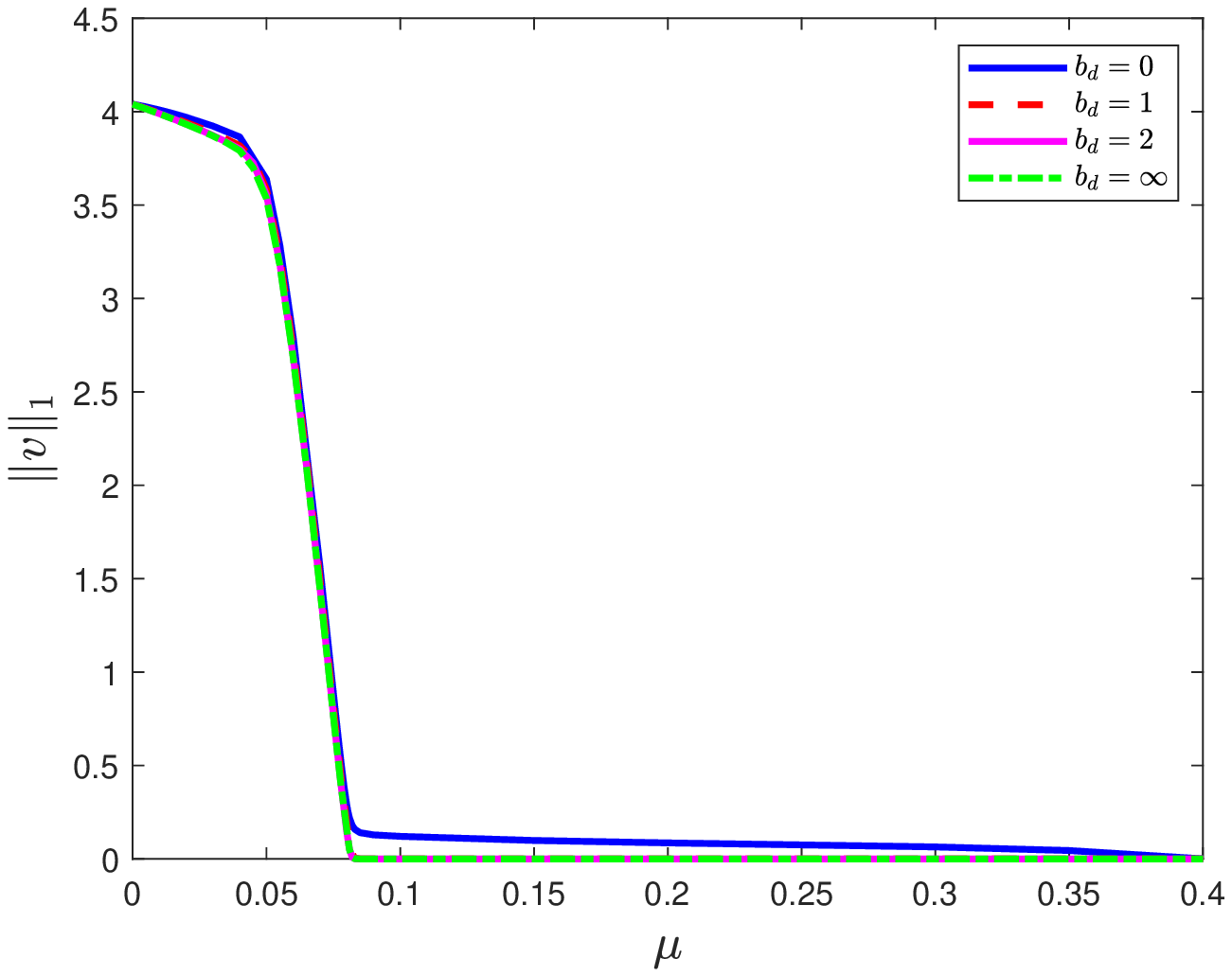}\\
\caption{The total biomass of the maximal steady state solution of \eqref{3s}  with respect to the transfer rate $\mu$ under different boundary conditions. The horizontal axis is $\mu$, the vertical axis are $\|u(\cdot,t)\|_1=\int_0^L u(x,t)dx$ (Left) and $\|v(\cdot,t)\|_1=\int_0^L v(x,t)dx$ (Right). Here the parameters satisfy \eqref{434},  $q=0.2$ and the initial condition is $u_0=0.2$, $v_0=0.2$.}
\label{3norm1mu}
\end{figure}

Figure \ref{3norm1mu} shows the variation of total biomass of the maximal steady states for different $b_d$ and varying transfer rate $\mu$ from the benthic zone to the drift zone. It can be observed from the right panel that The total biomass of the benthic population is always decreasing with respect to $\mu$ since it can be regarded as a loss of the benthic population. When $\mu=0$, the drift population does not have the source of growing and it cannot live. At the lower $\mu$ level, with the increase of transfer rate $\mu$, the drift population becomes larger, but  after an optimal intermediate $\mu$ value (about $\mu^*\approx 0.05$), the drift population starts to decline with respect to $\mu$ for the drift population. We can calculate the two threshold values  $\mu_1=\mu_2=0.77$ and $\mu_3=0.07$, defined in the conditions $(H1)$-$(H3)$. One can observe that in $(H3)$ regime ($0<\mu<\mu_3$), the population persists robustly for all boundary  conditions (see Theorem \ref{H3}); and in $(H2)$ regime ($\mu_3<\mu<\mu_2$), the biomass is nearly zero for open environment and is larger than zero for closed environment (see Theorem \ref{H2} for a partial justification).
Figure \ref{3norm1mu} only shows the biomass up to $\mu=0.4$, and for $0.4<\mu<\mu_2$, the biomass for even the NF/NF boundary condition becomes so small which cannot be distinguished from zero. For $\mu>\mu_2$ ($(H1)$ regime), the extinction of population is ensured in Theorem \ref{pro:4.3}.

In Figure \ref{3mu}, the maximal steady state solutions under three types boundary conditions (NF/H, NF/FF, NF/FF) and for varying transfer rate $\mu$ are plotted. For all boundary conditions, the benthic population is decreasing in $\mu$. The drift population is increasing in $\mu$ for $0<\mu<\mu^*$ ($\mu^*$ is the peak transfer rate where the drift biomass reachs the maximum), and for $\mu^*<\mu<\mu_3$, the drift population is decreasing in downstream part but increasing in upstream part. Similarly in Figure \ref{3q}, the maximal steady state solutions under three types boundary conditions (NF/H, NF/FF, NF/FF) and for varying advection rate $q$. One can observe that a larger advection rate leads to a larger benthic population for every point in the river, but for drift population, a larger advection rate decreases the downstream population and increases the upstream population.

\begin{figure}[htb!]
\centering
\includegraphics[width=0.5\textwidth]{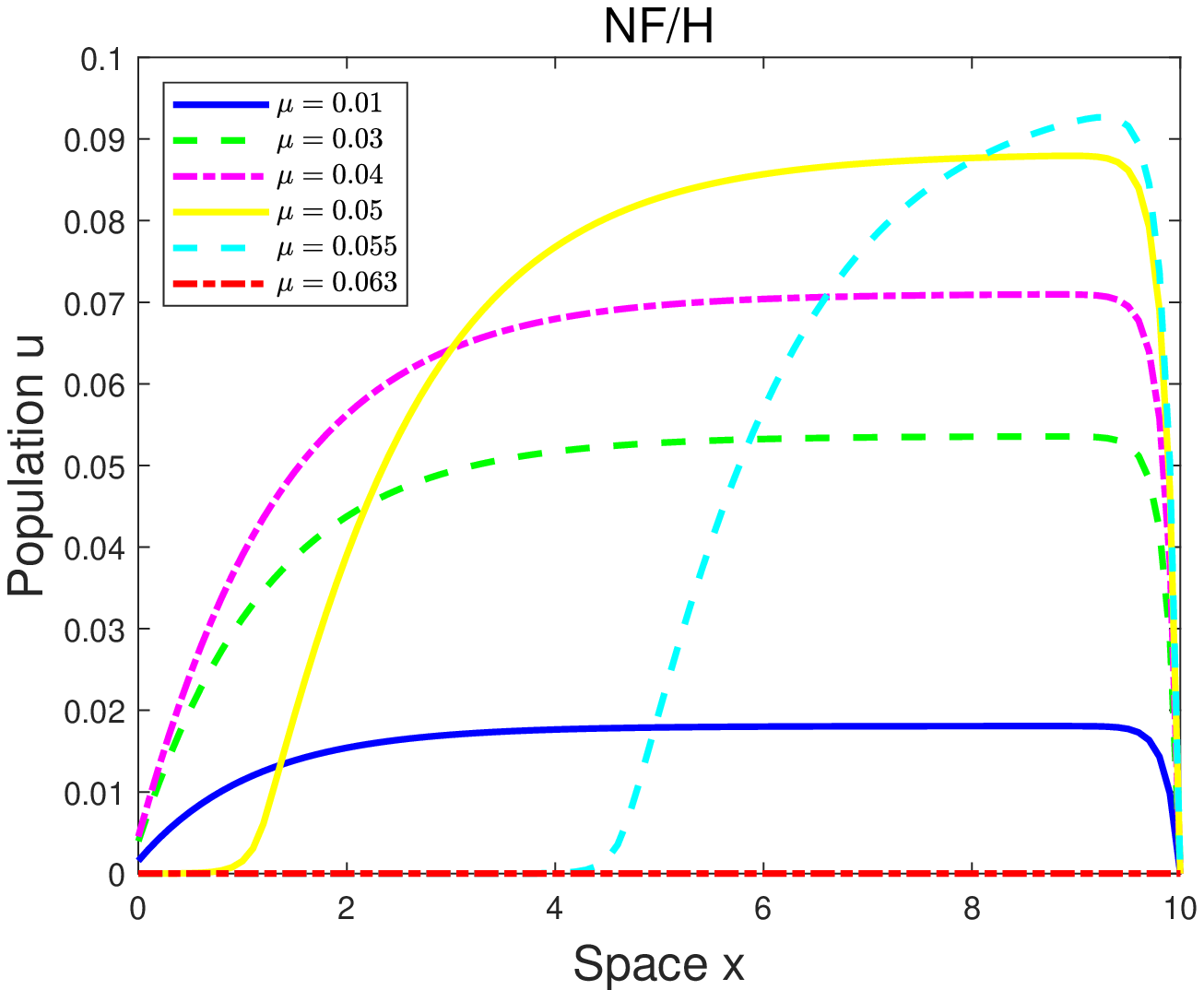}\includegraphics[width=0.5\textwidth]{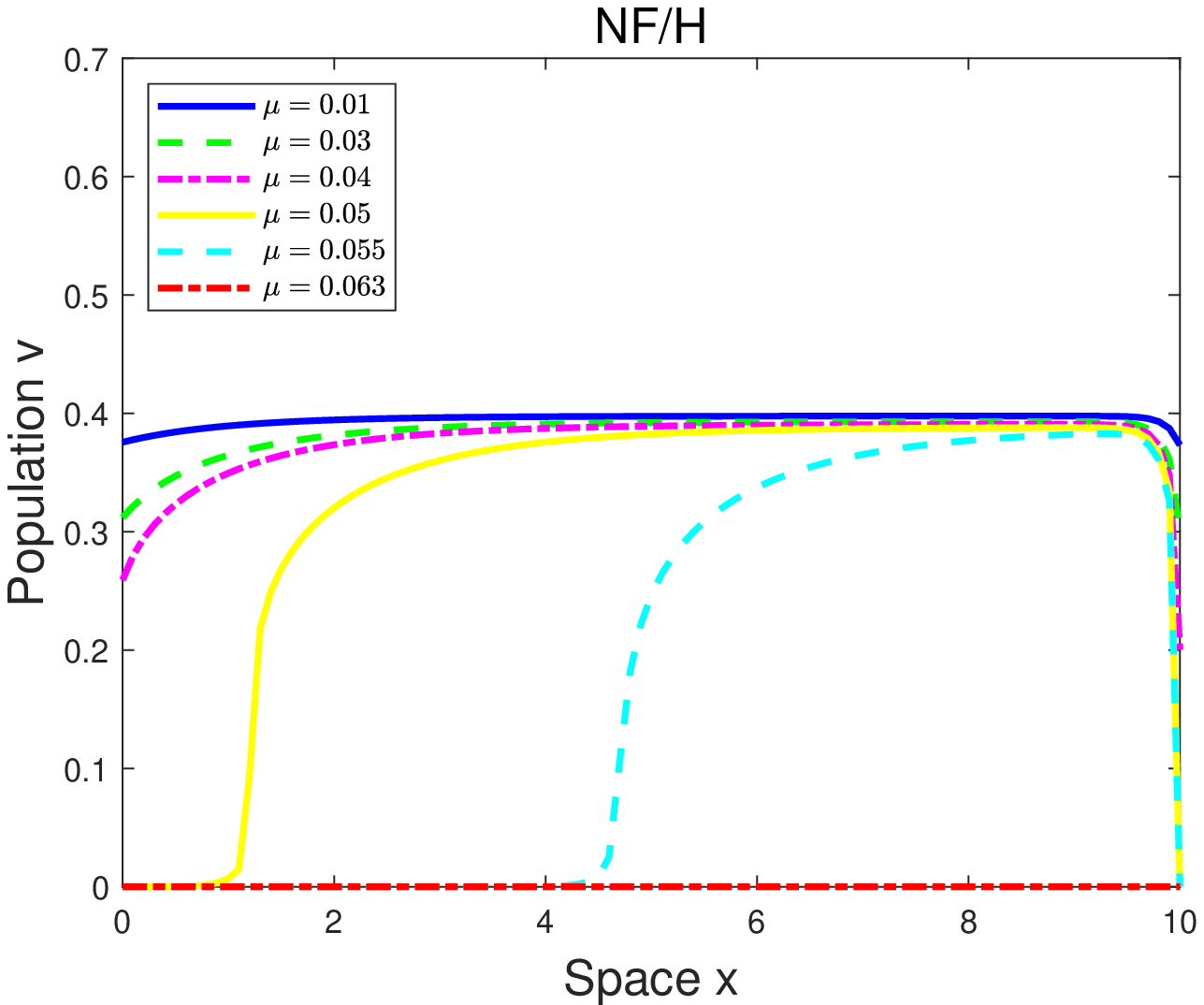}\\
\includegraphics[width=0.5\textwidth]{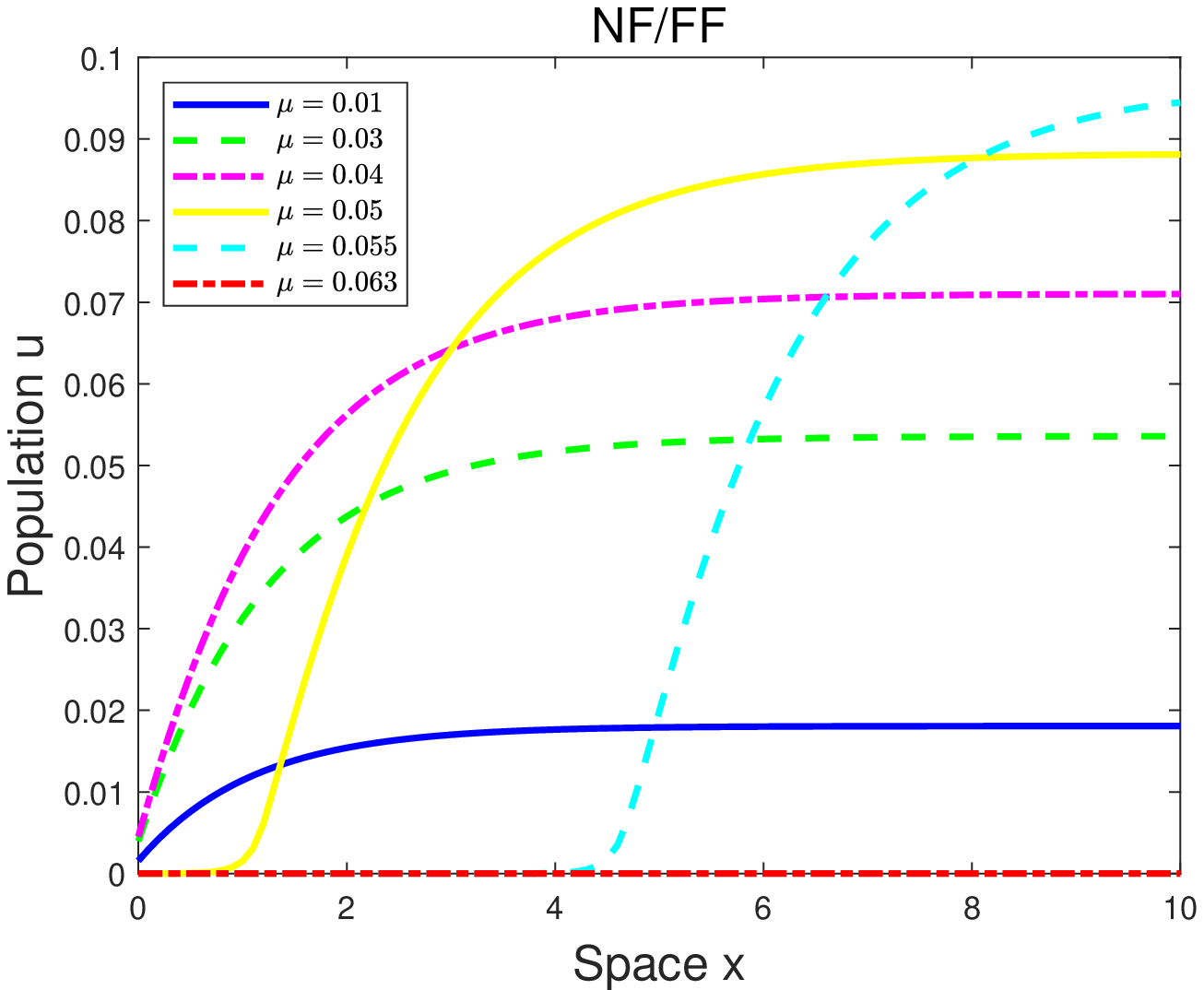}\includegraphics[width=0.5\textwidth]{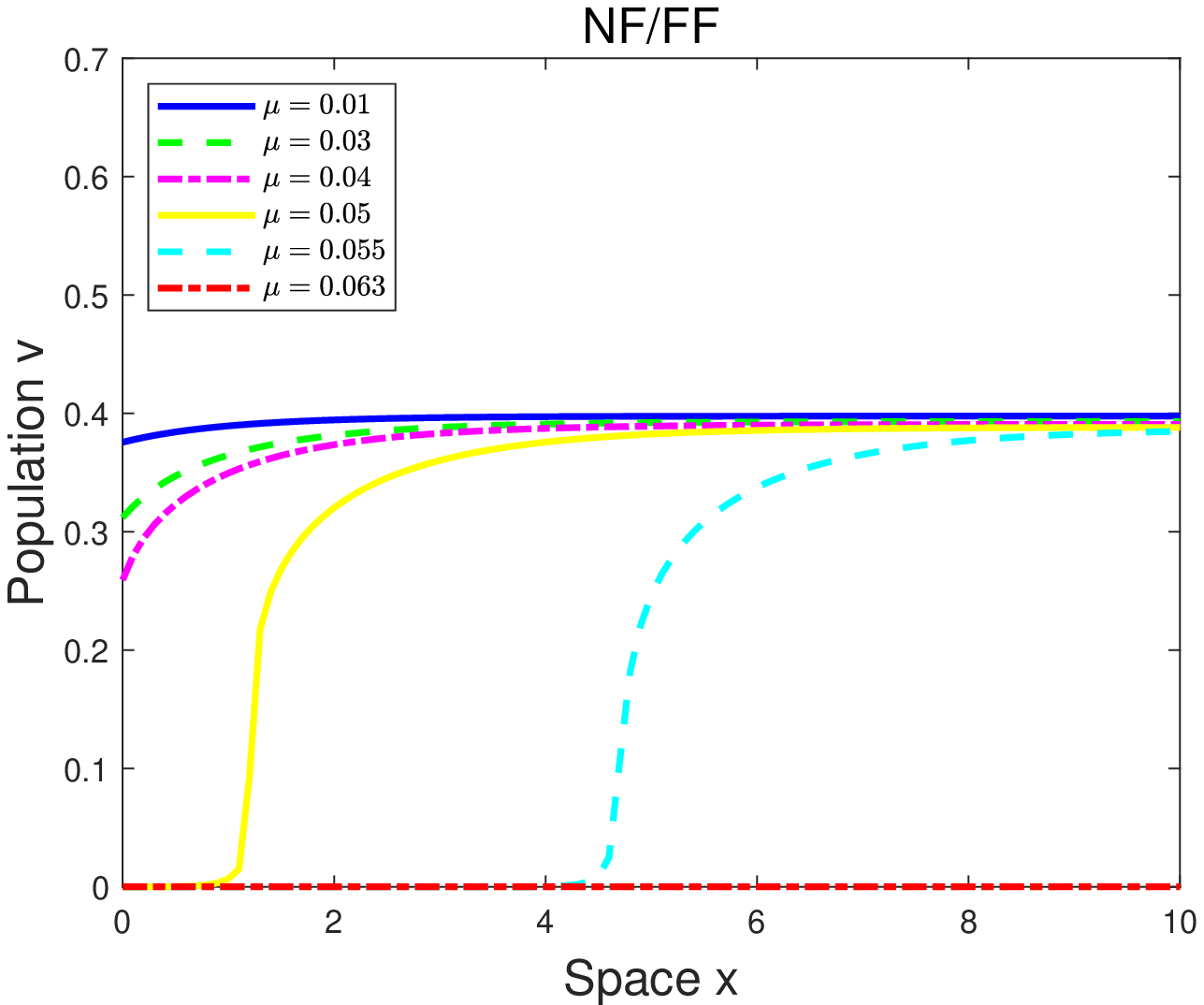}\\
\includegraphics[width=0.5\textwidth]{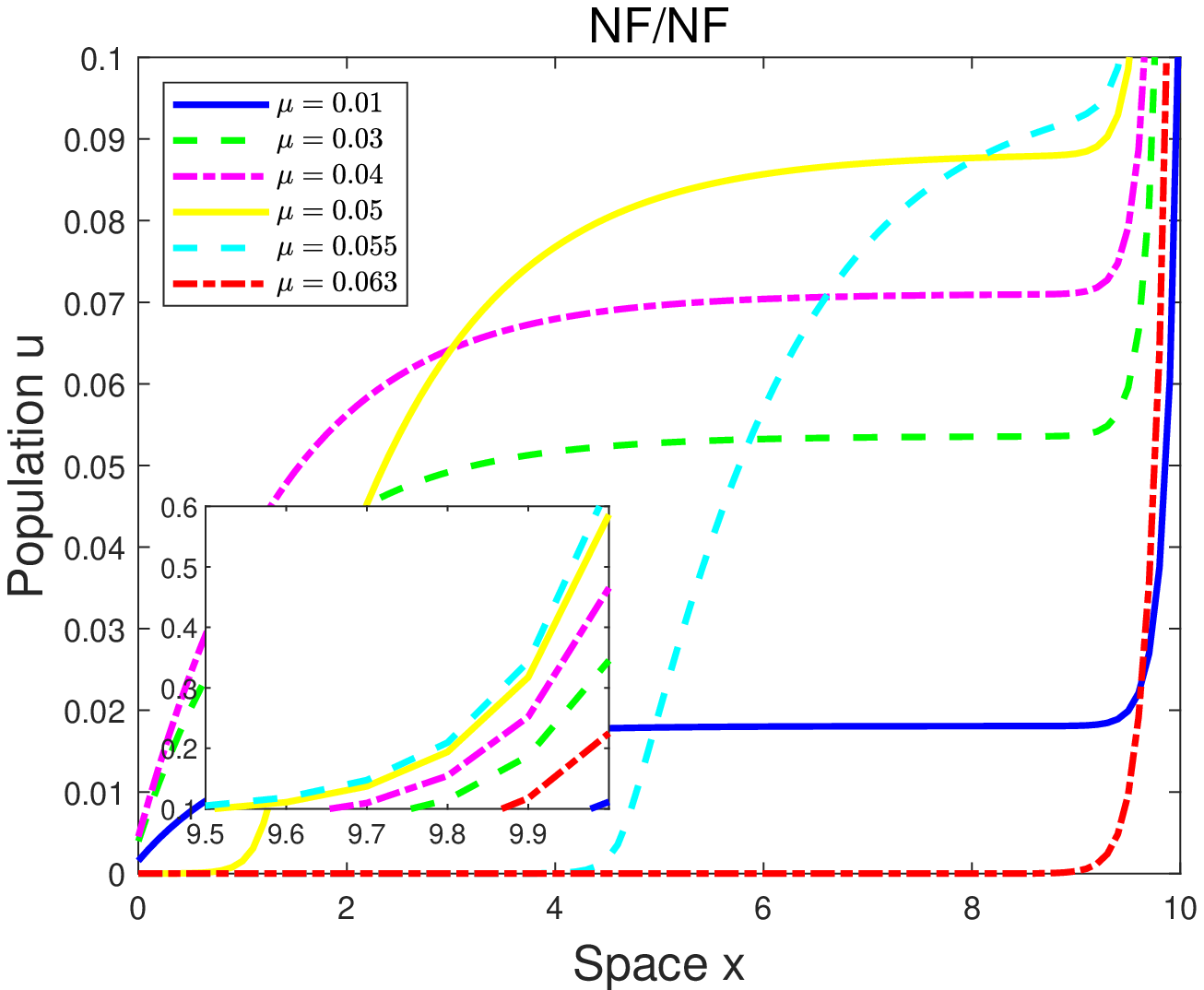}\includegraphics[width=0.5\textwidth]{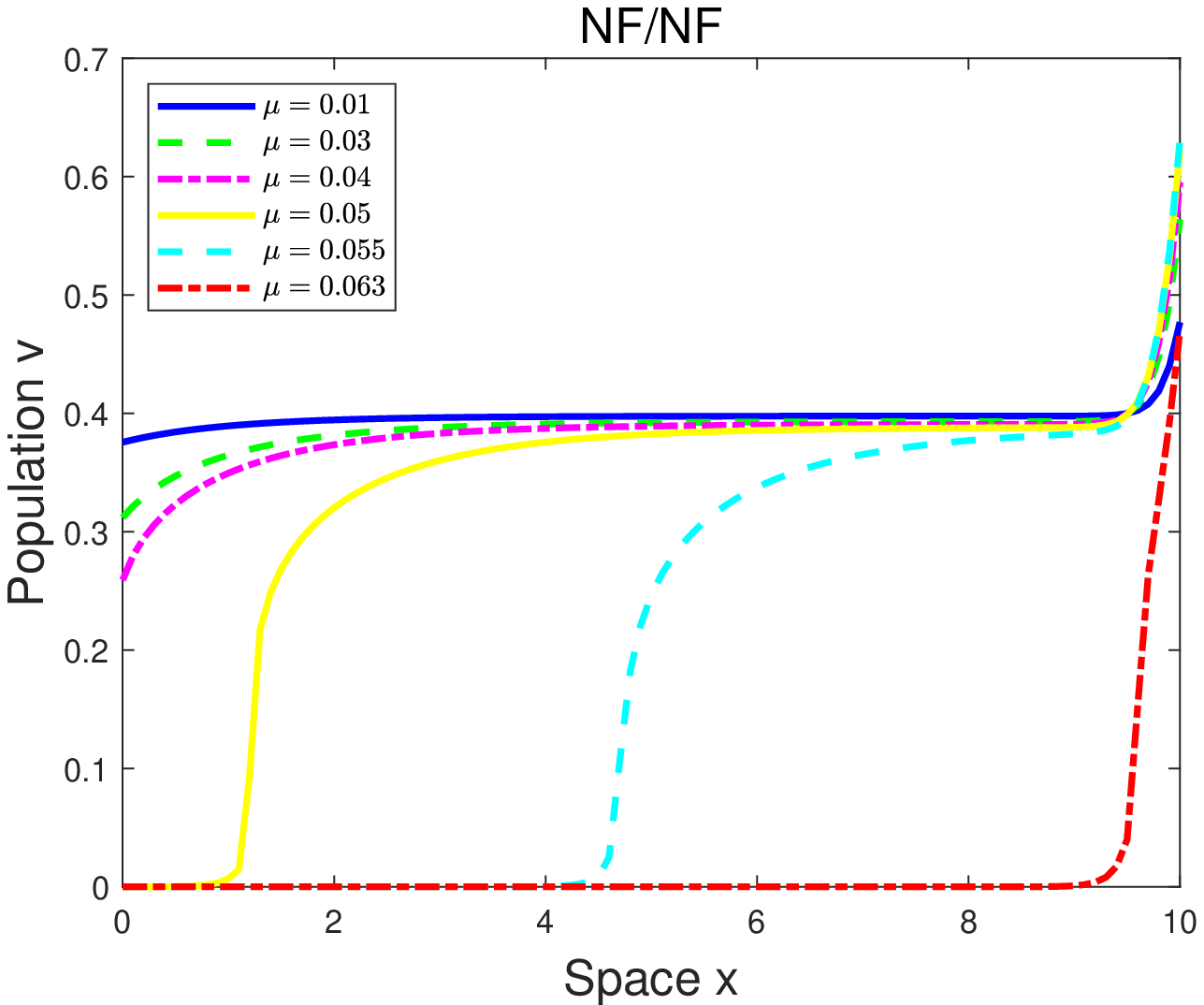}\\
\caption{The dependence of maximal steady state solution of \eqref{3s} on the varying transfer rate $\mu$ under three types boundary conditions. Here  the parameters satisfy \eqref{434},  $q=0.2$,  and the initial condition is $u_0=0.2$, $v_0=0.2$. Left: The drift population; Right: The benthic population.}
\label{3mu}
\end{figure}

\begin{figure}[htb!]
\centering
\includegraphics[width=0.5\textwidth]{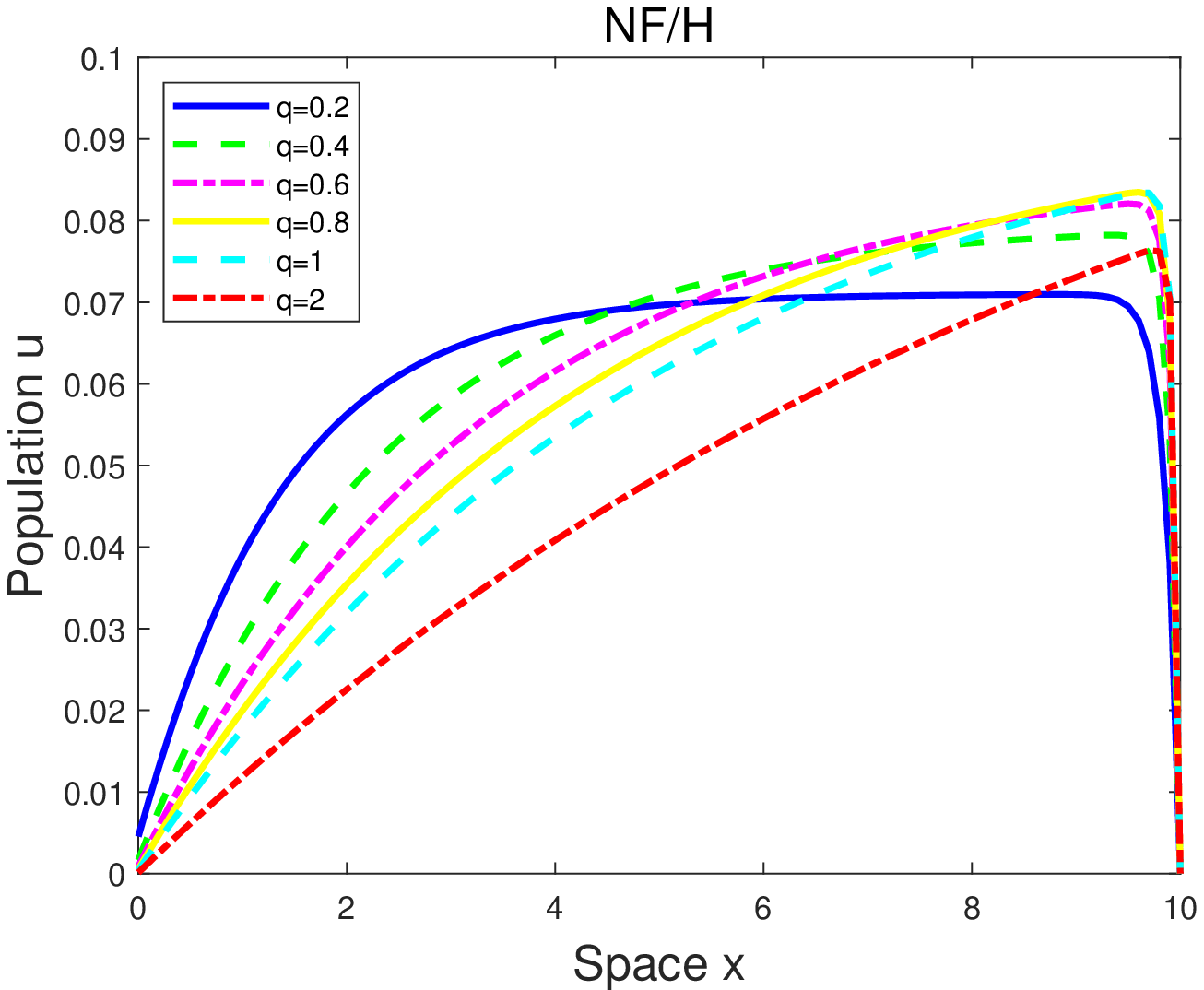}\includegraphics[width=0.5\textwidth]{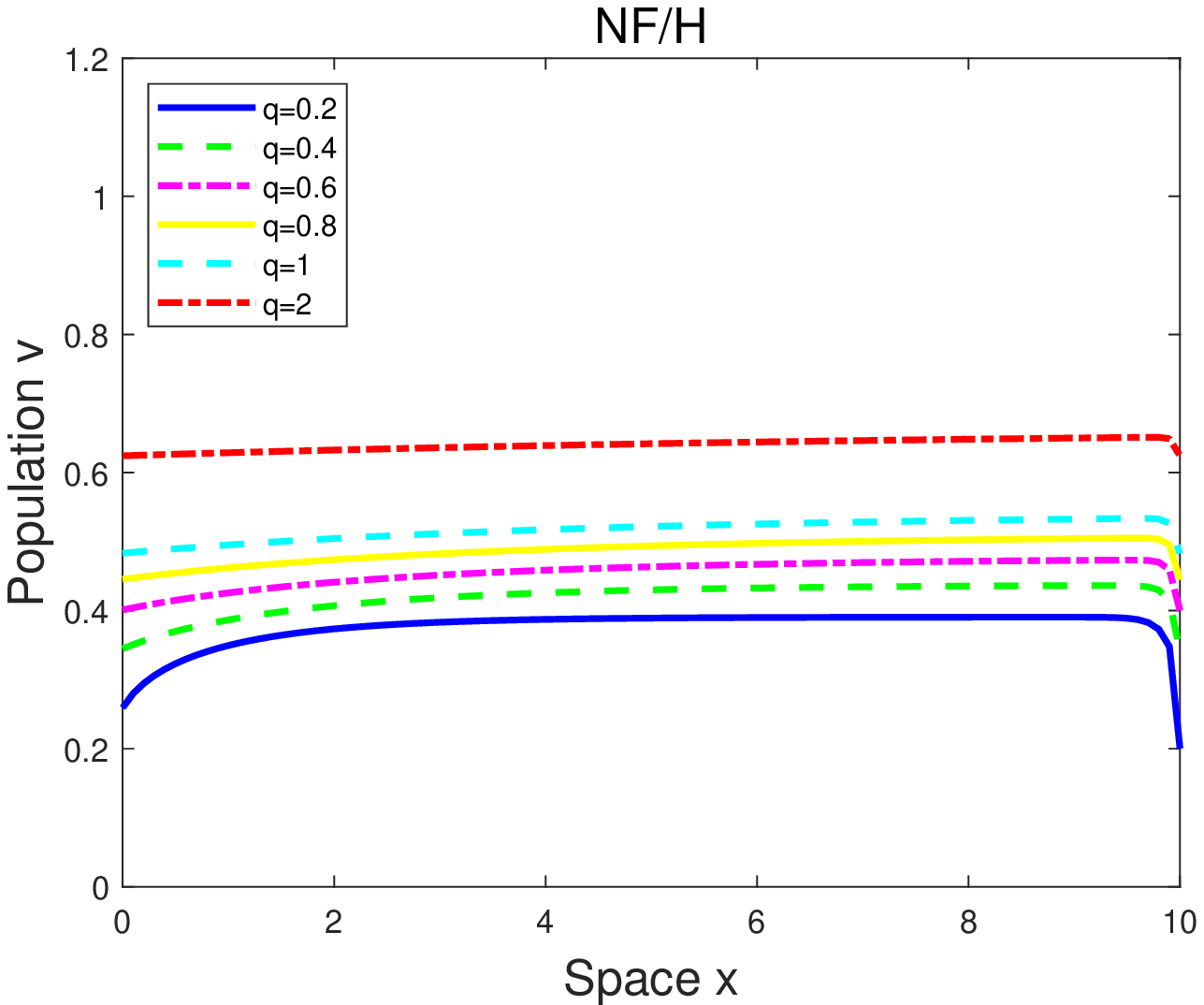}\\
\includegraphics[width=0.5\textwidth]{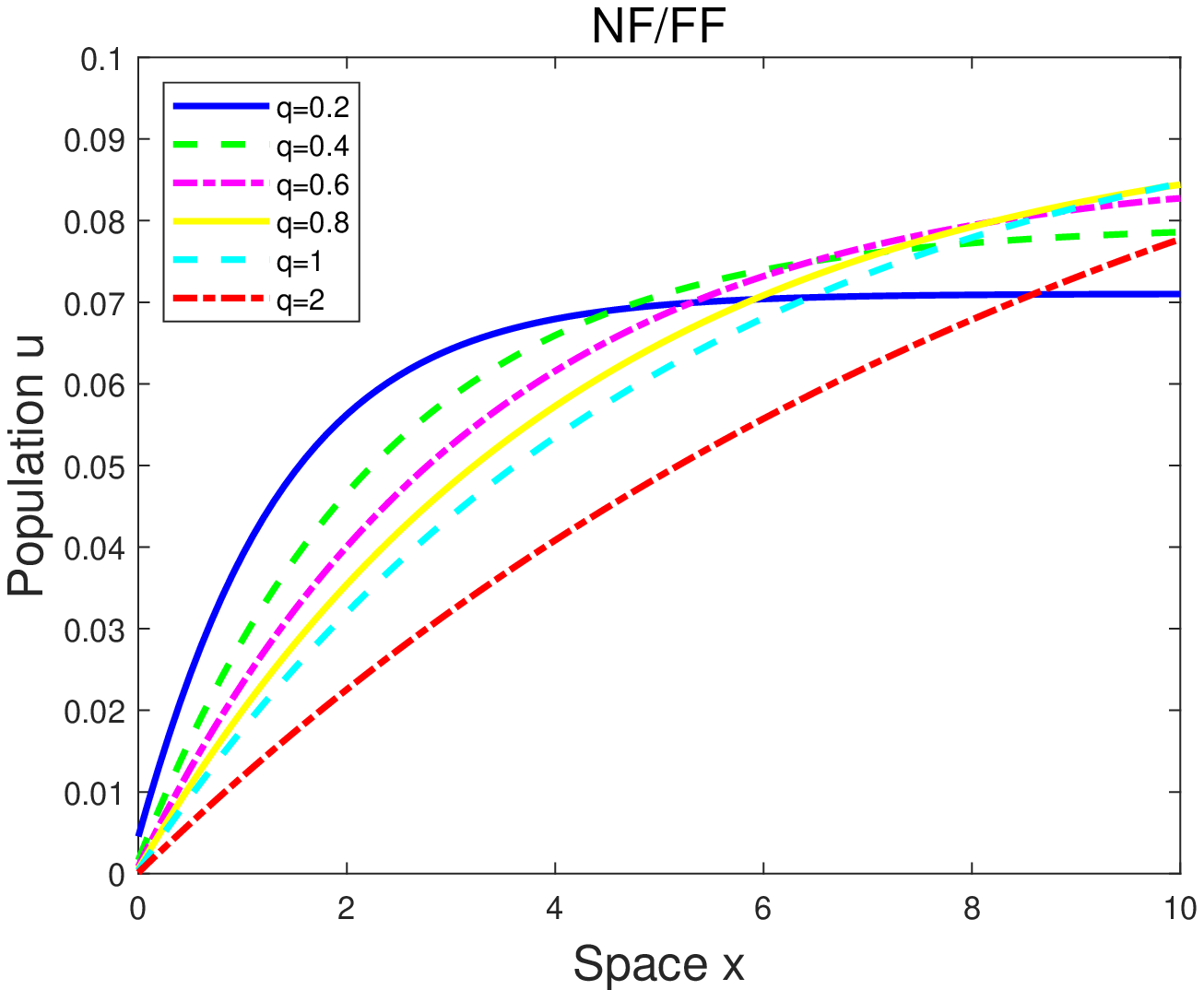}\includegraphics[width=0.5\textwidth]{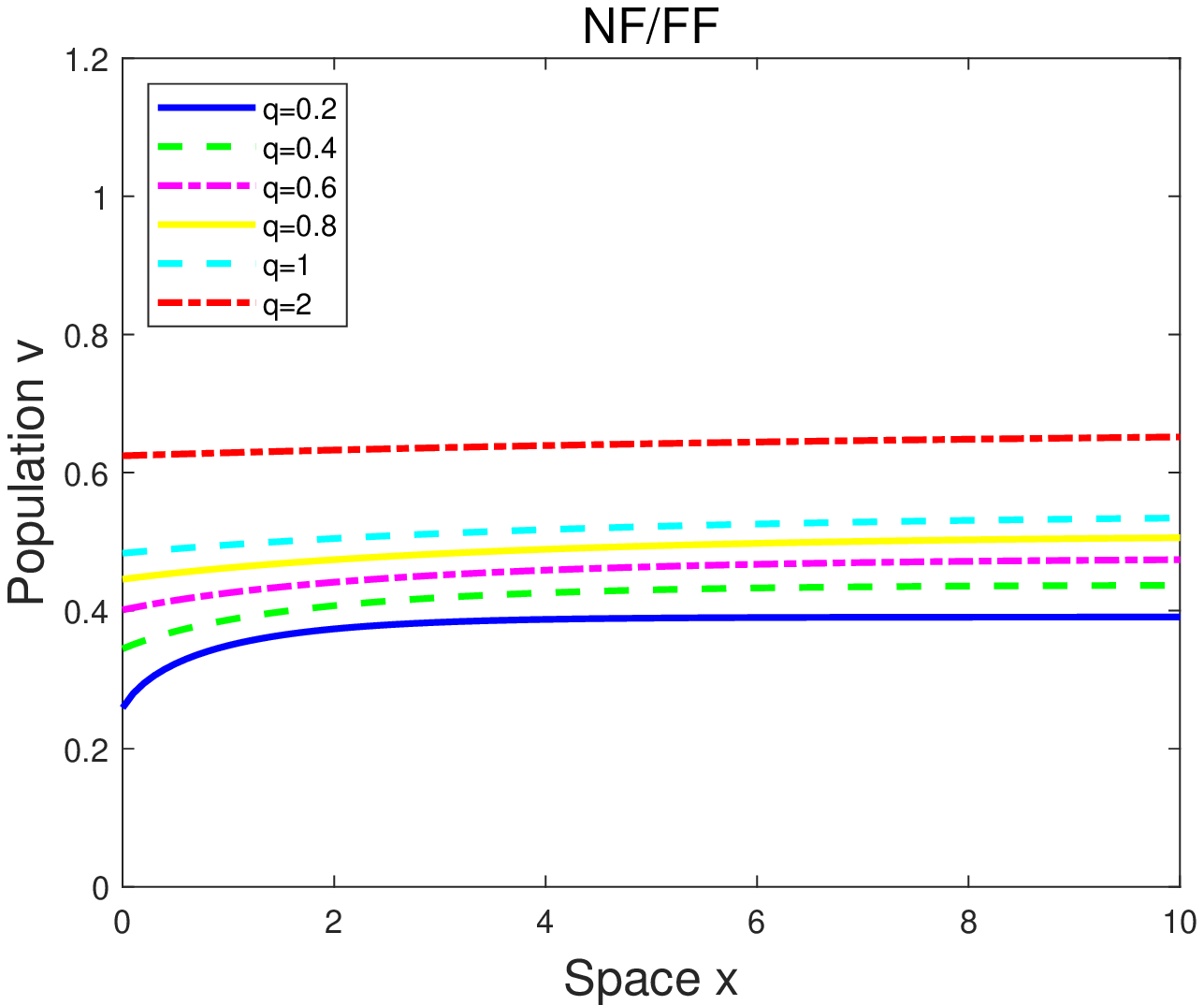}\\
\includegraphics[width=0.5\textwidth]{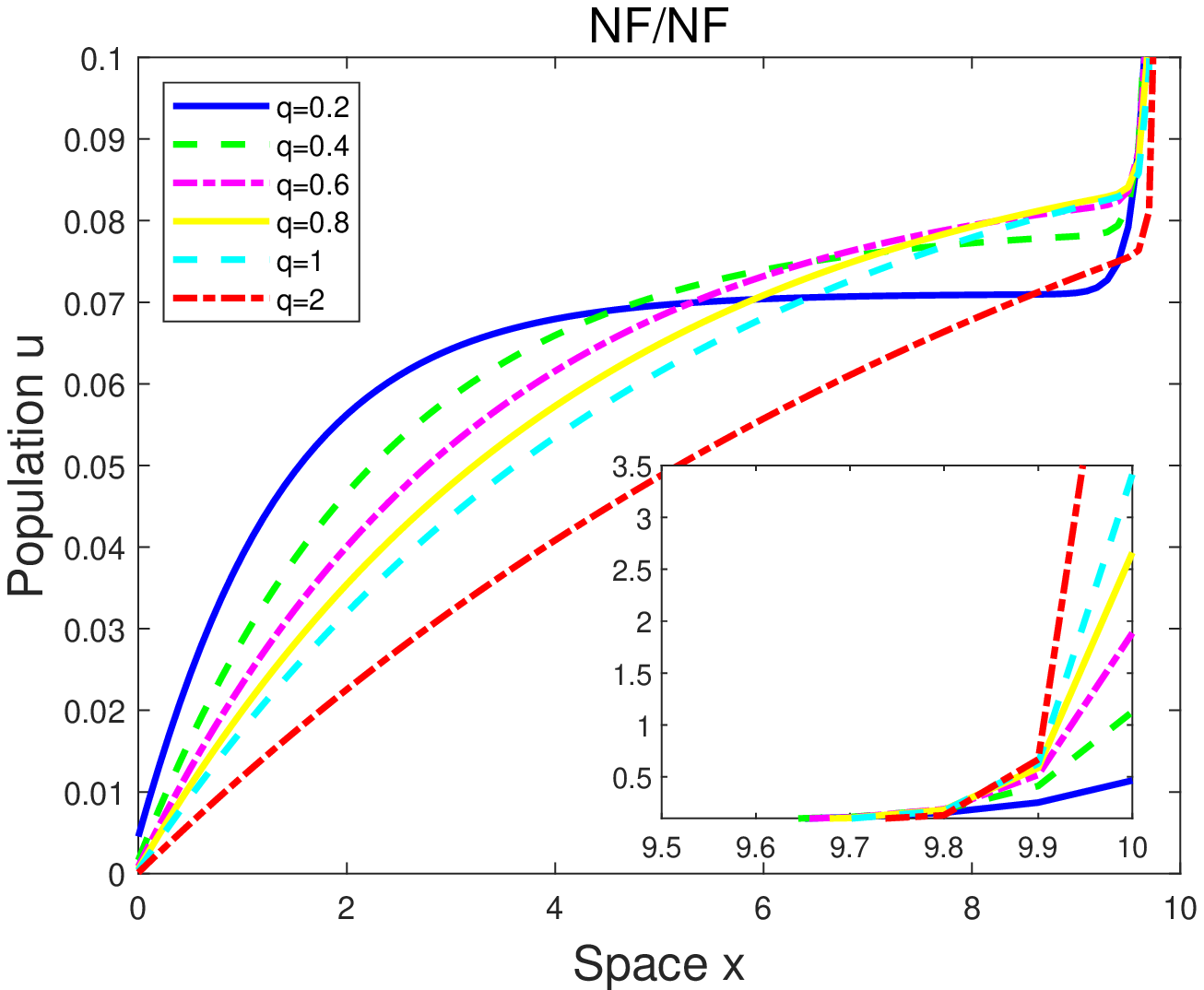}\includegraphics[width=0.5\textwidth]{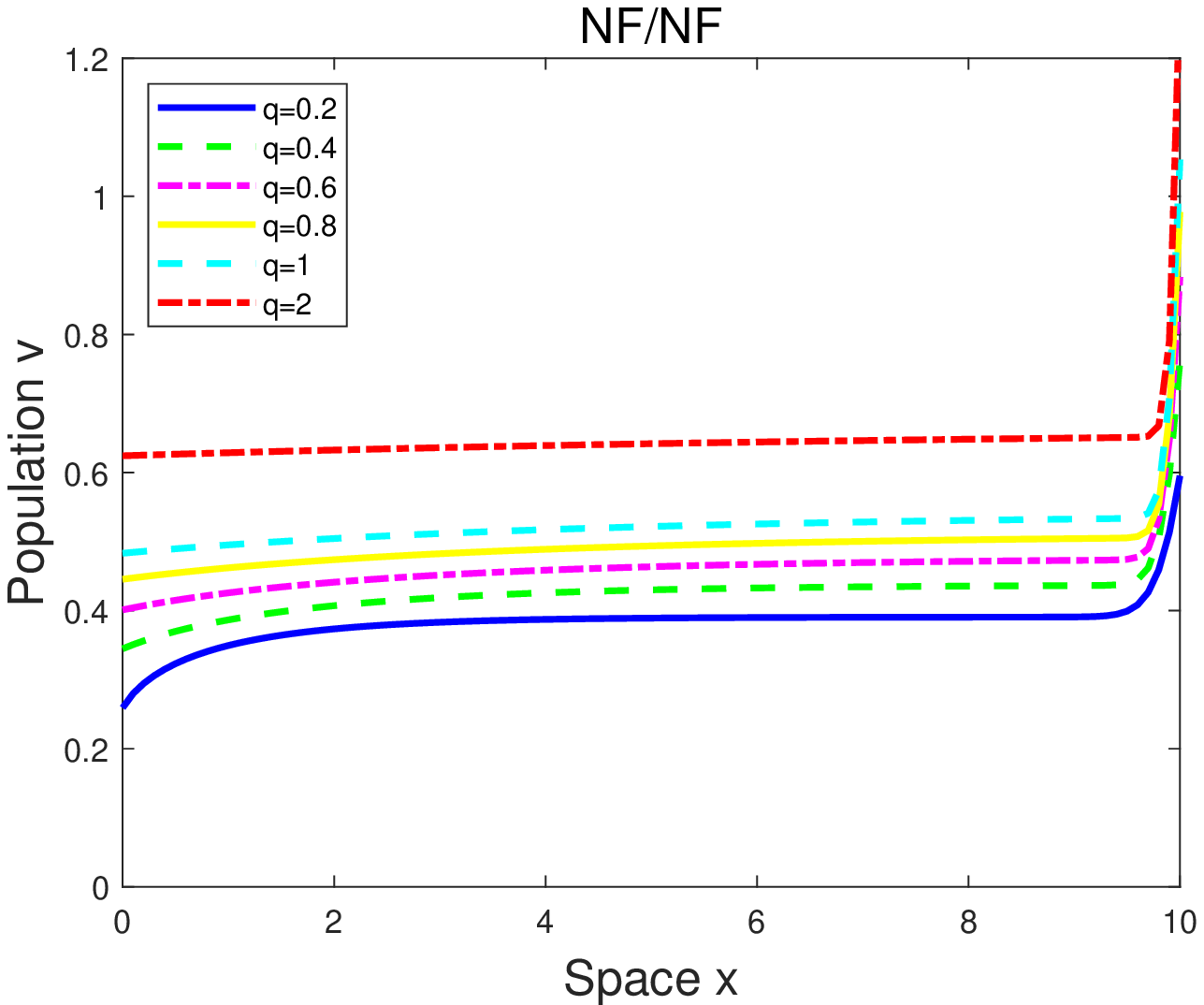}\\
\caption{The dependence of maximal steady state solution of \eqref{3s} on the varying advection rate $q$ under three types boundary conditions. Here   the parameters satisfy \eqref{434},  $\mu=0.04$, and the initial condition is $u_0=0.2$, $v_0=0.2$. Left: The drift population; Right: The benthic population.}
\label{3q}
\end{figure}

Finally Figure \ref{3bi} demonstrates the bistable nature of system \eqref{3s} under NF/FF ($b_u=0$ and $b_d=1$) boundary condition and $(H3)$ is satisfied, and Figure \ref{3bi2} describes the bistable phenomenon under NF/NF ($b_u=b_d=0$) boundary condition and $(H2)$  is satisfied. And when the cross-sectional areas of the benthic zone $A_b(x)$
and drift zone $A_b(x)$ are spatially heterogeneous, then the bistable structure is shown in Figure \ref{3bi3}. The population becomes extinct when  starting from small initial population  (first panel in Figure \ref{3bi},  \ref{3bi2} and \ref{3bi3}); and the population reaches the maximal steady state when starting from relatively large initial population (third panel in Figure \ref{3bi}, \ref{3bi2} and \ref{3bi3}). And the second panel in Figure \ref{3bi},  \ref{3bi2} and \ref{3bi3} also shows a ``stable" pattern with a transition layer. We conjecture that the transition layer solution is unstable and metastable (with a small positive eigenvalue), so the pattern can be observed for a long time in numerical simulation.

\section{Conclusion}
For a aquatic species that reproduce on the bottom of the river and release their larval stages into the water column, the longitudinal movement occurs only in the drift zone and individuals in the benthic zone in stream channel stays immobile. Through a benthic-drift model, we investigated the population persistence and extinction regarding the strength of interacting between zones. Moreover, this benthic-drift model has the feature of a coupled partial differential equation (PDE) for the drift population and an ``ordinary differential equation" (ODE) for the benthic population. This degenerate model causes a lack of the compactness of the solution orbits, which brings extra obstacles in the analysis. To overcome these difficulties, we turn to the Kuratowski measure of noncompactness in order to use the Lyapunov function.

For single compartment reaction-diffusion-advection equation, when the growth rate exhibits logistic type, it is well-known that the dynamics is either the population extinction or convergence to a positive steady state (monostable). If the species follows the strong Allee effect growth, when both the diffusion coefficient and the advection rate are small, there exist multiple positive steady state solutions hence the dynamics is bistable so that different initial conditions lead to different asymptotic behavior. On the other hand, when the advection rate is large, the population becomes extinct regardless of initial condition under most boundary conditions \cite{ws2018}.

Unlike the single compartment reaction-diffusion-advection equation with a strong Allee effect growth rate, in which the advection rate $q$ plays a important role in the persistence/extinction dynamics, the benthic-drift model dynamics with strong Allee effect relies more critically on the strength of interacting between zones, especially the transfer rate from the benthic zone to the drift zone $\mu$. In this paper, we show that how the transfer rates between benthic zone to the drift zone influence the population dynamics. We divided the $\mu$ (transfer rate from benthic zone to drift zone) and $\sigma$ (transfer rate from drift zone to benthic zone) phase plane into regions and studied the dynamical behavior on these parameter regions. When we have a relatively large $\mu$ (in H1), population extinction will happen independent of the initial conditions, the boundary condition, the diffusive and advective movement and the transfer rate from the drift zone to the benthic zone $\sigma$. For small $\mu$ (in (H3)), for large initial conditions, population persistence will happen regardless of the boundary condition, the diffusive and advective movement and the transfer rate from the drift zone to the benthic zone $\sigma$. Along with the locally stability of the zero steady state solution, bistable dynamical behavior can be confirmed. When the transfer rate $\mu$ is in the intermediate range (in (H2)), the persistence or extinction depends on the diffusive and advective movement. And under closed environment, a multiplicity result for the steady state solutions is also obtained for small advection rate.

\begin{figure}[h]
\centering
\includegraphics[width=0.5\textwidth]{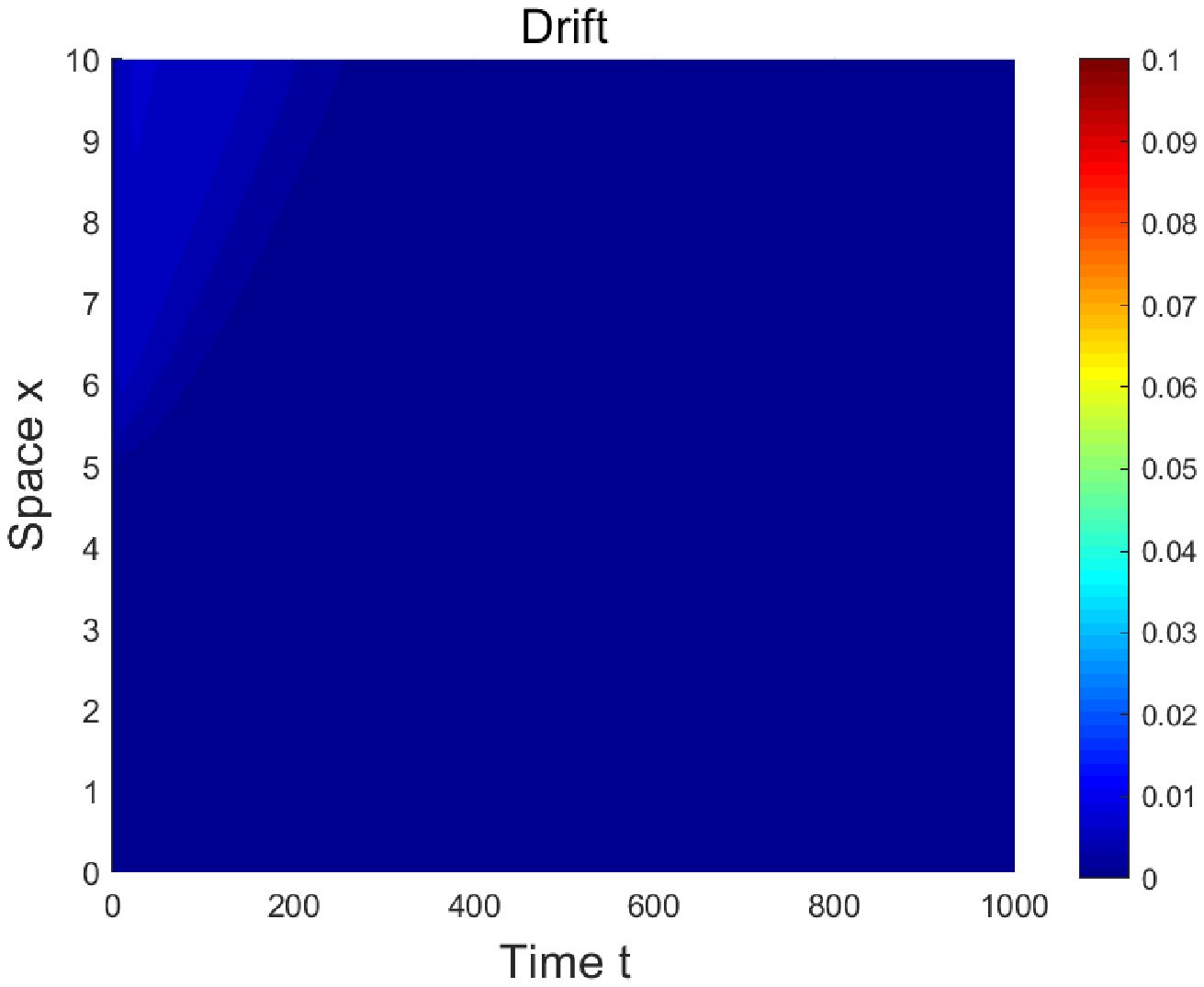}\includegraphics[width=0.5\textwidth]{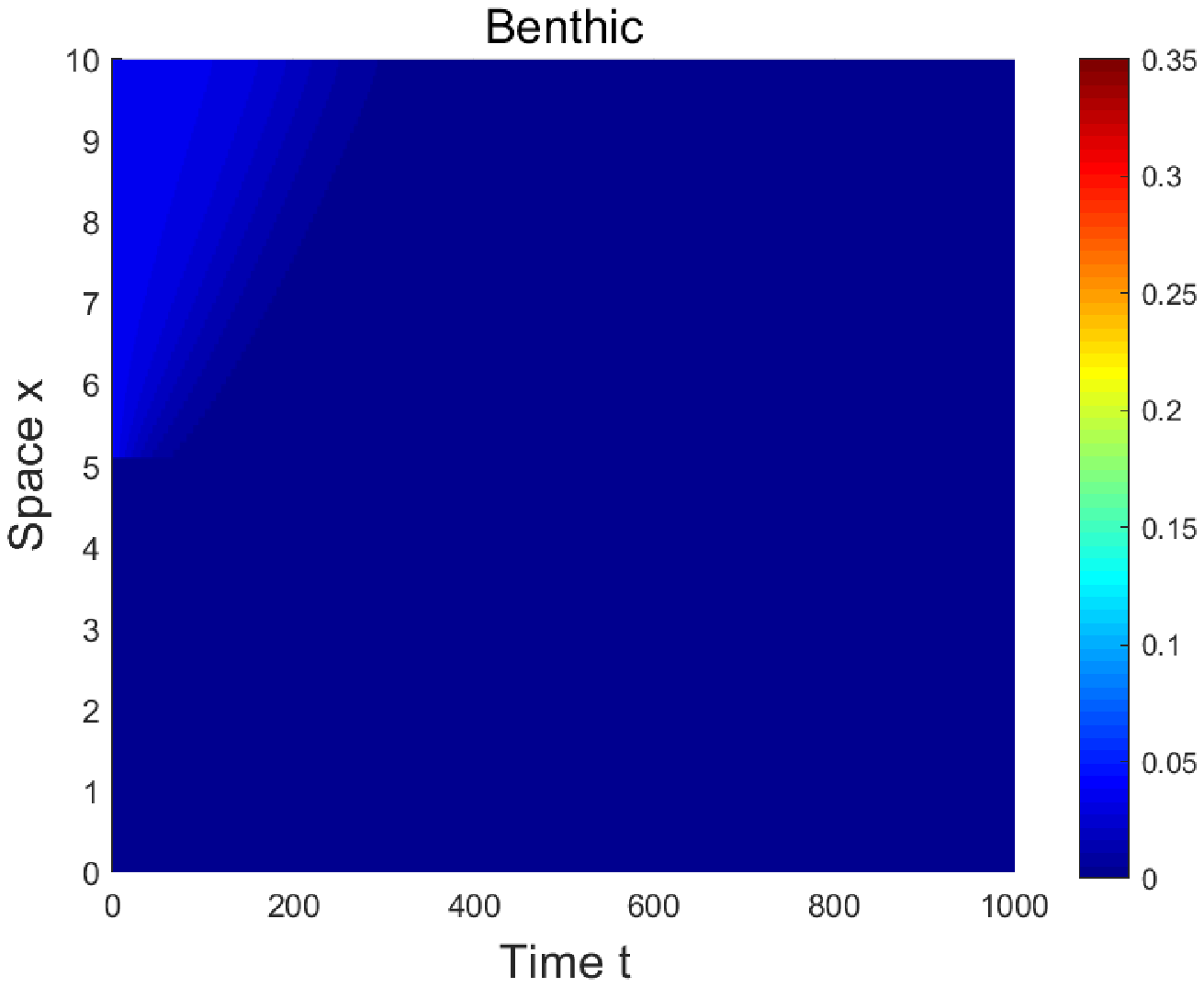}\\
\includegraphics[width=0.5\textwidth]{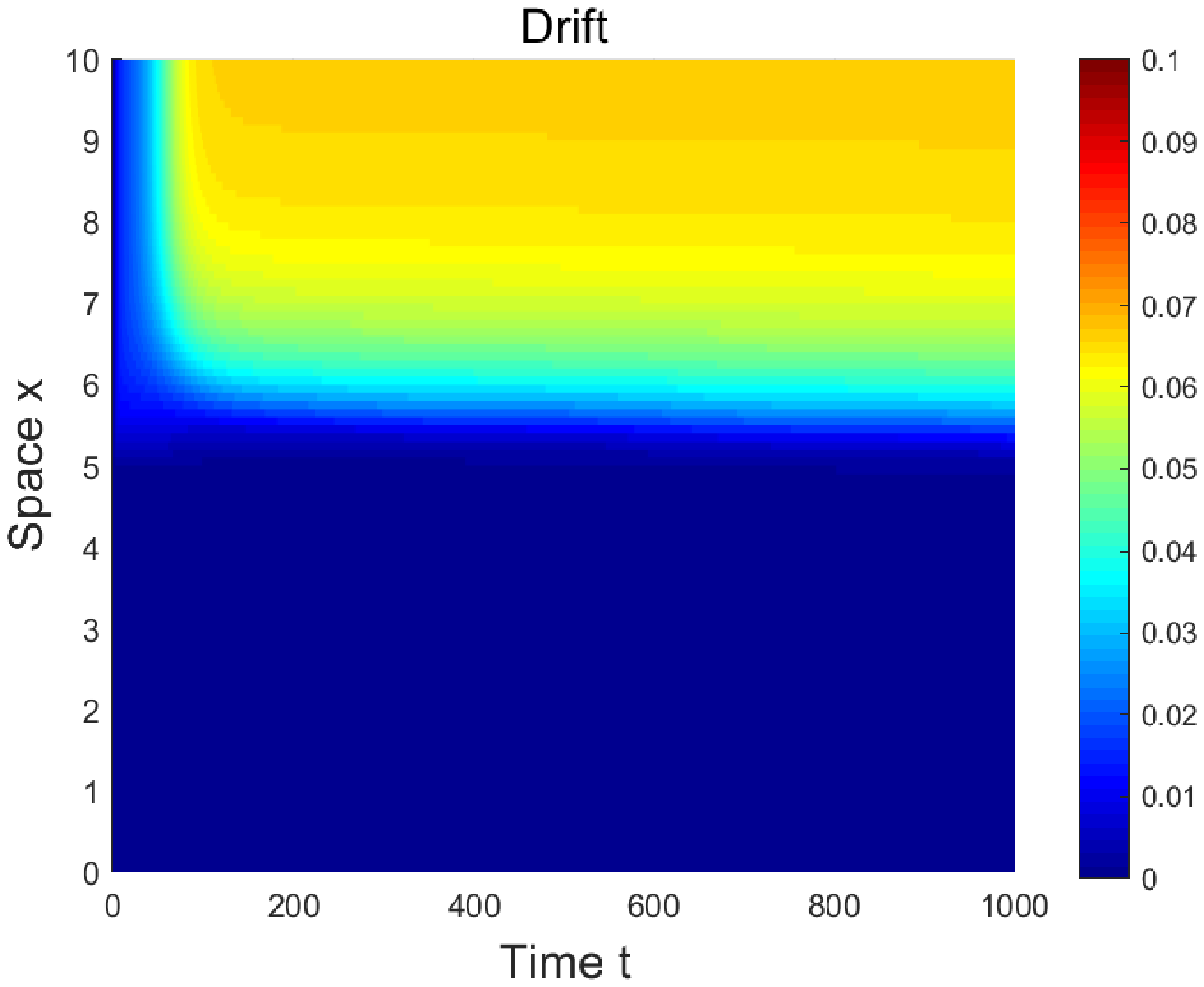}\includegraphics[width=0.5\textwidth]{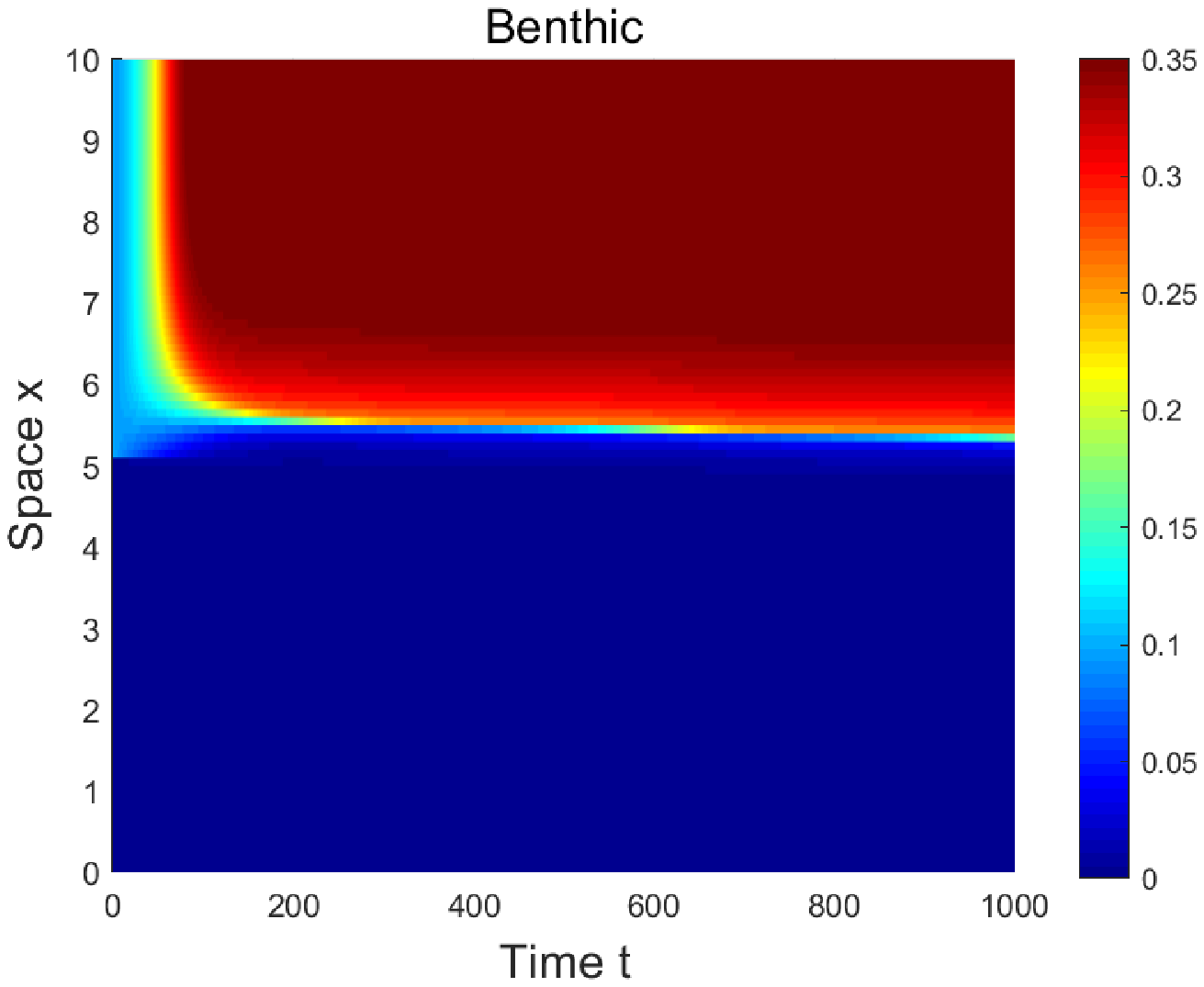}\\
\includegraphics[width=0.5\textwidth]{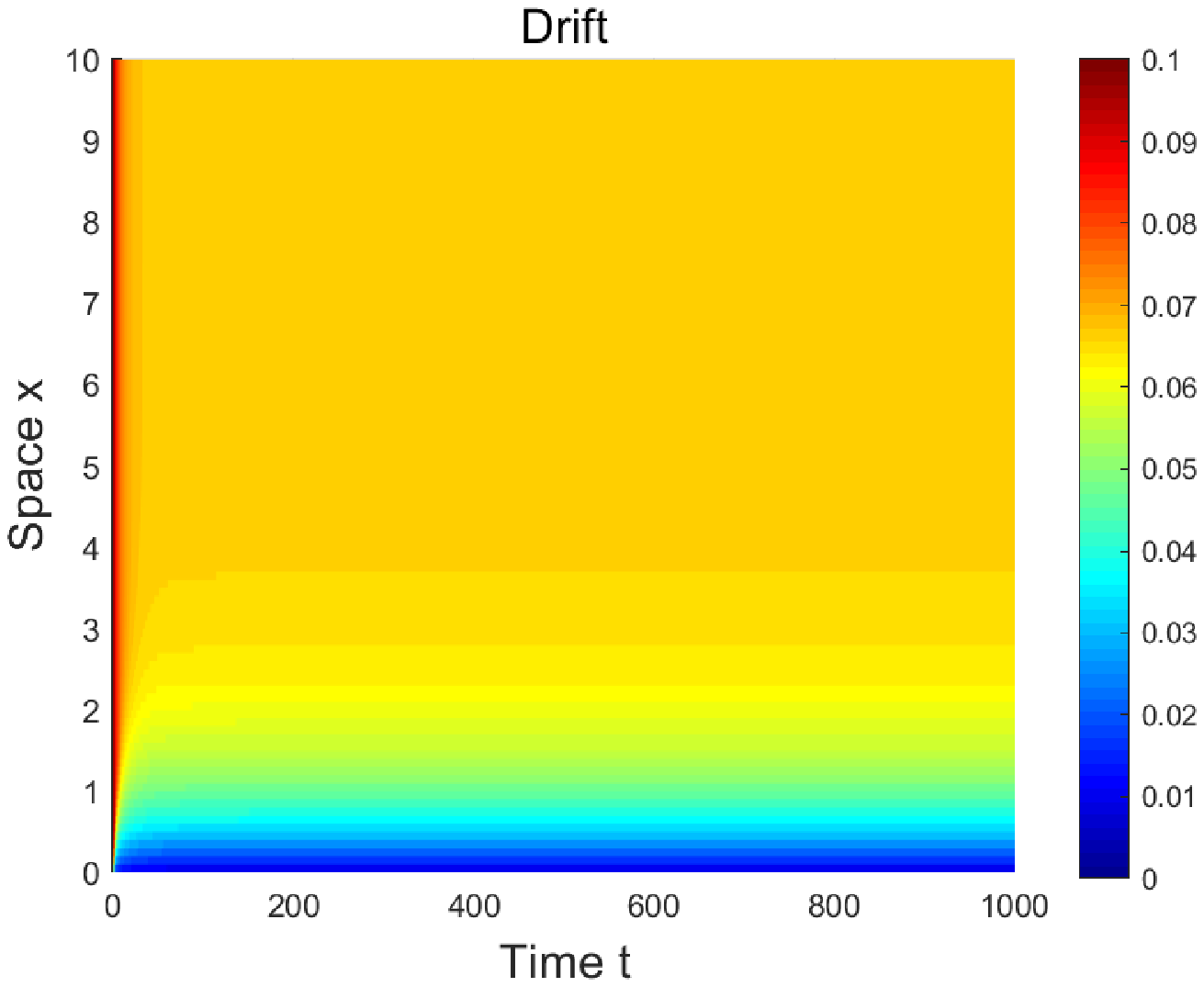}\includegraphics[width=0.5\textwidth]{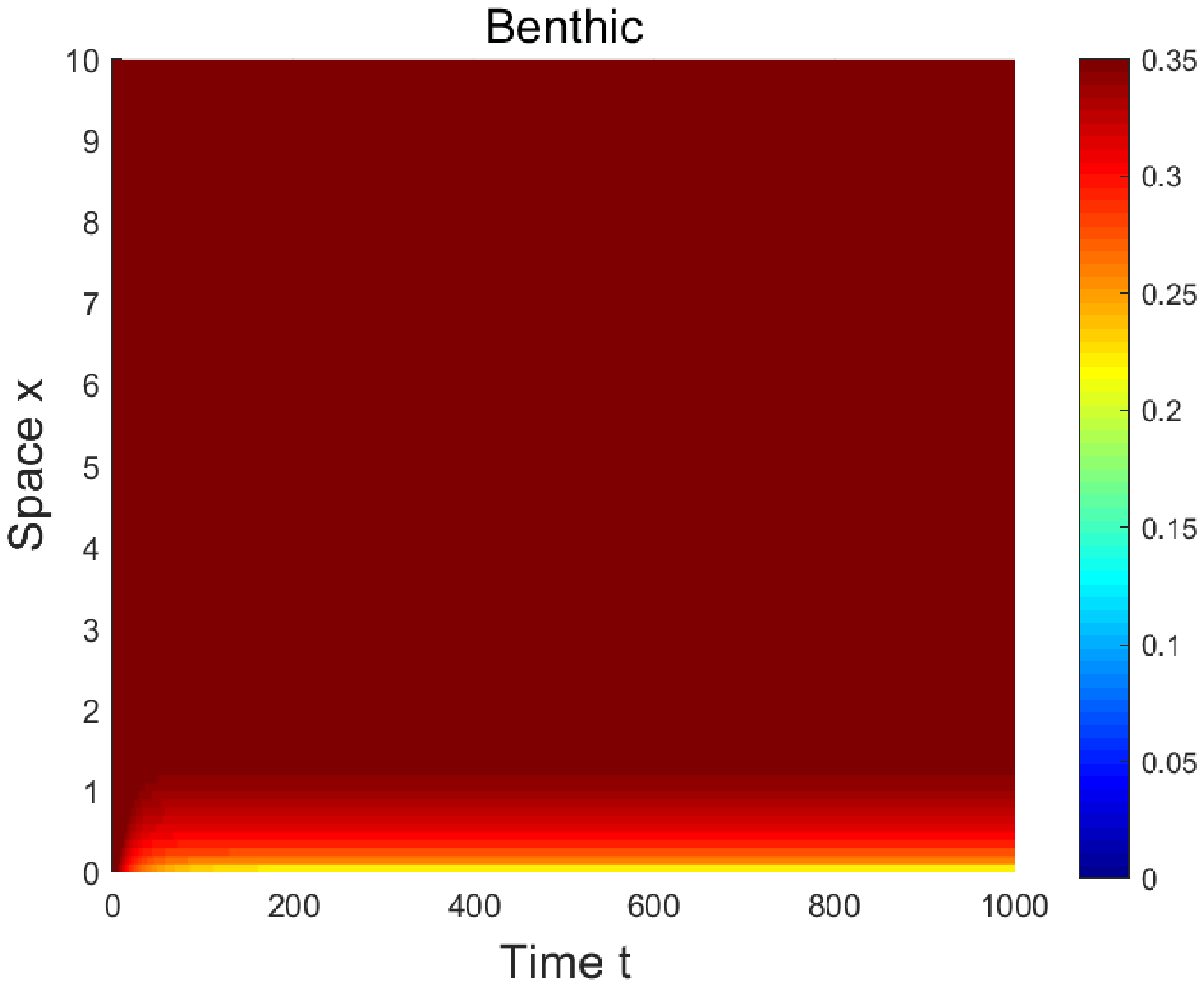}\\
\caption{Bistable dynamics for different initial conditions. Here  the parameters satisfy \eqref{434}, $q=0.11$,  $\mu=0.04$, $b_u=0$, $b_d=1$. The initial conditions from first row to thrid row are $u_0(x)=0$, $v_0(x)=0$ for $x\in[0, L/2]$ and $v_0(x)=0.04$ for $x\in[L/2, L]$; $u_0(x)=0$, $v_0(x)=0$ for $x\in[0, L/2]$ and $v_0(x)=0.1$ for $x\in[L/2, L]$; $u_0(x)=0.1$, $v_0(x)=0$ for $x\in[0, L/2]$ and $v_0(x)=0.1$ for $x\in[L/2, L]$; $u_0(x)=0.1$, $v_0(x)=0.4$. Left: the drift population; Right: the benthic population.}
\label{3bi}
\end{figure}

\begin{figure}[h]
\centering
\includegraphics[width=0.5\textwidth]{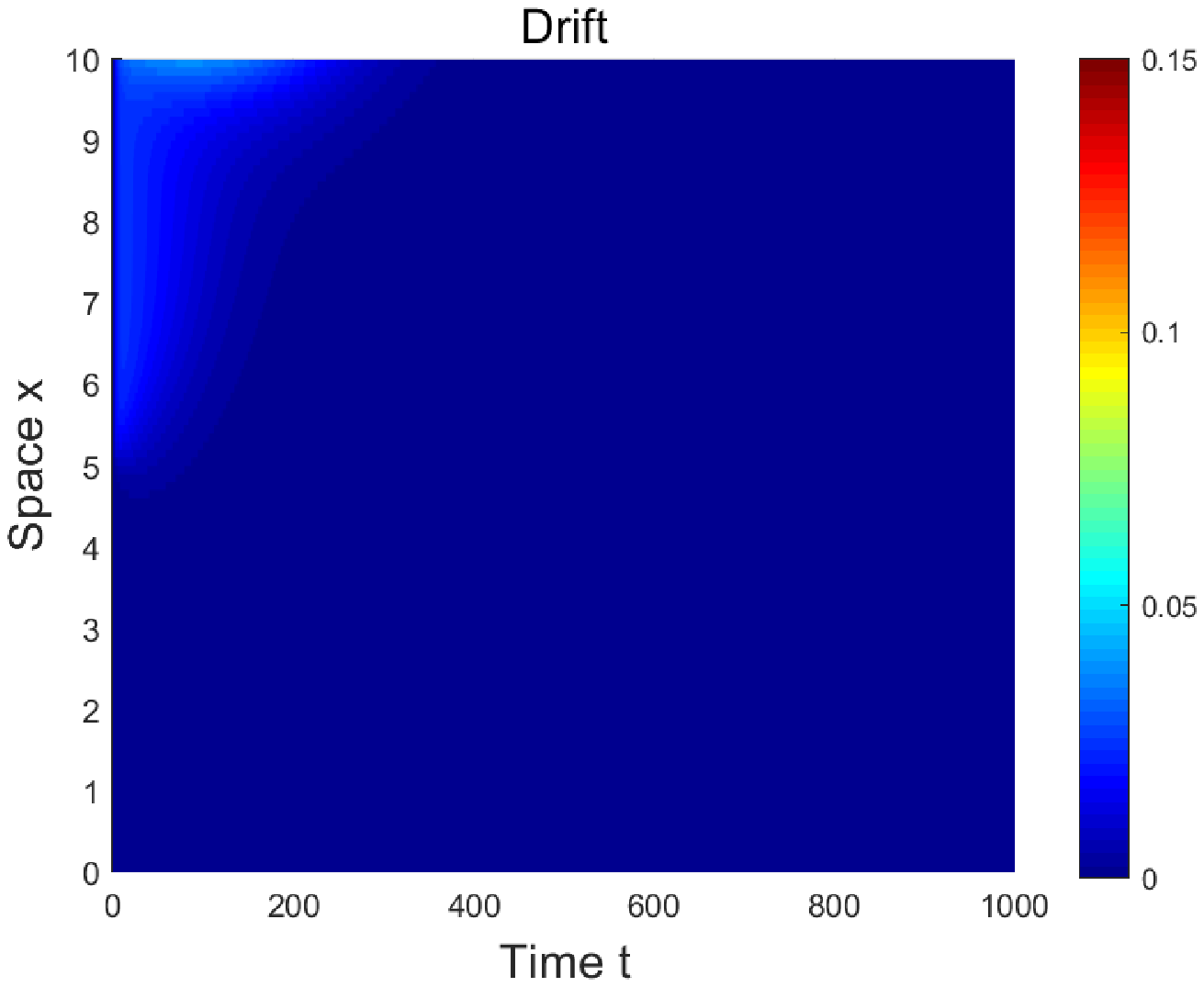}\includegraphics[width=0.5\textwidth]{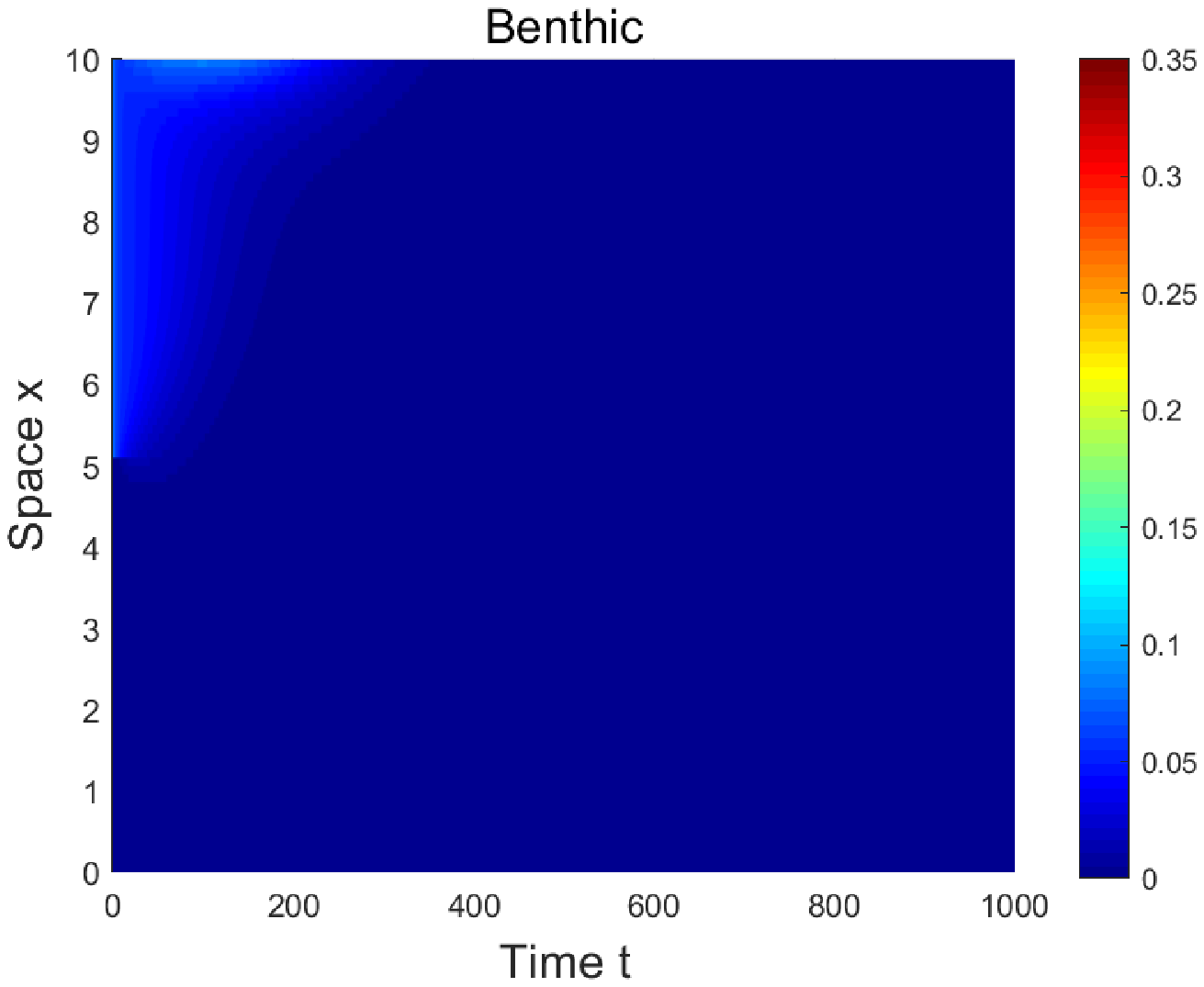}\\
\includegraphics[width=0.5\textwidth]{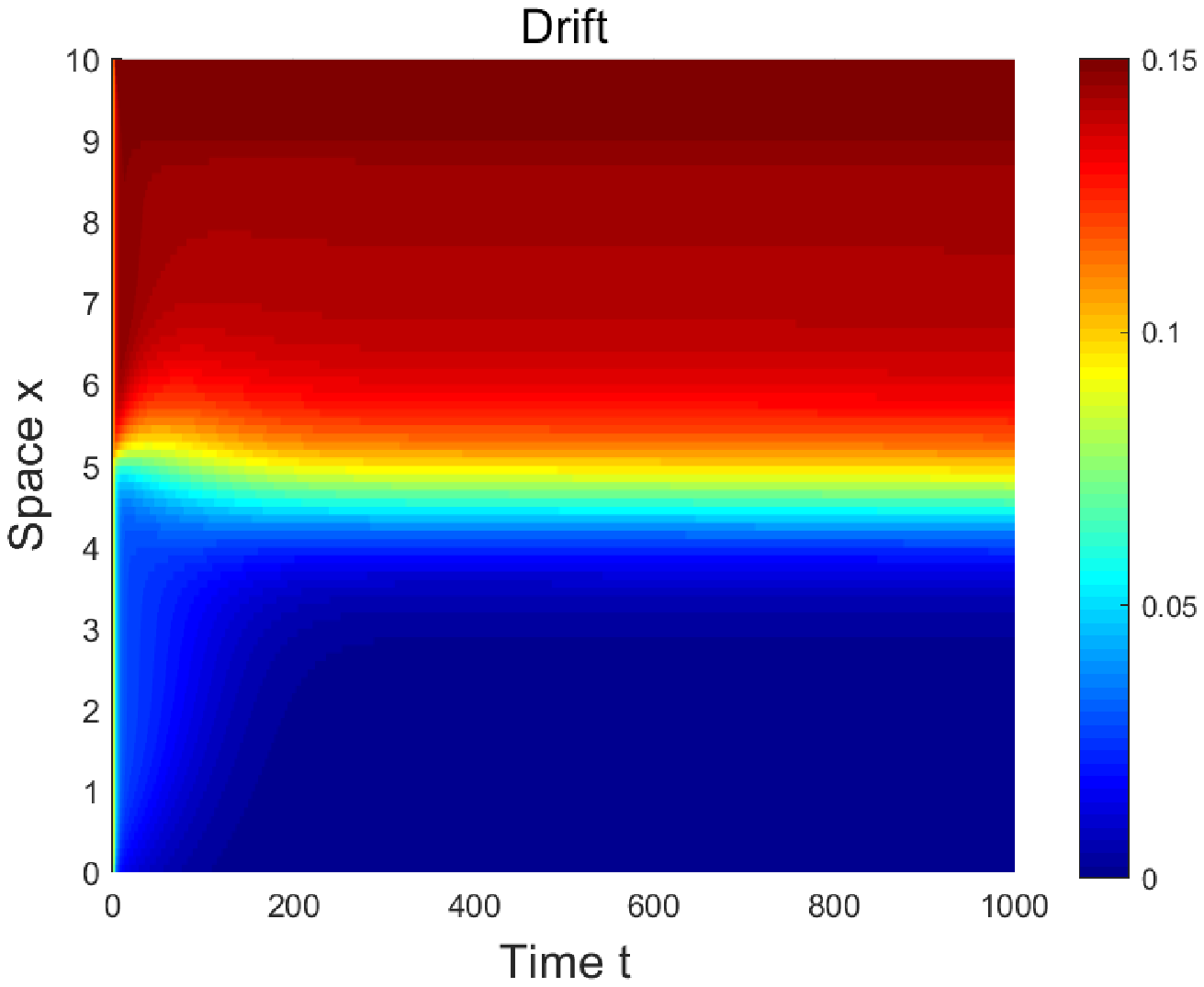}\includegraphics[width=0.5\textwidth]{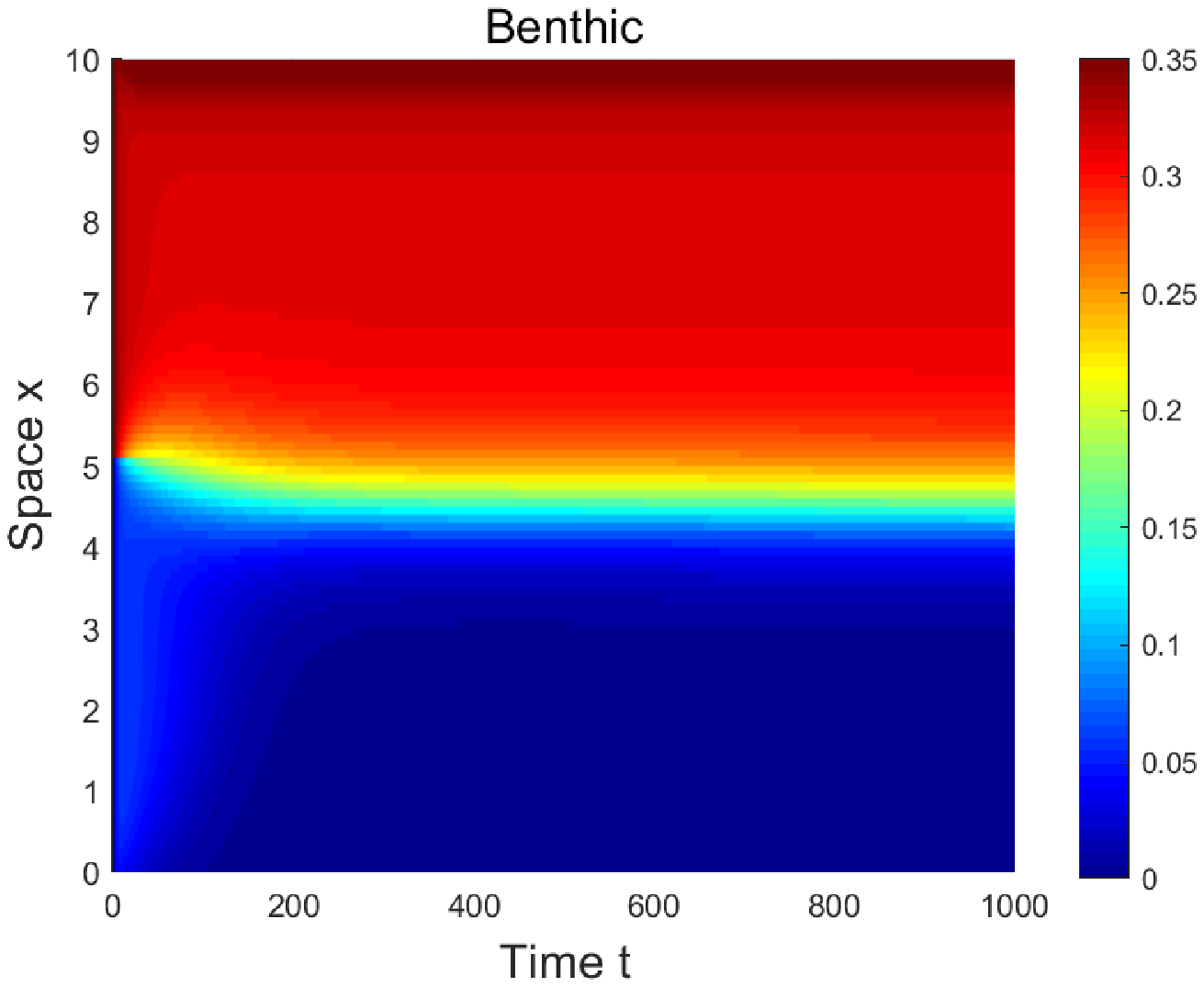}\\
\includegraphics[width=0.5\textwidth]{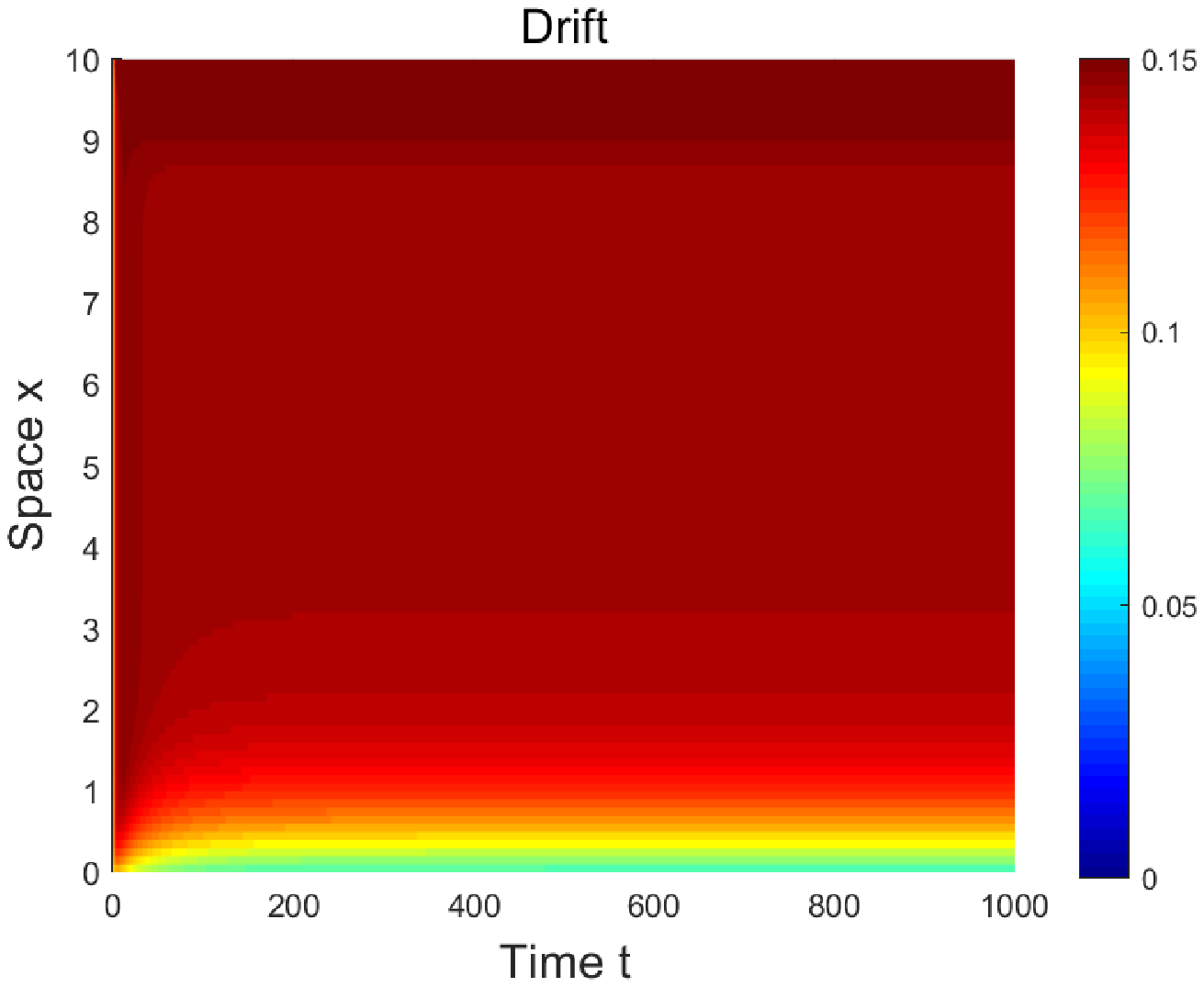}\includegraphics[width=0.5\textwidth]{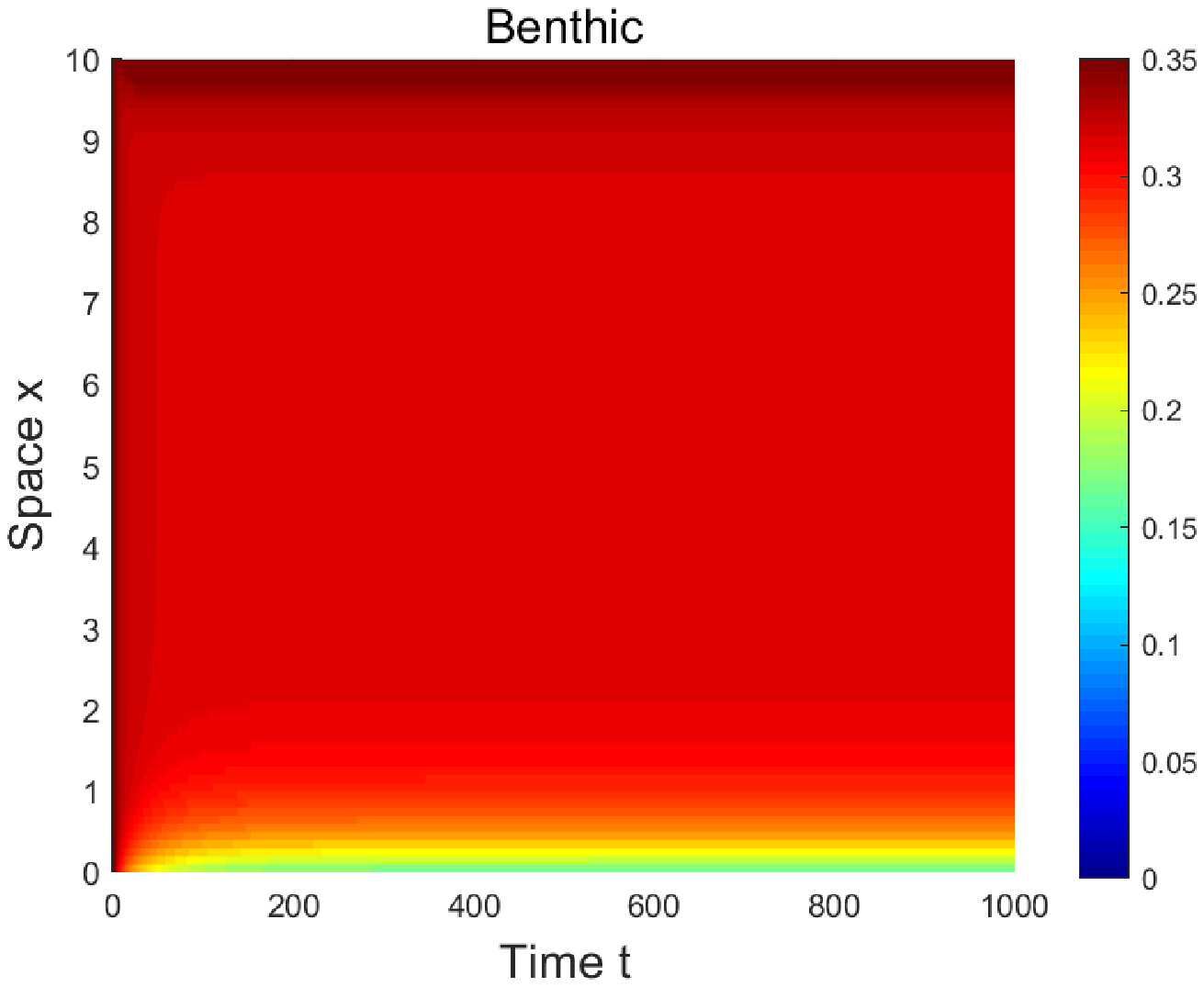}\\
\caption{Bistable dynamics for different initial conditions. Here  the parameters satisfy \eqref{434}, $q=0.025$, $\mu=0.1$, $b_u=b_d=0$. The initial conditions from first row to third row are $u_0(x)=0$, $v_0(x)=0$ for $x\in[0, L/2]$ and $v_0(x)=0.08$ for $x\in[L/2, L]$; $u_0(x)=0.1$, $v_0(x)=0$ for $x\in[0, L/2]$ and $v_0(x)=0.4$ for $x\in[L/2, L]$; $u_0(x)=0.1$, $v_0(x)=0.4$. Left: the drift population; Right: the benthic population.}
\label{3bi2}
\end{figure}

\begin{figure}[h]
\centering
\includegraphics[width=0.5\textwidth]{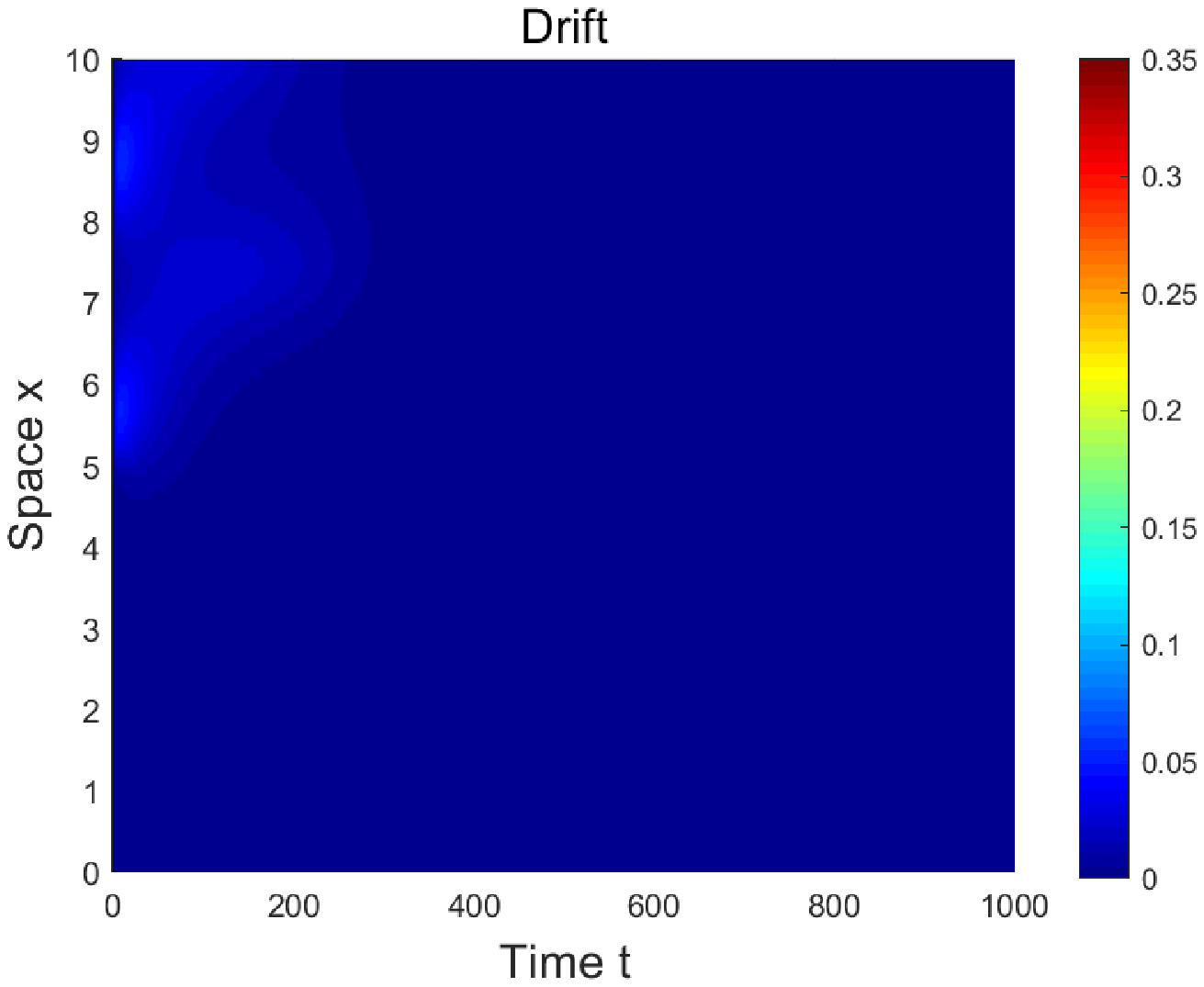}\includegraphics[width=0.5\textwidth]{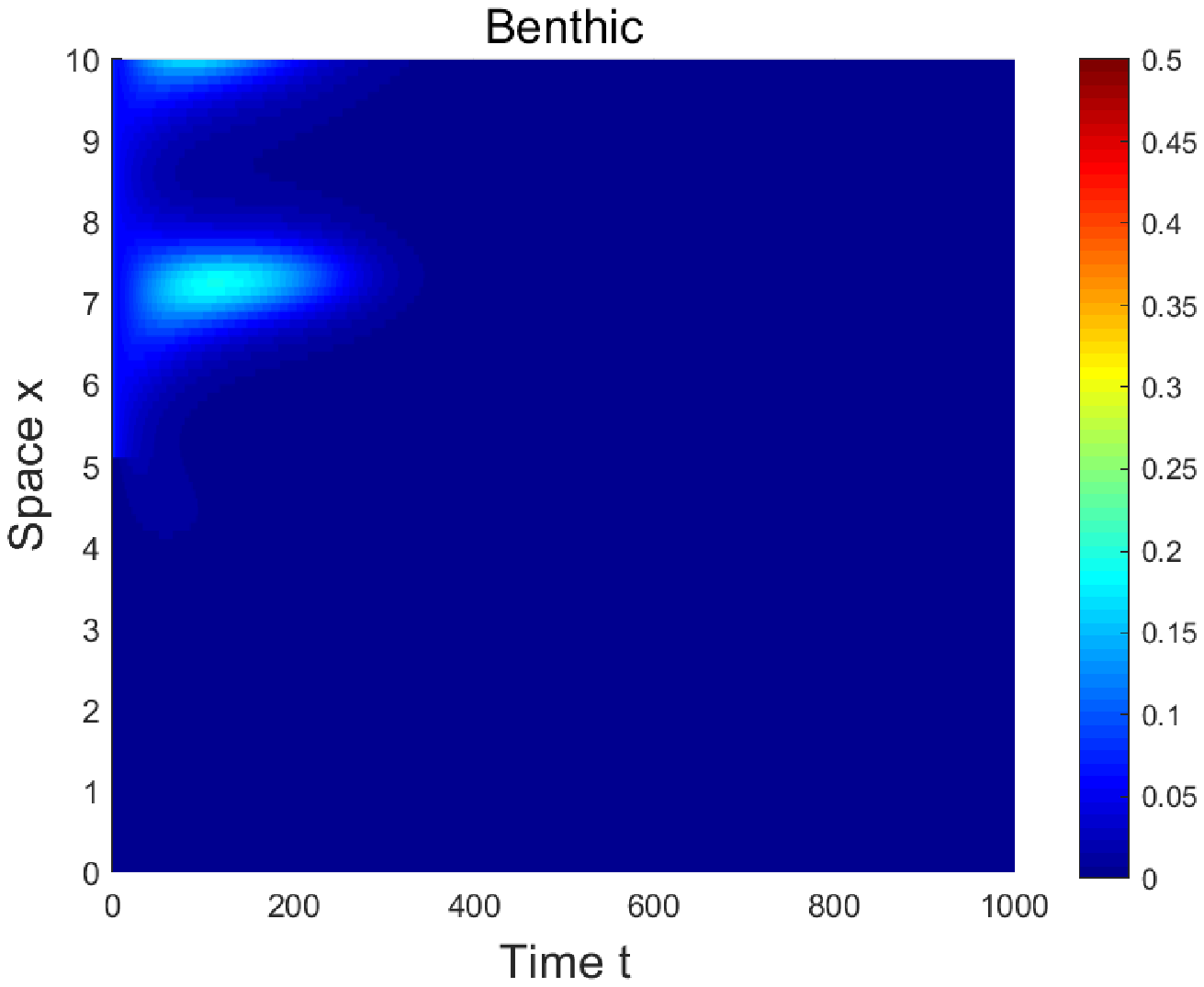}\\
\includegraphics[width=0.5\textwidth]{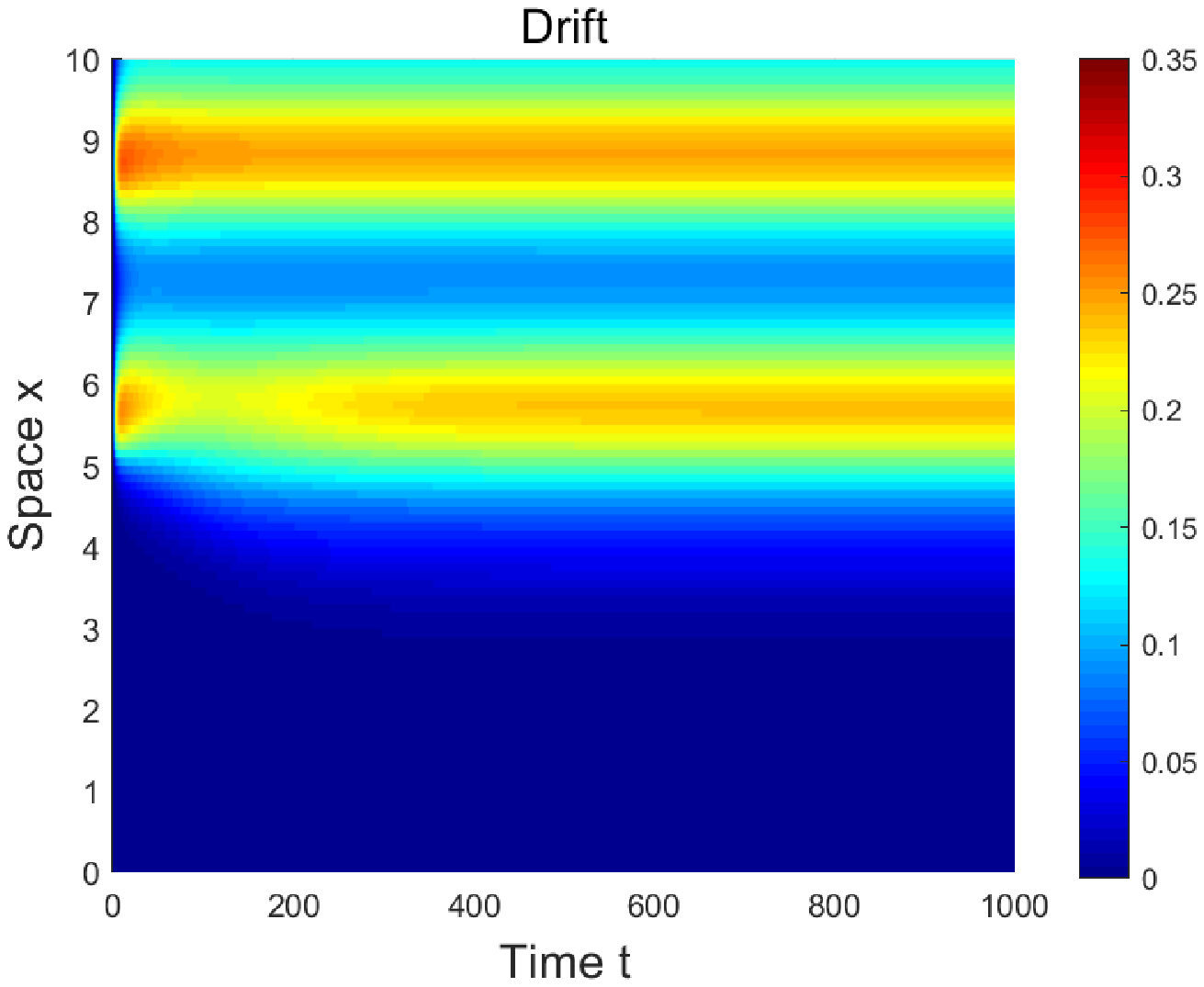}\includegraphics[width=0.5\textwidth]{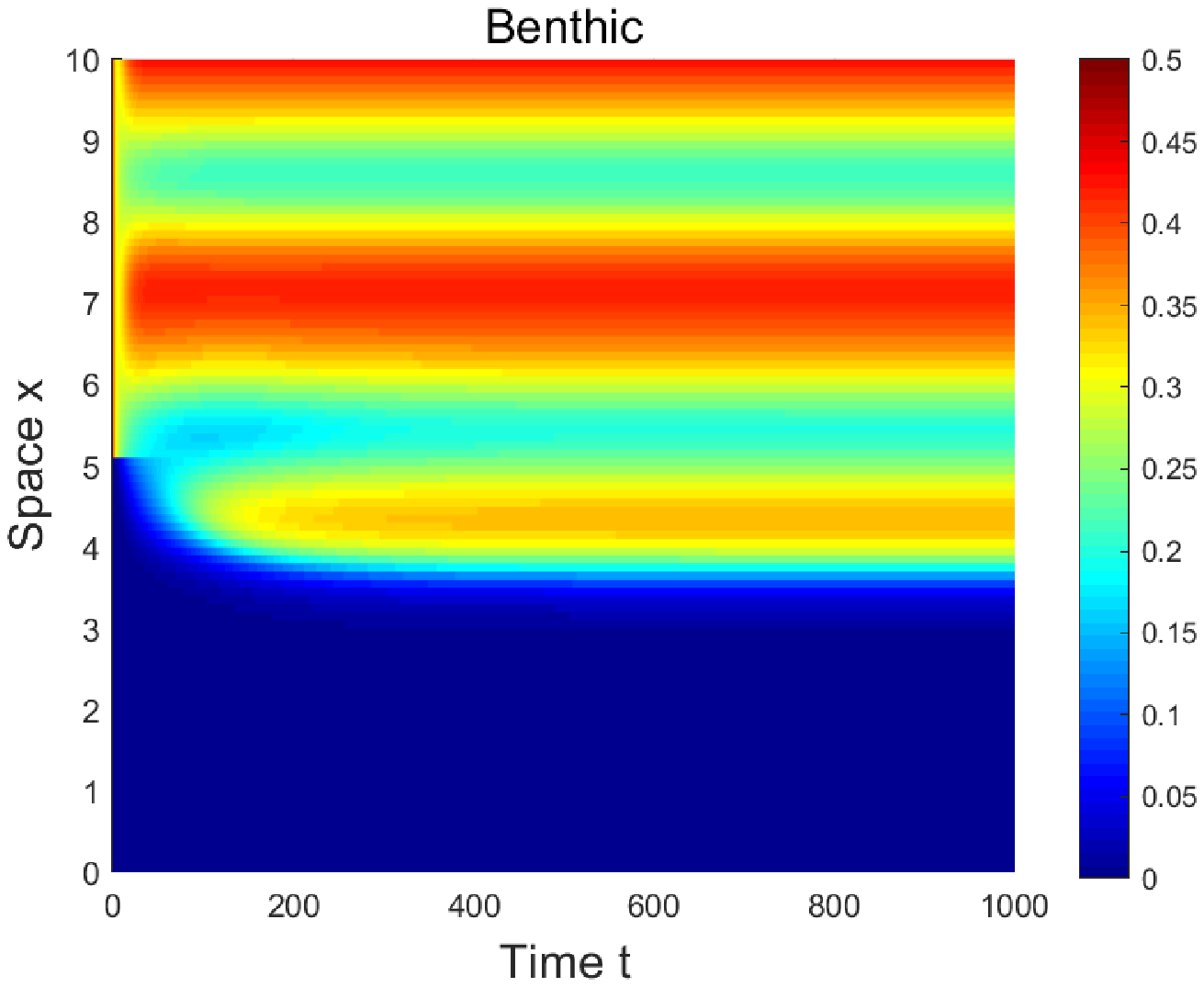}\\
\includegraphics[width=0.5\textwidth]{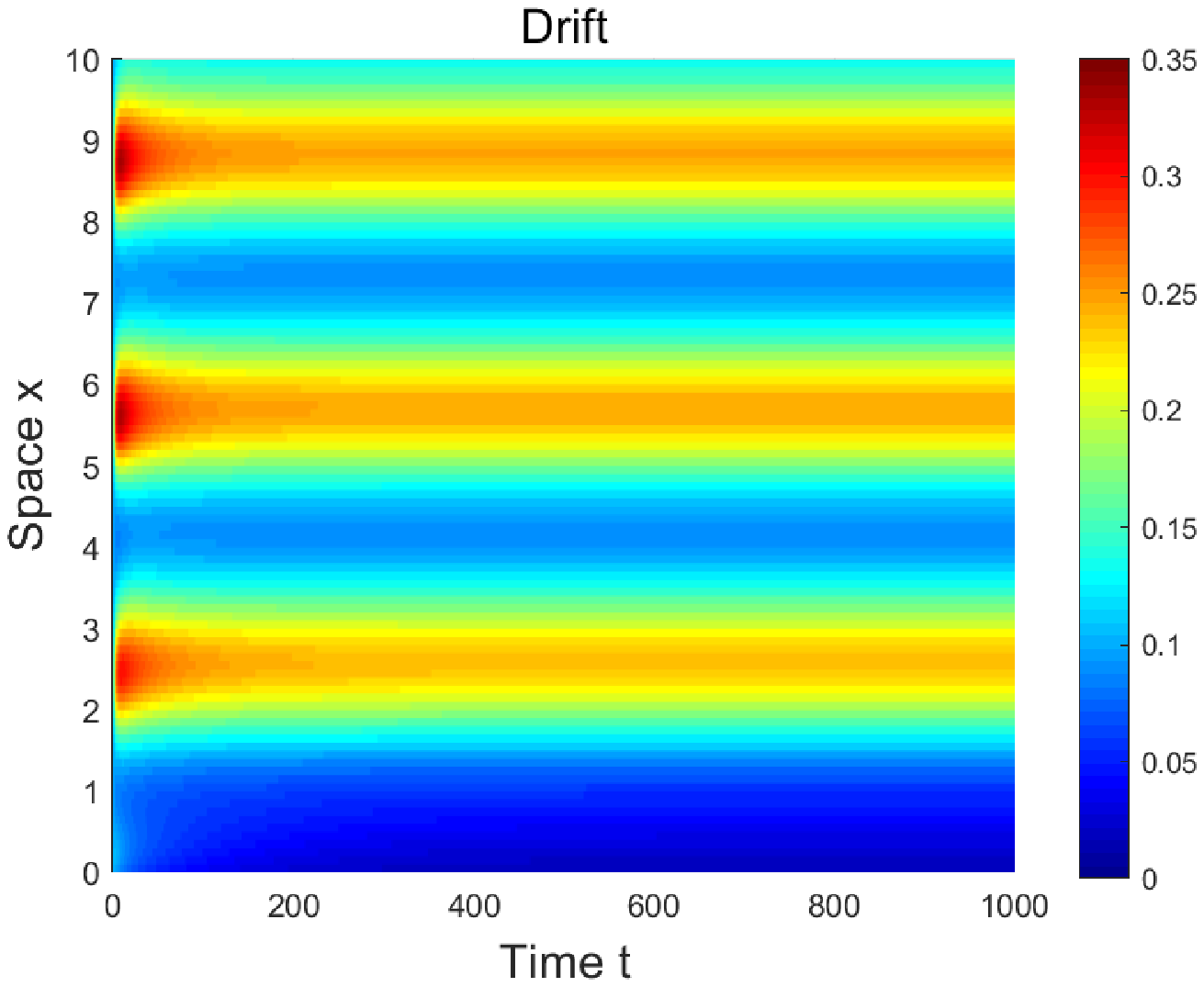}\includegraphics[width=0.5\textwidth]{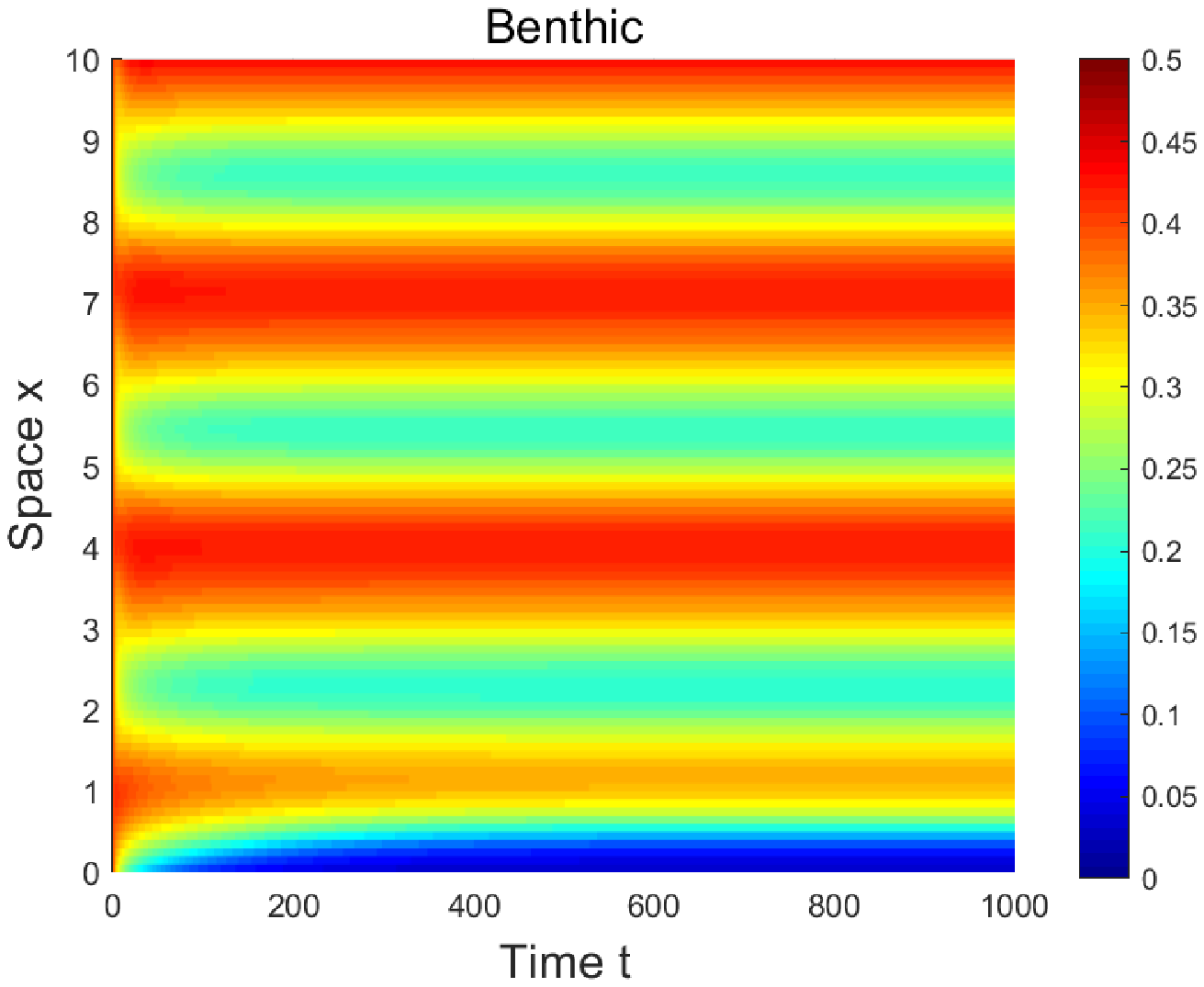}\\
\caption{Bistable dynamics for different initial conditions. Here the parameters satisfy \eqref{434} (except $A_b(x)$ and $A_d(x)$), $q=0.025$, $\mu=0.1$,  $A_d(x)=\sin2x+2$, $A_d(x)=\sin(2x-10)+2$. The initial conditions from first row to third row are $u_0(x)=0$, $v_0(x)=0$ for $x\in[0, L/2]$ and $v_0(x)=0.1$ for $x\in[L/2, L]$; $u_0(x)=0.08$, $v_0(x)=0$ for $x\in[0, L/2]$ and $v_0(x)=0.4$ for $x\in[L/2, L]$; $u_0(x)=0.1$, $v_0(x)=0.4$. Left: the drift population; Right: the benthic population.}
\label{3bi3}
\end{figure}



\bibliographystyle{plain}
\bibliography{advection}

\begin{thebibliography}{10}

\bibitem{amann1976}
H.~Amann.
\newblock Fixed point equations and nonlinear eigenvalue problems in ordered
  {B}anach spaces.
\newblock {\em SIAM Rev.}, 18(4):620--709, 1976.

\bibitem{bencala1983simulation}
K.~E. Bencala and R.~A. Walters.
\newblock Simulation of solute transport in a mountain pool-and-riffle stream:
  A transient storage model.
\newblock {\em Water Resources Research}, 19(3):718--724, 1983.

\bibitem{CC2003}
R.~S. Cantrell and C.~Cosner.
\newblock {\em Spatial ecology via reaction-diffusion equations}.
\newblock Wiley Series in Mathematical and Computational Biology. John Wiley \&
  Sons, Ltd., Chichester, 2003.

\bibitem{Cui2014JDE}
R.-H. Cui, J.-P. Shi, and B.-Y. Wu.
\newblock Strong {A}llee effect in a diffusive predator-prey system with a
  protection zone.
\newblock {\em J. Differential Equations}, 256(1):108--129, 2014.

\bibitem{deangelis1995modelling}
D.~L. DeAngelis, M.~Loreau, D.~Neergaard, P.~J. Mulholland, and E.~R. Marzolf.
\newblock Modelling nutrient-periphyton dynamics in streams: the importance of
  transient storage zones.
\newblock {\em Ecological Modelling}, 80(2-3):149--160, 1995.

\bibitem{du2019JDE}
K.~Du, R.~Peng, and N.-K. Sun.
\newblock The role of protection zone on species spreading governed by a
  reaction-diffusion model with strong {A}llee effect.
\newblock {\em J. Differential Equations}, 266(11):7327--7356, 2019.

\bibitem{Grover2009}
J.~P. Grover, S.-B. Hsu, and F.-B. Wang.
\newblock Competition and coexistence in flowing habitats with a hydraulic
  storage zone.
\newblock {\em Math. Biosci.}, 222(1):42--52, 2009.

\bibitem{henry1981}
D.~Henry.
\newblock {\em Geometric theory of semilinear parabolic equations}, volume 840
  of {\em Lecture Notes in Mathematics}.
\newblock Springer-Verlag, Berlin-New York, 1981.

\bibitem{Hsu2010}
S.-B. Hsu and Y.~Lou.
\newblock Single phytoplankton species growth with light and advection in a
  water column.
\newblock {\em SIAM J. Appl. Math.}, 70(8):2942--2974, 2010.

\bibitem{hsu2011}
S.-B. Hsu, F.-B. Wang, and X.-Q. Zhao.
\newblock Dynamics of a periodically pulsed bio-reactor model with a hydraulic
  storage zone.
\newblock {\em Journal of Dynamics and Differential Equations}, 23(4):817--842,
  2011.

\bibitem{hsu2013}
S.-B. Hsu, F.-B. Wang, and X.-Q. Zhao.
\newblock Global dynamics of zooplankton and harmful algae in flowing habitats.
\newblock {\em J. Differential Equations}, 255(3):265--297, 2013.

\bibitem{huang2016}
Q.-H. Huang, Y.~Jin, and M.~A. Lewis.
\newblock {$R_0$} analysis of a {B}enthic-drift model for a stream population.
\newblock {\em SIAM J. Appl. Dyn. Syst.}, 15(1):287--321, 2016.

\bibitem{huisman2002sinking}
J.~Huisman, M.~Array{\'a}s, U.~Ebert, and B.~Sommeijer.
\newblock How do sinking phytoplankton species manage to persist?
\newblock {\em Amer. Naturalist}, 159(3):245--254, 2002.

\bibitem{Jin2011}
Y.~Jin and M.~A. Lewis.
\newblock Seasonal influences on population spread and persistence in streams:
  critical domain size.
\newblock {\em SIAM J. Appl. Math.}, 71(4):1241--1262, 2011.

\bibitem{Jin2017}
Y.~Jin, F.~Lutscher, and Y.~Pei.
\newblock Meandering rivers: how important is lateral variability for species
  persistence?
\newblock {\em Bull. Math. Biol.}, 79(12):2954--2985, 2017.

\bibitem{Jin2019}
Y.~Jin, R.~Peng, and J.-P. Shi.
\newblock Population dynamics in river networks.
\newblock {\em J. Nonlinear Sci.}, pages 1--45, 2019 (to appear).

\bibitem{jin2018dynamics}
Y.~Jin and F.-B. Wang.
\newblock Dynamics of a benthic-drift model for two competitive species.
\newblock {\em J. Math. Anal. Appl.}, 462(1):840--860, 2018.

\bibitem{LLL2015JBD}
K.~Y. Lam, Y.~Lou, and F.~Lutscher.
\newblock Evolution of dispersal in closed advective environments.
\newblock {\em J. Biol. Dyn.}, 9(suppl. 1):188--212, 2015.

\bibitem{LLL2016SIAM}
K.~Y. Lam, Y.~Lou, and F.~Lutscher.
\newblock The {E}mergence of {R}ange {L}imits in {A}dvective {E}nvironments.
\newblock {\em SIAM J. Appl. Math.}, 76(2):641--662, 2016.

\bibitem{LL2014JMB}
Y.~Lou and F.~Lutscher.
\newblock Evolution of dispersal in open advective environments.
\newblock {\em J. Math. Biol.}, 69(6-7):1319--1342, 2014.

\bibitem{lou2019global}
Y.~Lou, X.-Q. Zhao, and P.~Zhou.
\newblock Global dynamics of a {L}otka-{V}olterra
  competition-diffusion-advection system in heterogeneous environments.
\newblock {\em J. Math. Pures Appl. (9)}, 121:47--82, 2019.

\bibitem{Louzhou2015}
Y.~Lou and P.~Zhou.
\newblock Evolution of dispersal in advective homogeneous environment: the
  effect of boundary conditions.
\newblock {\em J. Differential Equations}, 259(1):141--171, 2015.

\bibitem{Lutscher2006}
F.~Lutscher, M.~A. Lewis, and E.~McCauley.
\newblock Effects of heterogeneity on spread and persistence in rivers.
\newblock {\em Bull. Math. Biol.}, 68(8):2129--2160, 2006.

\bibitem{Lutscher2005}
F.~Lutscher, E.~Pachepsky, and M.~A. Lewis.
\newblock The effect of dispersal patterns on stream populations.
\newblock {\em SIAM J. Appl. Math.}, 65(4):1305--1327, 2005.

\bibitem{mzhao2005}
P.~Magal and X.-Q. Zhao.
\newblock Global attractors and steady states for uniformly persistent
  dynamical systems.
\newblock {\em SIAM J. Math. Anal.}, 37(1):251--275, 2005.

\bibitem{marciniak2017}
A.~Marciniak-Czochra, G.~Karch, and K.~Suzuki.
\newblock Instability of {T}uring patterns in reaction-diffusion-{ODE} systems.
\newblock {\em J. Math. Biol.}, 74(3):583--618, 2017.

\bibitem{martin1990}
R.~H. Martin, Jr. and H.~L. Smith.
\newblock Abstract functional-differential equations and reaction-diffusion
  systems.
\newblock {\em Trans. Amer. Math. Soc.}, 321(1):1--44, 1990.

\bibitem{Mckenzie2012}
H.~W. Mckenzie, Y.~Jin, J.~Jacobsen, and M.~A. Lewis.
\newblock {$R_0$} analysis of a spatiotemporal model for a stream population.
\newblock {\em SIAM J. Appl. Dyn. Syst.}, 11(2):567--596, 2012.

\bibitem{muller1954investigations}
K.~M\"uller.
\newblock Investigations on the organic drift in {N}orth {S}wedish streams.
\newblock {\em Report of the Institute of freshwater research, Drottningholm},
  35:133--148, 1954.

\bibitem{Pachepsky200561}
E.~Pachepsky, F.~Lutscher, R.M. Nisbet, and M.A. Lewis.
\newblock Persistence, spread and the drift paradox.
\newblock {\em Theor. Popul. Biol.}, 67(1):61--73, 2005.

\bibitem{cvpao}
C.~V. Pao.
\newblock {\em Nonlinear parabolic and elliptic equations}.
\newblock Plenum Press, New York, 1992.

\bibitem{pao1996JMAA}
C.~V. Pao.
\newblock Dynamics of nonlinear parabolic systems with time delays.
\newblock {\em J. Math. Anal. Appl.}, 198(3):751--779, 1996.

\bibitem{Ramirez2012}
J.~M. Ramirez.
\newblock Population persistence under advection-diffusion in river networks.
\newblock {\em J. Math. Biol.}, 65(5):919--942, 2012.

\bibitem{Sarhad2014}
J.~Sarhad, R.~Carlson, and K.~E. Anderson.
\newblock Population persistence in river networks.
\newblock {\em J. Math. Biol.}, 69(2):401--448, 2014.

\bibitem{sattinger}
D.~H. Sattinger.
\newblock Monotone methods in nonlinear elliptic and parabolic boundary value
  problems.
\newblock {\em Indiana Univ. Math. J.}, 21:979--1000, 1971/72.

\bibitem{SS2006JMB}
J.-P. Shi and R.~Shivaji.
\newblock Persistence in reaction diffusion models with weak {A}llee effect.
\newblock {\em J. Math. Biol.}, 52(6):807--829, 2006.

\bibitem{smith2008}
H.~L. Smith.
\newblock {\em Monotone dynamical systems: An introduction to the theory of
  competitive and cooperative systems}, volume~41 of {\em Mathematical Surveys
  and Monographs}.
\newblock American Mathematical Society, Providence, RI, 1995.

\bibitem{speirs2001}
D.~C. Speirs and W.~S.~C. Gurney.
\newblock Population persistence in rivers and estuaries.
\newblock {\em Ecology}, 82(5):1219--1237, 2001.

\bibitem{vasilyeva2019population}
O.~Vasilyeva.
\newblock Population dynamics in river networks: analysis of steady states.
\newblock {\em Jour. Math. Biol.}, pages 1--38, 2019 (to appear).

\bibitem{Vasilyeva2010}
O.~Vasilyeva and F.~Lutscher.
\newblock Population dynamics in rivers: analysis of steady states.
\newblock {\em Can. Appl. Math. Q.}, 18(4):439--469, 2010.

\bibitem{Wang2015}
F.-B. Wang, S.-B. Hsu, and X.-Q. Zhao.
\newblock A reaction-diffusion-advection model of harmful algae growth with
  toxin degradation.
\newblock {\em J. Differential Equations}, 259(7):3178--3201, 2015.

\bibitem{Wang2015b}
F.-B. Wang, J.-P. Shi, and X.-F. Zou.
\newblock Dynamics of a host-pathogen system on a bounded spatial domain.
\newblock {\em Commun. Pure Appl. Anal.}, 14(6):2535--2560, 2015.

\bibitem{zhao2012}
W.-D. Wang and X.-Q. Zhao.
\newblock Basic reproduction numbers for reaction-diffusion epidemic models.
\newblock {\em SIAM J. Appl. Dyn. Syst.}, 11(4):1652--1673, 2012.

\bibitem{ws2019}
Y.~Wang and J.-P. Shi.
\newblock Persistence and extinction of population in
  reaction-diffusion-advection model with weak {A}llee effect growth.
\newblock {\em SIAM J. Appl. Math.}, 79(4):1293--1313, 2019.

\bibitem{ws2018}
Y.~Wang, J.-P. Shi, and J.-F. Wang.
\newblock Persistence and extinction of population in
  reaction-diffusion-advection model with strong {A}llee effect growth.
\newblock {\em J. Math. Biol.}, 78(7):2093--2140, 2019.

\bibitem{Zhang2007}
K.~F. Zhang and X.-Q. Zhao.
\newblock Asymptotic behaviour of a reaction-diffusion model with a quiescent
  stage.
\newblock {\em Proc. R. Soc. Lond. Ser. A Math. Phys. Eng. Sci.},
  463(2080):1029--1043, 2007.

\bibitem{zz2016}
X.~Q. Zhao and P.~Zhou.
\newblock On a {L}otka-{V}olterra competition model: the effects of advection
  and spatial variation.
\newblock {\em Calc. Var. Partial Differential Equations}, 55(4):Art. 73, 25,
  2016.

\bibitem{z2016cvpde}
P.~Zhou.
\newblock On a {L}otka-{V}olterra competition system: diffusion vs advection.
\newblock {\em Calc. Var. Partial Differential Equations}, 55(6):Art. 137, 29,
  2016.

\bibitem{zhoup2018}
P.~Zhou and D.-M. Xiao.
\newblock Global dynamics of a classical {L}otka-{V}olterra
  competition-diffusion-advection system.
\newblock {\em J. Funct. Anal.}, 275(2):356--380, 2018.

\bibitem{zz2018}
P.~Zhou and X.~Q. Zhao.
\newblock Evolution of passive movement in advective environments: {G}eneral
  boundary condition.
\newblock {\em J. Differential Equations}, 264(6):4176--4198, 2018.

\end{thebibliography}

\end{document}